\documentclass[11pt,reqno]{amsart}
\usepackage{amsmath,amscd,amssymb,amsfonts,amsthm,courier,relsize,bm}
\usepackage{hyperref,enumerate,mathrsfs,mathtools,slashed}

\usepackage{tikz-cd}
\tikzcdset{arrow style=math font}
\tikzcdset{diagrams={nodes={inner sep=3pt}}}

\usepackage{xcolor}  
\hypersetup{
    colorlinks,
    linkcolor={red!50!black},
    citecolor={blue!70!black},
    urlcolor={blue!80!black}
}

\newtheorem{theorem}{Theorem}[section]
\newtheorem{corollary}[theorem]{Corollary}
\newtheorem{proposition}[theorem]{Proposition}

\newtheorem{lemma}[theorem]{Lemma}
\theoremstyle{definition}    
\newtheorem{definition}[theorem]{Definition}
\theoremstyle{remark}

\newtheorem{remark}[theorem]{Remark}

\newtheorem{example}[theorem]{Example}

\newcommand{\pair}[2]{\langle #1, #2 \rangle}
\newcommand{\ignore}[1]{}

\newcommand{\ol}[1]{\overline{#1}}
\newcommand{\ul}[1]{\underline{#1}}
\newcommand{\ti}[1]{\widetilde{#1}}

\newcommand{\wh}[1]{\widehat{#1}}
\newcommand{\scr}[1]{\mathscr{#1}}
\newcommand{\mf}[1]{\mathfrak{#1}}

\newcommand{\tn}[1]{\textnormal{#1}}
\renewcommand{\i}{{\mathrm{i}}}

\def\d{\ensuremath{\mathrm{d}}}

\def\Ad{\ensuremath{\textnormal{Ad}}}
\def\ad{\ensuremath{\textnormal{ad}}}
\def\g{\ensuremath{\mathfrak{g}}}
\def\k{\ensuremath{\mathfrak{k}}}
\def\t{\ensuremath{\mathfrak{t}}}
\def\n{\ensuremath{\mathfrak{n}}}

\def\h{\ensuremath{\mathfrak{h}}}

\def\r{\ensuremath{\mf{r}}}

\def\B{\ensuremath{\mathcal{B}}}

\def\F{\ensuremath{\mathcal{F}}}
\def\G{\ensuremath{\mathcal{G}}}
\def\H{\ensuremath{\mathcal{H}}}
\def\I{\ensuremath{\mathcal{I}}}
\def\J{\ensuremath{\mathcal{J}}}

\def\L{\ensuremath{\mathcal{L}}}

\def\O{\ensuremath{\mathcal{O}}}
\def\P{\ensuremath{\mathcal{P}}}

\def\T{\ensuremath{\mathcal{T}}}

\def\V{\ensuremath{\mathcal{V}}}

\def\bC{\ensuremath{\mathbb{C}}}
\def\bR{\ensuremath{\mathbb{R}}}
\def\bZ{\ensuremath{\mathbb{Z}}}

\def\bL{\ensuremath{\mathbb{L}}}

\def\End{\ensuremath{\textnormal{End}}}
\def\Hom{\ensuremath{\textnormal{Hom}}}

\def\ker{\ensuremath{\textnormal{ker}}}
\def\supp{\ensuremath{\textnormal{supp}}}
\def\id{\ensuremath{\textnormal{id}}}

\def\pr{\ensuremath{\textnormal{pr}}}
\def\dim{\ensuremath{\textnormal{dim}}}
\def\Sym{\ensuremath{\textnormal{Sym}}}

\def\Cl{\ensuremath{\textnormal{Cl}}}

\def\index{\ensuremath{\textnormal{index}}}

\def\f{\ensuremath{\mathscr{F}}}

\def\x{\ensuremath{\mf{X}}}

\title{Riemannian foliations and geometric quantization}
\author{Yi Lin}\address{Georgia Southern University}\email{yilin@georgiasouthern.edu}
\author{Yiannis Loizides}\address{George Mason University}\email{yloizide@gmu.edu}
\author{Reyer Sjamaar}\address{Cornell University}
\email{sjamaar@math.cornell.edu}
\author{Yanli Song}\address{Washington University in St. Louis}\email{yanlisong@wustl.edu}
\begin{document}
\sloppy
\maketitle

\vspace{-0.5cm}

\begin{abstract}
We introduce geometric quantization for constant rank presymplectic structures with Riemannian null foliation and compact leaf closure space. We prove a quantization-commutes-with-reduction theorem in this context. Examples related to symplectic toric quasi-folds, suspensions of isometric actions of discrete groups, and K-contact manifolds are discussed.
\end{abstract}

\tableofcontents

\section{Introduction}
Let $G$ be a compact connected Lie group with Lie algebra $\g$. Let $(M,\omega,\mu)$ be a compact Hamiltonian $G$-space equipped with a $G$-equivariant prequantum line bundle $(L,\nabla^L)$. A choice of $G$-invariant $\omega$-compatible almost complex structure $J$ on $M$ determines a $G$-equivariant Spin$_c$ Dirac operator
\[ D\colon \Gamma(M,\wedge^{\tn{even}} T^*M^{0,1}\otimes L)\rightarrow \Gamma(M,\wedge^{\tn{odd}}T^*M^{0,1}\otimes L).\]
Since $D$ is elliptic and $M$ is compact, the kernel and cokernel of $D$ are finite dimensional representations of $G$. Thus the equivariant index of $D$, denoted $RR_G(M,L)=\index_G(D)\in R(G)$, is well-defined (here the notation `$RR_G$' stands for `$G$-equivariant Riemann-Roch'). The equivariant index has a well-known interpretation as a crude sort of geometric quantization of $(M,\omega,\mu)$. For example, if $M$ is K{\"a}hler, $L$ is holomorphic and sufficiently positive, then $\index_G(D)$ agrees with the representation of $G$ on the space of holomorphic sections of $L$, the K{\"a}hler quantization of $(M,\omega,\mu)$.

The $G$-equivariant index has a remarkable property with respect to symplectic reduction. For simplicity assume $G$ acts freely on $\mu^{-1}(0)$ and let $M_0=\mu^{-1}(0)/G$ be the symplectic quotient, with its induced prequantum line bundle $L_0=L|_{\mu^{-1}(0)}/G$ and Spin$_c$ Dirac operator $D_0$. The \emph{quantization commutes with reduction} ($[Q,R]=0$) theorem \cite{GuilleminSternbergConjecture, MeinrenkenSymplecticSurgery} states that
\[ RR_G(M,L)^G=RR(M_0,L_0) \in \bZ,\]
i.e. the multiplicity of the trivial representation in $RR_G(M,L)$ equals  the Riemann-Roch number of the symplectic quotient at $0$. Various proofs of the $[Q,R]=0$ theorem have been given in the literature, cf. \cite{TianZhang, ParadanRiemannRoch}. The theorem has been generalized to the case of singular symplectic quotients \cite{MeinrenkenSjamaar} and to various non-compact settings (cf. \cite{MaZhangVergneConj, ParadanFormalI, HochsMathaiQuantizingProperActions} and references therein for a sample of results in this direction). There are refinements of this result in algebro-geometric settings (cf. \cite{TelemanConjecture}). Analogues of the $[Q,R]=0$ theorem have been proven in non-symplectic settings, such as quasi-Hamiltonian manifolds (cf. \cite{MeinrenkenKHomology}) and log symplectic manifolds (cf. \cite{guillemin2018geometric, BLSbsympl, LLSS2}).

In this article we study an analogous geometric quantization procedure in a setting where the closed 2-form $\omega$ is permitted to have a non-zero constant rank kernel. We will call such a 2-form a \emph{presymplectic form}. Since $\omega$ is closed and of constant rank, $\ker(\omega)=T\f$ is the tangent distribution to a regular foliation $\f$ of $M$, and $\omega$ restricts to a symplectic structure on the fibers of the normal bundle $N\f=TM/T\f$ to the leaves. The data $(M,\f,\omega)$ can be thought of as describing a symplectic structure on the leaf space of the foliation. More formally, $\omega$ determines a $0$-shifted symplectic structure on the leaf stack $M/\f$, cf. \cite{SjamaarHoffmanStacky} for this point of view. Analogues of many familiar theorems from symplectic geometry have been studied by various authors in this context, including \cite{SjamaarLinConvexity, toben2014localization, goertsches2017localization,SjamaarLinLocalization, goertsches2018equivariant, casselmann2017basic, liu2021chern}.

In the symplectic setting, the Riemannian metric $g=\omega \circ J$ determined by $(\omega,J)$ plays a key role in the definition of the Dirac operator $D$. Therefore to obtain a setting in which the above index-theoretic geometric quantization construction has a straight-forward analogue, we shall make the rather restrictive assumption that $(M,\f)$ is a \emph{Riemannian foliation}, i.e. that there is a fiber metric $g$ on the normal bundle $N\f$ such that the holonomy action is isometric. One thinks of $g$ as a Riemannian metric on the leaf space $M/\f$, and from this point of view the study of $(M,\f,g)$ fits into the more general context of Riemannian metrics on differentiable stacks \cite{del2019riemannian}. The simplest non-trivial example of a Riemannian foliation is a line of irrational slope on the 2-torus.

There is a rich and detailed theory of Riemannian foliations $(M,\f,g)$ due to Molino \cite{MolinoBook} and further developed by several authors (see also, for example, the references \cite{tondeurbook, Kacimi, MoerdijkMrcun}); we survey parts of this theory in Section \ref{s:background}. Given a Spin$_c$ structure for $(N\f,g)$, one can construct a \emph{transverse Spin$_c$ Dirac operator} $D$ that differentiates in directions orthogonal to the leaves of $\f$; see Section \ref{s:transverseindex}. If $M$ is compact then it has been known for some time that the restriction of $D$ to the space of smooth $\f$-invariant sections is a Fredholm operator \cite{Kacimi}. Transverse Dirac operators for Riemannian foliations have been studied extensively by many authors, including \cite{Kacimi, glazebrook1991transversal, douglas1995index, park1996basic, BruningRiemFoliations, GorokhovskyLott}.

Let $(M,\f,\omega)$ be a presymplectic manifold. For simplicity we shall assume $M$ is compact in this introduction, although this will be relaxed somewhat in the main body of the article. Let $\g$ be a compact Lie algebra acting on $(M,\f)$ by transverse vector fields (see Section \ref{s:tact}), and suppose that the action is Hamiltonian with moment map $\mu \in C^\infty(M,\g^*)^\f$. Assume $0 \in \g^*$ is a regular value of the moment map. The \emph{symplectic quotient} of $(M,\f,\omega)$ in the category of presymplectic manifolds is $(M_0,\f_0,\omega_0)$ where $M_0=\mu^{-1}(0)$, $\omega_0=\omega|_{M_0}$ and $\f_0$ is the enlarged foliation of $M_0$ generated by $\f$, $\g$. The theory of Hamiltonian actions and reduction in this context is discussed in greater detail in Appendix \ref{s:reductiontheory}.

Assume $N\f$ admits an $\omega$-compatible complex structure $J$ which is both $\f$ and $\g$ invariant, and let $g=\omega\circ J$ be the corresponding fiber metric turning $(M,\f,g)$ into a Riemannian foliation. Let $(L,\nabla^L)$ be a foliated basic $\g$-equivariant prequantum line bundle for $(M,\f,\omega)$. One constructs a transverse Spin$_c$ Dirac operator
\[ D \colon \Gamma(M,\wedge^{\tn{even}} N^*\f^{0,1}\otimes L)^\f \rightarrow \Gamma(M,\wedge^{\tn{odd}}N^*\f^{0,1}\otimes L)^\f \]
on the spaces of smooth $\f$-invariant sections. The operator $D$ is Fredholm and $\g$-equivariant, hence has a well-defined $\g$-equivariant index, denoted $RR_\g(M,\f,L)=\index^\f_\g(D)$, which belongs to the representation ring $R(\g)$ generated by finite dimensional irreducible representations of $\g$. Likewise one defines $RR(M_0,\f_0,L_0)=\index^{\f_0}(D_0)$ where $L_0=L|_{M_0}$ and $D_0$ is the corresponding transverse Dirac operator on $M_0$. The main result of the article is the following.
\begin{theorem}
\label{t:qr0introduction}
$RR_\g(M,\f,L)^\g=RR(M_0,\f_0,L_0)$.
\end{theorem}
This is a special case of Theorem \ref{t:qr0} in Section \ref{s:qr0}. One can interpret this result as a quantization-commutes-with-reduction theorem for the $0$-shifted symplectic structure on the leaf stack $M/\f$. In this spirit we prove (see Section \ref{s:weakequivquant}) that $RR_\g(M,\f,L)$ is invariant under a suitable notion of weak (or Morita) equivalence of foliations.

Our approach to Theorem \ref{t:qr0introduction} resembles that of Paradan \cite{ParadanRiemannRoch} in the symplectic setting, which in turn is based on Atiyah's theory of $G$-transversely elliptic operators and `non-abelian localization' to the critical locus of the norm-square of the moment map. Aspects of the theories of Molino \cite{MolinoBook} and el Kacimi \cite{Kacimi} that we use extensively, are surveyed in Sections \ref{s:MolinoOverview}, \ref{sec:descent}, \ref{s:transversediff}. In Section \ref{s:transverseindex} we develop the index theory of $\g\ltimes \f$-transversely elliptic symbols; here $\g\ltimes \f=\g_M\times_{N\f}TM$ is a Lie algebroid encoding the foliation $\f$ and transverse $\g$-action. In Section \ref{s:localization} this theory is applied to prove abelian and non-abelian localization formulas for the index of a $\g$-equivariant transverse Spin$_c$ Dirac operator. In Section \ref{s:qr0} the non-abelian localization formula is applied to prove Theorem \ref{t:qr0introduction}. Several applications of the main theorem are described in Section \ref{s:ex}, including to toric quasifolds, K-contact  (in particular, Sasakian) manifolds, and Haefliger suspensions.

We learned recently of very interesting forthcoming work of Hekmati and Orseli on the equivariant index of transverse Dolbeault operators in the toric K-contact setting. We expect that our projects are mutually complementary. One notable difference in setting is that Hekmati and Orseli work with the index of transverse operators acting on the full space of smooth sections, rather than restricting attention to the $\f$-basic sections as we have in the present article.

\section{Background on Riemannian foliations}\label{s:background}
In this section we recall various definitions and results in the theory of Riemannian foliations that we shall need in later sections. We give a brief overview of Molino's theory \cite{MolinoBook} and discuss transfer of basic vector bundles to the Molino manifold.

\subsection{Riemannian foliations}\label{s:riemfol}
Let $(M,\f)$ be a manifold equipped with a foliation of codimension $q$. The tangent bundle to the leaves is a vector bundle denoted $T\f$. Let $N\f$ denote the normal bundle to the leaves, a real vector bundle of rank $q$. If $T \subset M$ is a transversal to the leaves, then $N\f|_T$ is canonically isomorphic to the tangent bundle of $T$. A subset of $M$ is $\f$-\emph{saturated} if it is a union of leaves of $\f$. An $\f$-saturated submanifold $M'\subset M$ has a restricted foliation $\f|_{M'}$. If an orientation is chosen on $N\f$, then $(M,\f)$ is said to be \emph{transversely oriented}.

Let $\x(\f)\subset \x(M)$ be the Lie subalgebra of smooth vector fields tangent to the leaves of $\f$. There is an action of $\x(\f)$ on smooth sections of $N\f$:
\begin{equation} 
\label{e:bottconn}
\L_X \ol{Y}=\ol{[X,Y]}, \qquad X \in \x(\f), Y \in \x(M) 
\end{equation}
where $\ol{Z} \in \Gamma(N\f)$ denotes the image of $Z \in \x(M)$ under the quotient map. The normalizer of $\x(\f)$ in $\x(M)$ is the Lie subalgebra $\x(M,\f)$ of \emph{foliate vector fields}. The quotient $\x(M/\f)=\x(M,\f)/\x(\f)\simeq \Gamma(N\f)^\f$ is the Lie algebra of \emph{transverse vector fields}; here $\Gamma(N\f)^\f$ is the space of sections fixed under the $\x(\f)$ action. A differential form $\alpha$ on $M$ is said to $\f$-\emph{basic} if $\iota(X)\alpha=\L_X \alpha=0$ for all $X \in \mf{X}(\f)$.

In this paper we will sometimes use of the language of Lie algebroids and Lie groupoids. The tangent distribution $T\f$ to the foliation $\f$ is a simple example of a Lie algebroid, with the anchor map being the inclusion $T\f\hookrightarrow TM$, and the bracket being the ordinary bracket of elements of $\x(\f)$. Via equation \eqref{e:bottconn}, the vector bundle $N\f$ is an example of a representation of this Lie algebroid. The Lie algebroid $T\f$ has two canonical Lie groupoid integrations: the \emph{holonomy groupoid} $\tn{Hol}(M,\f)$ (the minimal integration), and the \emph{monodromy groupoid} $\tn{Mon}(M,\f)$ (the source-simply connected integration). The representation of $T\f$ on $N\f$ integrates to a representation of $\tn{Hol}(M,\f)$ (or $\tn{Mon}(M,\f)$).

A \emph{transverse metric} is a fibre metric $g$ on $N\f$ such that $\L_X g=0$ for all $X \in \mf{X}(\f)$. A triple $(M,\f,g)$ where $g$ is a transverse metric is a \emph{Riemannian foliation}. The action of $\tn{Hol}(M,\f)$ is by isometries between the fibres of $N\f$; in particular, the isotropy group at a point $m \in M$ is a (possibly non-closed) subgroup of the isometry group $O(N_m\f)$. The holonomy groupoid of a Riemannian foliation is automatically Hausdorff (cf. \cite[Examples 5.8]{MoerdijkMrcun}). 

Let $\scr{K}(M/\f,g)\subset \x(M/\f)$ denote the Lie subalgebra of \emph{transverse Killing fields}, that is, transverse vector fields that induce a Killing vector field on any local transversal. The normal bundle $N\f$ carries a canonical connection called the \emph{transverse Levi-Civita connection}; its restriction to any foliation chart $T\times F$ is the pullback of the Levi-Civita connection for the transversal $(T,g|_T)$.

A Riemannian metric $g_M$ on $M$ (not $\f$-invariant) is called a \emph{bundle-like metric for} $\f$ if the restriction of $g_M$ to $T\f^\perp\simeq N\f$ is a transverse metric. A Riemannian foliation $(M,\f,g)$ is \emph{complete} if there exists a complete bundle-like metric $g_M$ for $\f$ such that $g_M|_{T\f^\perp}=g$. For later use, we record the fact that the completeness property is well-behaved under restriction to submanifolds. More precisely, suppose $(M,\f,g)$ is a Riemannian foliation with complete bundle-like metric $g_M$, and let $M'\subset M$ be an $\f$-saturated submanifold of $M$, then:
\begin{enumerate}
\item if $M'$ is closed (as a set), then by the Hopf-Rinow theorem $g_M|_{M'}$ is a complete bundle-like metric for $(M',\f|_{M'})$, hence $(M',\f|_{M'})$ with the restricted transverse metric is a complete Riemannian foliation;
\item if $M'$ is open, then $g_M|_{M'}$ is usually not complete, but by Proposition \ref{p:openMolino} there exists a complete bundle-like metric for $(M',\f|_{M'})$, hence $(M',\f|_{M'})$ with the corresponding transverse metric is a complete Riemannian foliation.
\end{enumerate}
We end this section by describing two sources of simple examples of Riemannian foliations.

\begin{example}[locally free isometric actions]
\label{ex:freeisometric}
Let $(M,g_M)$ be a Riemannian manifold with a locally free action of a Lie algebra $\h$ by Killing vector fields. Let $\f$ be the induced foliation of $M$ by $\h$-orbits, and let $g=g_M|_{T\f^\perp}$ be the restriction of $g_M$ to $T\f^\perp\simeq N\f$. Then $(M,\f,g)$ is a Riemannian foliation. The simplest non-trivial example (with $\h=\bR$) is a line of irrational slope on the 2-torus.
\end{example}
\begin{remark}
By results of Haefliger and Salem \cite[Corollary 3.5]{haefliger1990riemannian}, every Killing foliation (cf. Section \ref{s:MolinoOverview}) on a compact connected oriented manifold is weakly equivalent (cf. Section \ref{s:weakequiv}) to a Riemannian foliation of the kind in Example \ref{ex:freeisometric}. 
\end{remark}
\begin{example}[Haefliger suspensions]
\label{ex:Haefliger}
Let $\ti{B},N$ be manifolds with $\ti{B}$ connected and $N$ a $q$-dimensional Riemannian manifold with metric $g_N$. Let $\ti{M}=\ti{B}\times N$ equipped with the Riemannian foliation $(\ti{M},\ti{\f},\ti{g})$ having leaves given by the fibers of the projection map $\ti{M}\rightarrow N$. Let $\Gamma$ be a discrete group acting on $\ti{B}$ and $N$ such that the following two conditions hold:
\begin{enumerate}
\item the action of $\Gamma$ on $\ti{B}$ is free and proper,
\item the action of $\Gamma$ on $N$ is isometric.
\end{enumerate}
Then $\ti{B}$ is a principal $\Gamma$-bundle over $B=\ti{B}/\Gamma$, and we may form the associated bundle with fiber $N$,
\[ M=\ti{M}/\Gamma=\ti{B}\times_\Gamma N \rightarrow B=\ti{B}/\Gamma.\]
The quotient of $\ti{\f}$ by $\Gamma$ is a Riemannian foliation $\f$ of $M$ with transverse metric $g$. Any fiber of $M \rightarrow B$ is a complete transversal for $\f$. Each leaf of $\f$ is of the form $\ti{B}\times_\Gamma (\Gamma \cdot  y)$ for some $y \in N$, and hence in particular $M/\f\simeq N/\Gamma$. We refer to $(M,\f)$ as the \emph{Haefliger suspension} of $(\ti{B},N,g_N)\circlearrowleft \Gamma$. 
\end{example}

\subsection{Overview of Molino theory}\label{s:MolinoOverview}
For details of Molino's theory, see Molino's book \cite{MolinoBook}. Other textbook references include \cite{tondeurbook, MoerdijkMrcun}. For a brief overview with more details than are provided here, see for example \cite{SjamaarLinLocalization, goertsches2018equivariant, alexandrino2022leaf}. 
In the transversely oriented case, it is convenient in some applications to use the \emph{oriented} transverse orthonormal frame bundle; the discussion below applies to this setting with only cosmetic changes. The material on the Molino centralizer sheaf and structural Lie algebra, included mainly for expository completeness, is used only sparingly in later sections. 

Let $(M,\f,g)$ be a complete Riemannian foliation on a connected manifold $M$. Let $\pi_M\colon P \rightarrow M$ be the orthonormal frame bundle of $(N\f,g)$, with structure group $O_q$ (Lie algebra $\mf{o}_q$). The foliation $\f$ lifts canonically to an $O_q$-invariant codimension $q+\frac{1}{2}q(q-1)$ foliation $\f_P$ of $P$. The transverse Levi-Civita connection determines a \emph{transverse parallelism} $(P,\f_P)$, i.e. a trivialization $N\f_P\simeq P\times (\bR^q\times \mf{o}_q)$ as a foliated vector bundle. Using the horizontal lift of a bundle-like metric $g_M$ and any fixed inner product for $\mf{o}_q$, we obtain a complete bundle-like metric $g_P$ for $(P,\f_P)$. The transverse vector fields in the transverse parallelism are bounded with respect to this metric, and therefore $N\f_P$ is generated by the images of global, complete, foliate vector fields. There follow a number of desirable consequences (the `Molino package'), including (i) existence of a Molino manifold $W=P/\ol{\f}_P$ and generalized Lie groupoid morphism $\tn{Hol}(M,\f) \dashrightarrow O_q\ltimes W$, (ii) a locally constant Molino centralizer sheaf $\scr{C}_M$, (iii) leaf closure (singular) foliation $\ol{\f}$ with Hausdorff quotient $M/\ol{\f}$, (iv) transfer functor for basic vector bundles. We briefly discuss (i), (ii), (iii) and related results below which are due to Molino cf. \cite[Section 4.5]{MolinoBook}. Item (iv) is due to El Kacimi \cite{Kacimi} and will be discussed in Sections \ref{s:basicvb}, \ref{sec:descent}.


From now on we assume $(M,\f,g)$ is a complete Riemannian foliation. The closures of the leaves of $\f_P$ then form a \emph{strictly simple} foliation $\ol{\f}_P$, i.e. a fibration with connected fibers over a smooth manifold $W=P/\ol{\f}_P$ dubbed the \emph{Molino manifold}. The $O_q$ action on $P$ induces an $O_q$ action on $W$ such that the quotient map $\pi_W \colon P \rightarrow W$ is $O_q$-equivariant. 
The transverse metric on $(P,\f_P)$ determines an induced complete Riemannian metric on $W$ such that $N\ol{\f}_P \xrightarrow{\sim} \pi_W^*TW$ is metric-preserving. 

Let $\mf{s}_P=\ker(N\f_P\rightarrow N\ol{\f}_P)$. The Lie bracket on transverse vector fields ($\f_P$-invariant sections of $N\f_P$) restricts to a Lie bracket on $\f_P$-invariant sections of $\mf{s}_P$, which in turn induces a Lie bracket on each of the fibers of $\mf{s}_P$, making $\mf{s}_P$ into a bundle of Lie algebras. The flows of global complete foliate vector fields may be used to produce isomorphisms between the fibers of $\mf{s}_P$ at different points, and so all fibers are isomorphic to a Lie algebra $\mf{s}$ called the \emph{structural Lie algebra} of the Riemannian foliation.

Let $\scr{T}_P$ be the sheaf of transverse vector fields of $(P,\f_P)$. The subsheaf $\scr{C}_P \subset \scr{T}_P$ of transverse vector fields commuting with all global transverse vector fields $\Gamma(P,\scr{T}_P)=\mf{X}(P/\f_P)$ is called the \emph{Molino centralizer sheaf}. It follows, from the fact that $\mf{X}(P,\f_P)$ acts locally transitively on $P$, that $\scr{C}_P$ is locally constant and finitely generated, hence is given by flat sections of a flat bundle of Lie algebras $\mf{c}_P$.

Let $\L \ni p$ be a leaf of $\f_P$, and let $\ol{\L}=\pi_W^{-1}(\pi_W(p))$ be its closure. The fiber $\mf{s}\simeq\mf{s}_{P,p}$ can be identified with the space $\mf{X}(\ol{\L}/\f_P|_{\ol{\L}})$ of transverse vector fields, and thus acts locally freely and transitively on a suitable small transversal $T \ni p$ to $\L$ inside $\ol{\L}$. This identifies $T$ with a neighborhood of $1 \in S$, where $S$ is a Lie group integrating $\mf{s}$, such that $\mf{s}$ acts as left invariant vector fields. The commuting sheaf $\scr{C}_P|_T$ is then identified with the right invariant vector fields. In particular $\mf{c}_{P,p}, \mf{s}_{P,p}$ are isomorphic as Lie algebras, and the fiber product $\mf{c}_P\times_{N\f_P}TP$ is a Lie algebroid whose orbits are the leaves of $\ol{\f}_P$.

As elements of $\scr{C}_P$ commute with the transverse parallelism of $(P,\f_P)$, they are automatically lifts of transverse Killing fields on $M$. It follows that the local Killing fields on $M$ that lift to elements of $\scr{C}_P$ form a locally constant sheaf $\scr{C}_M$ on $M$, given by flat sections of a flat bundle of Lie algebras $\mf{c}_M=\mf{c}_P/O_q$. In the special case that $\scr{C}_M$ is globally constant, $(M,\f,g)$ is called a \emph{Killing foliation}. The quotient $\pi_M\colon P\rightarrow M$ maps leaf closures onto leaf closures, and it follows that leaf closures in $M$ are the $\mf{c}_M$-orbits of the leaves of $\f$. In particular the leaf closures in $M$ are the orbits of a Lie algebroid $\mf{c}\ltimes \f:=\mf{c}_M\times_{N\f} TM$, and so form a singular foliation $\ol{\f}$. The quotient $M/\ol{\f}$ is a Hausdorff space homeomorphic to $W/O_q$. $M$ is naturally stratified according to the dimension of the leaf closures, and on passing to the quotient one obtains a stratification of $M/\ol{\f}$ by orbifolds. Within a stratum on $M$, the holonomy of the $\ol{\f}$ action may vary, but is finite.

The relation between $M$ and $W$ is summarized in the \emph{Molino diagram}:
\[
\begin{tikzcd}
&P \ar[dl,"\pi_M"']\ar[dr,"\pi_W"] &\\
M \ar[d] & & W \ar[d]\\
M/\ol{\f} \ar[rr,leftrightarrow,"\sim"]& & W/O_q
\end{tikzcd}
\]
In the language of Lie groupoids, Molino's construction yields a generalized morphism from the holonomy groupoid $\tn{Hol}(M,\f)$ of the Riemannian foliation to the action Lie groupoid $O_q\ltimes W$. The fiber product $P\times_W P$ associated to the submersion $\pi_W \colon P \rightarrow W$ is a Lie groupoid over $P$. Its quotient
\begin{equation} 
\label{e:Molinogroupoid}
\G=(P\times_W P)/O_q
\end{equation}
is a Lie groupoid over $M$ that we shall refer to as the \emph{Molino groupoid}. Since $\pi_W\colon P \rightarrow W$ is the quotient map by $\ol{\f}_P$, $\G$ integrates $\ol{\f}_P/O_q=(\mf{c}\ltimes \f_P)/O_q=\mf{c}\ltimes \f$. The construction is such that $\G$ and $O_q\ltimes W$ are Morita equivalent Lie groupoids, with $P$ serving as the Morita equivalence bibundle. 
\begin{example}[locally free isometric actions]
\label{ex:freeisometric2}
Let $(M,\f,g)$ be the Riemannian foliation from Example \ref{ex:freeisometric}. Let $H$ be the simply connected Lie group integrating $\h$. The leaf space $M/\f=M/H$. Let $\ol{H}$ be the closure of the image of $H$ in the isometry group of $M$. The $\ol{H}$ action on $M$ lifts to an action on the transverse orthonormal frame bundle $P$, and $W=P/\ol{H}$. The quotient $M/\ol{\f}\simeq W/O_q\simeq M/\ol{H}$. 
\end{example}
\begin{example}[Haefliger suspensions]
\label{ex:Haefliger2}
Let $(M,\f,g)$ be the Riemannian foliation from Example \ref{ex:Haefliger}. If $N$ is complete then $(M,\f,g)$ is a complete Riemannian foliation. The leaf closures are the associated bundles of the $\ol{\Gamma}$ orbits in $N$, where $\ol{\Gamma}$ denotes the closure of the image of $\Gamma$ in the isometry group of $N$. The Molino manifold $W$ identifies with the quotient $P_N/\ol{\Gamma}$, where $P_N$ is the orthonormal frame bundle of $N$. The quotient $M/\ol{\f}\simeq W/O_q\simeq N/\ol{\Gamma}$. 
\end{example}

\subsection{Basic vector bundles}\label{s:basicvb}
\begin{definition}
A \emph{foliated principal bundle} $(Q,\f_Q)$ on $(M,\f)$ is a principal $H$-bundle $Q$ ($H$ a Lie group) equipped with an $H$-invariant lifted foliation $\f_Q$. A \emph{foliated vector bundle} is a vector bundle associated to a foliated principal bundle.
\end{definition} 
A foliated vector bundle $E$ acquires a foliation $\f_E$ of its total space (covering the foliation $\f$ on the base $M$) such that holonomy acts by linear isomorphisms between the fibers. We shall say that `$\f$ acts on $E$' for short; other ways of saying the same thing are (i) $E$ is a representation of the Lie algebroid $T\f$, (ii) $E$ is a representation of the source simply connected integration $\tn{Mon}(M,\f)$ of $T\f$. The following is a considerably more restrictive definition.
\begin{definition}
A \emph{basic principal bundle} is a foliated principal bundle $(Q,\f_Q)$ admitting a connection $1$-form $\theta$ that is $\f_Q$-basic. A \emph{basic vector bundle} is a vector bundle associated to a basic principal bundle.
\end{definition}
By the Cartan homotopy formula, $\theta$ is $\f_Q$-basic if and only if both $\theta$ and the curvature $F_\theta=\d \theta+\frac{1}{2}[\theta,\theta]$ have trivial contraction with $\mf{X}(\f_Q)$. A choice of basic connection is not regarded as part of the data of a basic principal bundle. Any two choices of basic connection are homotopic through basic connections.

\begin{example}
Let $(M,\f,g)$ be a Riemannian foliation. Then $N\f$ is a basic vector bundle. There is a canonical connection on $N\f$, the \emph{transverse Levi-Civita connection}, given locally on foliation charts $U\simeq T\times F$ by the pullback of the Levi-Civita connection on $(T,g_T)$, where $g_T$ is the induced Riemannian metric on the transversal $T$. That the local definitions patch together to a global connection follows from the fundamental theorem of Riemannian geometry.
\end{example}

Let $\V\B(M,\f\tn{-bas})$ be the category of basic vector bundles on $M$, with morphisms being $\f$-equivariant smooth vector bundle maps $E_1 \rightarrow E_2$. The morphisms are the same as in the category of foliated vector bundles. There is a similar category $\H\V\B(M,\f\tn{-bas})$ of Hermitian (Euclidean in the real case) basic vector bundles, for which the unitary frame bundle $\tn{Fr}_U(E)$ is basic. The following is a key property of basic principal bundles that is not difficult to prove.
\begin{proposition}
\label{prop:foliatelift}
Let $(P,\f_P)$ be a basic principal bundle over $(M,\f)$ equipped with a basic connection $\theta$. Horizontal lifts of foliate vector fields are foliate.
\end{proposition}
\ignore{
\begin{remark}[Foliate vector fields]\label{r:foliatevectorfield}
To check whether a vector field $X \in \x(M)$ is foliate, it suffices to verify $[X,Y] \in \x(\f)$ for vector fields $Y$ with support contained in the sets of an open cover of $M$ (use a partition of unity). On each open $U \subset M$, it suffices to check $[X,Y_i] \in \x(\f)$ for a local frame $Y_1,...,Y_{\dim(T\f)}$ of $T\f$. Indeed note that if $Y=\sum_i f_iY_i$ then
\[ [X,Y]=\sum_i f_i[X,Y_i]+(Xf_i)Y_i, \]
and $[X,Y_i],(Xf_i)Y_i \in \x(\f)$ for all $i$.
\end{remark}
\begin{proof}
Let $X \in \x(M,\f)$ be a foliate vector field and let $\ti{X}$ be its horizontal lift.  By Remark \ref{r:foliatevectorfield}, the result will follow once we show $[\ti{X},\ti{Y}]\in \x(\f_P)$ for all $Y \in \x(\f)$, since horizontal lifts are sufficient to construct local frames of $T\f_P$. The vector fields $\ti{X}$, $\ti{Y}$ are $\pi_M$-related to $X,Y$ respectively, hence $[\ti{X},\ti{Y}]$ is $\pi_M$-related to $[X,Y] \in \x(\f)$. Therefore $[\ti{X},\ti{Y}]$ is projectable and projects to a vector field tangent to the leaves.  On the other hand
\[ \theta([\ti{X},\ti{Y}])=\ti{X}\theta(\ti{Y})-\ti{Y}\theta(\ti{X})-d\theta(\ti{X},\ti{Y})=0-0-F_\theta(\ti{X},\ti{Y})+\tfrac{1}{2}[\theta(\ti{X}),\theta(\ti{Y})]=0 \]
using $\theta(\ti{X})=\theta(\ti{Y})=0$ (as these vectors are horizontal) and $\iota(\ti{Y})F_\theta=0$ as $\theta$ is basic. Thus $[\ti{X},\ti{Y}]$ is horizontal and projects onto a vector field that is tangent to $\f$. It follows that $[\ti{X},\ti{Y}]$ is tangent to $\f_P$.
\end{proof}
}

\subsection{Transfer of basic vector bundles}\label{sec:descent}
Let $(M,\f,g)$ be a complete Riemannian foliation. Let $E$ be a basic vector bundle over $M$. The space of $\f_P$-invariant sections of $\pi_M^*E \rightarrow P$ is a $C^\infty(P)^{\f_P}=C^\infty(W)$-module.

\begin{theorem}[\cite{Kacimi}, Section 2]\label{t:transfer}
Let $E$ be a basic vector bundle over a complete Riemannian foliation $(M,\f,g)$. The $C^\infty(W)$-module $C^\infty(P,\pi_M^*E)^{\f_P}$ is finitely generated and projective, hence is the space of smooth sections of a smooth $O_q$-equivariant vector bundle $\T(E)\rightarrow W$. There is a canonical isomorphism $C^\infty(M,E)^{\f}\simeq C^\infty(W,\T(E))^{O_q}$. The fiber $\T(E)_w$ of $\T(E)$ at a point $w\in W$ may be identified as
\[ \T(E)_w:=C^\infty(P,\pi_M^*E)^{\f_P}/\mf{m}_w C^\infty(P,\pi_M^*E)^{\f_P}\simeq C^\infty(\pi_W^{-1}(w),\pi_M^*E)^{\f_P},\]
where $\mf{m}_w$ is the vanishing ideal of $w$. The pullback $\pi_W^*\T(E)$ is a subbundle of $\pi_M^*E$.
\end{theorem}
The proof has elements in common with Molino's proof that $C^\infty(P)^{\f_P}$ is the ring of smooth functions on a manifold $W$, with Proposition \ref{prop:foliatelift} playing an important role. For example to show that the dimension of the fibers $\T(E)_w$ is locally constant, recall that complete foliate vector fields act locally transitively on $P$, and lift these to the total space of $\pi_M^*E$ using Proposition \ref{prop:foliatelift}. Local flows of such vector fields act by linear isomorphisms between the fibers, implying that all fibers have the same rank over any connected component of $P$. 

\begin{remark}
\label{rem:ranks}
For an example in which the rank of $\T(E)$ is strictly less than the rank of $E$, see \cite[Example 2.7.3]{Kacimi}. A simple sufficient condition for the ranks to be equal (in which case $\pi_W^*\T(E)\simeq \pi_M^*E$) is that $\pi_M^*E$ be the quotient of a trivial foliated vector bundle by a foliated vector subbundle.
\end{remark}

Morphisms in $\V\B(M,\f\tn{-bas})$ induce $O_q$-equivariant morphisms of vector bundles over $W$, yielding a functor
\[ \T \colon \V\B(M,\f\tn{-bas})\rightarrow \V\B(W,O_q), \]
where $\V\B(W,O_q)$ denotes the category of $O_q$-equivariant vector bundles over $W$. There is likewise a functor
\[ \T\colon \H\V\B(M,\f\tn{-bas})\rightarrow \H\V\B(W,O_q).\]
The functor $\T$ is monoidal and satisfies a compatibility condition with internal Hom's as well: one has $\T(1)=1$ and there are natural transformations
\[ \T (E_1)\otimes \T (E_2)\rightarrow \T(E_1\otimes E_2), \qquad \T(\Hom(E_1,E_2))\rightarrow \Hom(\T (E_1),\T (E_2)).\]
The first map exists because the tensor product of two $\f_P$-invariant sections is $\f_P$-invariant. The second map exists because if an $\f_P$-equivariant morphism is applied to an $\f_P$-invariant section, the result is an $\f_P$-invariant section.

\begin{definition}
\label{d:Tsub}
Let $M' \subset M$ be a closed $\f$-invariant submanifold, with restricted foliation $\f'$. Let $W'=\pi_W(\pi_M^{-1}(M'))$ be the corresponding $O_q$-invariant submanifold of $W$. Note that $W'$ is not the Molino manifold of $M'$. Nevertheless, there is a functor
\[ \T'\colon \V\B(M',\f'\tn{-bas})\rightarrow \V\B(W',O_q) \]
with analogous properties, defined as follows. Let $E'$ be a basic vector bundle over $M'$. Using the tubular neighborhood theorem of \cite{SjamaarLinLocalization}, there is an $\f$-invariant tubular neighborhood $\ti{M}$ of $M'$. The pullback $\ti{E}$ of $E'$ to $\ti{M}$ is a basic vector bundle. Applying the functor $\ti{\T}$ of the complete Riemannian foliation $(\ti{M},\ti{\f}=\f|_{\ti{M}})$ to $\ti{E}$ yields an $O_q$-equivariant vector bundle $\ti{\T}(\ti{E})$ over the Molino manifold $\ti{W}$. The latter is identified with an open neighborhood of $W'$ in $W$, and we define $\T'(E')=\ti{\T}(\ti{E})|_{W'}$.
\end{definition}

\subsection{Examples of transfer of vector bundles}\label{s:transferexamples}
Examples of basic vector bundles include: $N\f$, $N\f_P/O_q$, $\Ad(P)$, $N\ol{\f}_P/O_q$. For the case of $N\f$, the transverse Levi-Civita connection provides a basic connection. Applying the functor $\T$ yields the vector bundles $\T(N\f)$, $\T(N\f_P/O_q)$, $W \times \mf{o}_q$, $TW$ respectively. Note that $N\f_P$, $\pi_M^*N\f$, $\pi_M^*\Ad(P)=P\times \mf{o}_q$ are trivial, while $N\ol{\f}_P$ is a quotient of $N\f_P$, hence in particular Remark \ref{rem:ranks} shows that for these 4 examples the rank does not drop.  The foliated vector bundle $\pi_M^*N\f$ has a canonical $O_q$-equivariant trivialization, hence $\T(N\f)\simeq W\times \bR^q$, with the trivial $O_q$-action on the fibres.

Note that there are obvious vector bundle maps
\[ \Ad(P)\rightarrow N\f_P/O_q\rightarrow N\f, \qquad b\colon N\f_P/O_q\rightarrow N\ol{\f}_P/O_q.\]
The first pair constitute the transverse Atiyah sequence for $P$. The transverse Levi-Civita connection determines splitting maps
\begin{equation} 
\label{e:transverseLC}
\theta \colon N\f_P/O_q\rightarrow \Ad(P), \qquad l\colon N\f \rightarrow N\f_P/O_q. 
\end{equation}
We also define
\begin{equation}
\label{e:anchorsplitting}
h=b\circ l \colon N\f\rightarrow N\ol{\f}_P/O_q.
\end{equation}
All of these vector bundle maps are $\f$-equivariant, hence induce morphisms $\T(b),\T(l),\T(\theta),\T(h)$ of $O_q$-equivariant vector bundles over $W$.
\ignore{
\begin{remark}
The morphism $\T(b)\colon \T(N\f_P/O_q)\rightarrow TW$ is the anchor map of a Lie algebroid (cf. \cite[p.13]{SjamaarLinLocalization}). The morphism $\T(h) \colon \T(N\f)\rightarrow TW$ will appear frequently below.
\end{remark}
}

Let $M'\subset M$ be a closed $\f$-saturated submanifold with restricted foliation $\f'$, and let $W'=\pi_W(\pi_M^{-1}(M'))$. Using the transverse metric, there is an orthogonal splitting
\[ N\f|_{M'}=N\f'\oplus NM' \]
where $NM'=TM|_{M'}/TM'$ is the normal bundle to $M'$ in $M$. Applying the functor $\T'$ from Definition \ref{d:Tsub} yields $O_q$-equivariant vector bundles $\T'(N\f|_{M'})\simeq \T'(N\f')\oplus \T'(NM')$ over $W$. Remark \ref{rem:ranks} implies that the ranks of these vector bundles does not decrease. Moreover $\T'(NM')=NW'$. 

\subsection{Transverse actions}\label{s:tact}
Let $(M,\f)$ be a foliation and let $\g$ be a finite dimensional Lie algebra. A \emph{transverse action} of $\g$ on $(M,\f)$ is a Lie algebra homomorphism
\[ a \colon \g \rightarrow \mf{X}(M/\f)=\Gamma(N\f)^\f.\]
Let $(M,\f,g)$ be a complete Riemannian foliation. An action $a\colon \g \rightarrow \mf{X}(M/\f)$ is \emph{isometric} if $a$ takes values in transverse Killing fields $\scr{K}(M/\f,g)$. 
\begin{example}
Let $(M,\f,g)$ be a complete Riemannian foliation with trivial Molino centralizer sheaf $\scr{C}_M$. Then $\mf{c}_M \simeq M\times \mf{c}$ is a trivial bundle of Lie algebras and $(M,\f,g)$ acquires a transverse isometric action of $\mf{c}$. We remark that it is known (cf. \cite{MolinoBook}) that if $\scr{C}_M$ is trivial then $\mf{c}$ must be abelian. 
\end{example}
Note that $\g$ does not act on $M$ in general; instead the Lie algebroid $T\f$ and the transverse action may be amalgamated into a Lie algebroid that acts on $M$.
\begin{definition}[\cite{SjamaarLinLocalization}]
\label{d:tact}
Let $a \colon \g \rightarrow \mf{X}(M/\f)$ be a transverse action. The fiber product
\[ \g\ltimes \f=\g_M\times_{N\f} TM, \qquad \g_M=M\times \g \] 
is a Lie algebroid. The anchor map is projection to the second factor, and the bracket is induced by the Lie brackets on $\g$ and $\mf{X}(M)$.
\end{definition}
This Lie algebroid is integrable \cite{SjamaarLinLocalization}. There is a similar Lie algebroid more generally in case $\g_M$ is replaced with a flat bundle of Lie algebras (such as $\mf{c}_M$) equipped with a bundle map to $N\f$ such that the induced map on local sections restricts to a Lie algebra homomorphism from flat sections to transverse vector fields. 

A transverse isometric action lifts to an action $a^\dagger \colon \g \rightarrow \scr{K}(P/\f_P,g_P)^{O_q}$, which determines a bundle map $a^\dagger\colon \g_M\rightarrow N\f_P/O_q$. The transverse isometric $\g$ action on $P$ projects to an isometric $\g$ action on $W$. Since $W$ is complete, the isometric $\g$ action integrates to an isometric action of the connected simply connected integration $G$ of $\g$. Thus there is homomorphism of Lie groups 
\[ \rho \colon G \rightarrow \tn{Isom}(W).\] 
This homomorphism need not be an embedding. Let $\ol{\rho(G)}$ denote the the closure of the image $\rho(G)$ in $\tn{Isom}(W)$; $\ol{\rho(G)}$ is thus a connected embedded Lie subgroup of $\tn{Isom}(W)$.


There is a category $\V\B_\g(M,\f\tn{-bas})$ of $\g\ltimes \f$-equivariant vector bundles such that $\tn{Fr}(E)$ admits a $\g\ltimes \f$-invariant $\f$-basic connection. The functor $\T$ generalizes to a functor
\[ \T\colon \V\B_\g(M,\f\tn{-bas})\rightarrow \V\B_\g(W,O_q).\]
Let $\rho_E$ denote the induced action of $G$ on $\T(E)$, and let $\ol{\rho_E(G)}$ denote the closure of the image of $G$ in the Lie group consisting of automorphisms of $\T(E)$ covering an isometry of $W$.

Likewise there is a category $\H\V\B_\g(M,\f\tn{-bas})$ of $\g\ltimes \f$-equivariant Hermitian (Euclidean in the real case) vector bundles such that the unitary frame bundle $\tn{Fr}_U(E)$ admits a $\g\ltimes \f$-invariant $\f$-basic connection. In this case the induced map $\ol{\rho_E(G)}\rightarrow \ol{\rho(G)}$ is guaranteed to be surjective, since the map from unitary automorphisms of $\T(E)$ to isometries of $W$ is proper.
\begin{example}
Given a transverse isometric $\g$ action on $M$, $N\f$ acquires a $\g$-action, and the induced $G$-action on the trivial vector bundle $\T(N\f)$ is the trivial lift of the $G$-action on $W$.
\end{example}

\subsection{Weak equivalence of Riemannian foliations}\label{s:weakequiv}
If $\pi \colon B \rightarrow M$ is a smooth map with connected fibers and transverse to a foliation $\f$ of $M$, then $\pi^*\f$ denotes the pullback foliation, with leaves $\pi^{-1}(\L)$ for $\L$ a leaf of $\f$. A \emph{weak equivalence} of foliations $(M_1,\f_1)$, $(M_2,\f_2)$ is a triple $(B,\pi_1,\pi_2)$ where for $i=1,2$, $\pi_i \colon B\rightarrow M_i$ is a surjective submersion with connected fibers, such that $\pi_1^*\f_1=\pi_2^*\f_2$. Weakly equivalent foliations are also called transversely equivalent (cf. \cite[Definition 2.1]{MolinoBook}).

A weak equivalence induces an isomorphism
\begin{equation} 
\label{e:normaliso}
\pi_1^*N\f_1\simeq \pi_2^*N\f_2.
\end{equation}
Hence $(M_1,\f_1)$ admits a transverse metric if and only if $(M_2,\f_2)$ admits a transverse metric. An \emph{isometric weak equivalence of Riemannian foliations} is a weak equivalence such that \eqref{e:normaliso} intertwines the transverse metrics. If $(M_1,\f_1,g_1)$ and $(M_2,\f_2,g_2)$ are complete Riemannian foliations, then we say that an isometric weak equivalence $(B,\pi_1,\pi_2)$ is \emph{complete} if $(B,\f,g)$ is a complete Riemannian foliation, where $\f=\pi_1^*\f_1=\pi_2^*\f_2$ and $g=\pi_1^*g_1=\pi_2^*g_2$.

\begin{proposition}
\label{p:moritamolino}
A complete isometric weak equivalence of complete Riemannian foliations induces an $O_q$-equivariant isometry of Molino manifolds.
\end{proposition}
\begin{proof}
By symmetry it suffices to prove the result in the special case where one of the maps $\pi_i$ is the identity map. 
Thus we are reduced to proving that if $\pi\colon (M_1,\f_1,g_1)\rightarrow (M_2,\f_2,g_2)$ is surjective submersion with connected fibers such that $\pi^*\f_2=\f_1$, $\pi^*g_2=g_1$, then $\pi$ induces an equivariant isometry of Molino manifolds $W_1\simeq W_2$. In this case the differential of $\pi$ induces an isometry of foliated vector bundles $N\f_1\simeq \pi^*N\f_2$, and hence an isomorphism of foliated principal bundles $(P_1,\f_{P_1})\simeq (\pi^*P_2,\pi^*\f_{P_2})$. By the definition of pullback foliation, a $\pi^*\f_{P_2}$-invariant smooth function on $\pi^*P_2$ descends to an $\f_{P_2}$-invariant smooth function on $P_2$, and conversely. Thus there are canonical isomorphisms
\[ C^\infty(W_1)\simeq C^\infty(P_1)^{\f_{P_1}}\simeq C^\infty(\pi^*P_2)^{\pi^*\f_{P_2}}\simeq C^\infty(P_2)^{\f_{P_2}}\simeq C^\infty(W_2),\]
inducing an equivariant diffeomorphism of Molino manifolds. Since the transverse metric of $(M_1,\f_1)$ is the pullback of that of $(M_2,\f_2)$, the map of Molino manifolds is an isometry.
\end{proof}
\begin{corollary}
A complete isometric weak equivalence induces a homeomorphism of spaces of leaf closures.
\end{corollary}
\begin{proof}
This follows immediately from Proposition \ref{p:moritamolino} since $M/\ol{\f}\simeq W/O_q$.
\end{proof}

\begin{proposition}
\label{p:equivMorita}
A weak equivalence (resp. isometric weak equivalence) of foliations (resp. Riemannian foliations) determines a one-one correspondence between transverse actions (resp. transverse isometric actions).
\end{proposition}
\begin{proof}
As in the proof of Proposition \ref{p:moritamolino}, it suffices to consider the case of a foliated submersion $\pi\colon (M_1,\f_1) \rightarrow (M_2,\f_2)$ with connected fibers. In this case $N\f_1\simeq \pi^*N\f_2$. A transverse action $\g \rightarrow \mf{X}(M_2/\f_2)=\Gamma(N\f_2)^{\f_2}$ pulls back uniquely to a transverse action $\g \rightarrow \Gamma(\pi^*N\f_2)^{\pi^*\f_2}=\Gamma(N\f_1)^{\f_1}=\mf{X}(M_1/\f_1)$. Conversely since the leaves of $\f_1$ contain the fibers of $\pi$, a transverse action $\g \rightarrow \mf{X}(M_1/\f_1)=\Gamma(N\f_1)^{\f_1}$ is projectable, producing a transverse action $\g \rightarrow \Gamma(N\f_2)^{\f_2}=\mf{X}(M_2/\f_2)$. In either direction the correspondence is compatible with brackets. Indeed it suffices to check this property locally. On a local foliation chart transverse vector fields are in one-one correspondence with vector fields on the transversal. But $\pi$ induces a diffeomorphism between sufficiently small transversals in $M_1,M_2$, and diffeomorphisms are compatible with Lie brackets.
\end{proof}
A weak equivalence between foliations that intertwines transverse actions as in the proof of Proposition \ref{p:equivMorita} will be called \emph{equivariant}.

Let $(B,\pi_1,\pi_2)$ be a weak equivalence of foliations $(M_i,\f_i)$, $i=1,2$. Let $E_1,E_2$ be foliated vector bundles (resp. Hermitian vector bundles) over $M_1,M_2$ respectively. We will say that $(B,\pi_1,\pi_2)$ \emph{intertwines} $E_1,E_2$ if $\pi_1^*E_1\simeq \pi_2^*E_2$ are isomorphic as $\pi_1^*\f_1=\pi_2^*\f_2$-foliated vector bundles (resp. foliated Hermitian vector bundles). A choice of isomorphism $\pi_1^*E_1\simeq \pi_2^*E_2$ determines an isomorphism between the corresponding spaces of invariant sections. If one of $E_1,E_2$ is a basic vector bundle, then the other is also basic. Supposing in addition that $(B,\pi_1,\pi_2)$ is a complete isometric weak equivalence of complete Riemannian foliations, a choice of isomorphism $\pi_1^*E_1\simeq \pi_2^*E_2$ determines an equivariant isomorphism $\T_1(E_1)\simeq \T_2(E_2)$ covering the isometry $W_1\simeq W_2$. In the $\g$-equivariant case, we require that the isomorphism $\pi_1^*E_1\simeq \pi_2^*E_2$ intertwine transverse $\g$-actions.

\begin{remark}
Riemannian foliations $(M_1,\f_1,g_1)$, $(M_2,\f_2,g_2)$ are weakly equivalent if and only if the holonomy groupoids $\tn{Hol}(M_1,\f_1)$, $\tn{Hol}(M_2,\f_2)$ are Morita equivalent in the sense of Lie groupoids. A Morita equivalence bibundle $M_1\leftarrow B \rightarrow M_2$ is a weak equivalence in the above sense, and for the other direction see for example \cite[Theorem 4.6.3]{del2013lie}. Note also that all manifolds that arise in this discussion are automatically Hausdorff, since the holonomy groupoid of any Riemannian foliation is automatically Hausdorff (cf. \cite[Examples 5.8]{MoerdijkMrcun}).
\end{remark}

\begin{remark}
As a warning, a weak equivalence $(B,\pi_1,\pi_2)$ does not determine an equivalence of categories of foliated vector bundles on $(M_1,\f_1)$, $(M_2,\f_2)$. The reason is that foliated vector bundles are not required to be representations of the holonomy groupoid. If we fix source connected groupoids $\G_1,\G_2$ integrating $T\f_1,T\f_2$ respectively, and if the weak equivalence $B$ is a Morita equivalence bibundle for $\G_1,\G_2$, then $B$ induces an equivalence between the categories of $\G_1$, $\G_2$ equivariant vector bundles over $M_1$, $M_2$ respectively.
\end{remark}

\begin{remark}
Weak equivalences $(B,\pi_1,\pi_2)$, $(B',\pi_1',\pi_2')$ are themselves said to be equivalent if there is a weak equivalence $(B'',\pi_1'',\pi_2'')$ between $B, B'$ equipped with the pullback foliations, and such that the obvious diagram involving $\pi_i,\pi_i',\pi_i''$ commutes. There are analogous definitions for weak equivalences that are isometric, complete, equivariant, and so on. Equivalent complete isometric weak equivalences of complete Riemannian foliations induce the same equivariant isometry of Molino manifolds.
\end{remark}

\section{Transverse index theory}\label{s:transverseindex}
In this section we discuss the index theory of transversely elliptic operators on complete Riemannian foliations. We combine the theory of operators elliptic in directions transverse to a foliation (cf. \cite{Kacimi}) with elements of the theory of operators in directions transverse to the orbits of a compact group action (cf. \cite{AtiyahTransEll}) to study $\g\ltimes \f$-transversely elliptic operators, for $\g$ acting transversely and isometrically. Throughout this section $(M,\f,g)$ shall be a complete Riemannian foliation of codimension $q$ with Molino manifold $W$.

\subsection{Transverse differential operators}\label{s:transversediff}
We begin with a brief overview of an approach to transverse differential operators due to El Kacimi-Alaoui \cite{Kacimi}, to which we refer the reader for greater detail. Let $E$ be a basic complex vector bundle. Let $C^\infty_{E,\f}$ denote the sheaf $U\mapsto C^\infty(U,E)^{\f_U}$, where $\f_U$ is the induced foliation (with connected leaves) of $U$. For $r \in \bZ_{\ge 0}$ and $m \in M$ define
\[ \J^r(E/\f)_m=C^\infty_{E,\f,m}/\mf{m}_m^{r+1}C^\infty_{E,\f,m} \]
where $C^\infty_{E,\f,m}$ is the stalk at $m$ and $\mf{m}_m \subset C^\infty_{\f,m}$ is the vanishing ideal of $m$. $\J^r(E/\f)_m$ may be identified with the space of $r$-jets at $m$ of sections of $E|_T$ where $T$ is a transversal through $m$. The fibres $\J^r(E/\f)_m$, $m \in M$ fit together into a smooth foliated fibre bundle $\J^r(E/\f)\rightarrow M$, and there is a map of sheaves $\J^r \colon C^\infty_{E,\f}\rightarrow \J^r(E/\f)$.

\begin{definition}
Let $E,F$ be basic vector bundles and $r \in \bZ_{\ge 0}$. An $r$-th order transverse differential operator is a map of sheaves of vector spaces $D\colon C^\infty_{E,\f}\rightarrow C^\infty_{F,\f}$ that factors through a smooth foliated bundle morphism $\J^r(E/\f)\rightarrow F$.
\end{definition}

A transverse differential operator $D$ induces an ordinary $O_q$-equivariant differential operator $\T(D)\colon C^\infty(W,\T(E))\rightarrow C^\infty(W,\T(F))$ as follows. Let $T_m$ be a small transversal through $m \in M$. Given $p \in \pi_M^{-1}(m)$, lift $T_m$ to a diffeomorphic submanifold $T_p \subset P$ through $p$ using parallel translation along radial geodesics in $T_m$ (making use of the transverse Levi-Civita connection). There is an induced map $\J^r(\pi_M^*E/\f_P)_p\rightarrow \J^r(E/\f)_m$ obtained by restricting jets to $T_p \simeq T_m$. This map determines a lifted $O_q$-equivariant transverse differential operator
\[ l(D)\colon C^\infty_{\pi_M^*E,\f_P}\rightarrow C^\infty_{\pi_M^*F,\f_P}. \]
Taking global sections we obtain the differential operator
\[ \T(D)\colon C^\infty(W,\T(E))\rightarrow C^\infty(W,\T(F)).\]
Notice also that by construction the restriction of $\T(D)$ to $O_q$-invariant sections \emph{coincides} with $D$ under the identifications $C^\infty(M,E)^{\f}\simeq C^\infty(W,\T(E))^{O_q}$, $C^\infty(M,F)^{\f}\simeq C^\infty(W,\T(F))^{O_q}$.

In case $D$ is of order $0$ (a foliated bundle morphism), this coincides with the previous definition of $\T(D)$. Transverse vector fields $X \in \x(M/\f)$ are order $1$ examples (note however that the lift $l(X)$ used here is different from the lift used in Section \ref{s:tact} to define the $\g$-action on $W$, in particular here there is no requirement that $X$ be transverse Killing). A more interesting order $1$ example is a \emph{transverse Dirac operator}. For the next definition assume $q$ is even and that $(M,\f)$ is transversely oriented with $P$ denoting the transverse oriented orthonormal frame bundle.

\begin{definition}
\label{d:basicspinc}
A \emph{transverse basic Spin$^c$ structure} on $(M,\f,g)$ is a basic principal $\tn{Spin}^c_q$-bundle $(Q,\f_Q)$ together with a foliated $\tn{Spin}^c_q$-equivariant bundle map $(Q,\f_Q)\rightarrow (P,\f_P)$, and that furthermore admits a \emph{basic Clifford connection}, that is, a basic connection $\theta$ such that the composition of $\theta$ with the projection $\mf{spin}^c_q\rightarrow \mf{so}_q$ equals the pullback of the transverse Levi-Civita connection.
\end{definition}
Given a transverse basic Spin$^c$ structure $(Q,\f_Q)$, the corresponding spinor bundle is the basic $\bZ_2$-graded Hermitian vector bundle
\[ S=Q\times_{\tn{Spin}^c_q}\Delta \]
where $\Delta$ is the unique irreducible representation of $\bC l(\bR^q)$. The Clifford action is denoted $c\colon \bC l(N^*\f)\xrightarrow{\sim}\End(S)$. Choose a basic Clifford connection on $Q$ and let $\nabla$ be the corresponding basic connection on $S$. For any $\f$-invariant section $s \in C^\infty(M,S)^\f$, $\nabla_X s=0$ for $X \in \mf{X}(\f)$ and hence $\nabla s$ may be regarded as a section of $N^*\f\otimes S$.
\begin{definition}
\label{d:transverseDirac}
The \emph{transverse Dirac operator} $D$ associated to $(S,\nabla)$ is the composition
\[ C^\infty(M,S)^\f\xrightarrow{\nabla}C^\infty(M,N^*\f\otimes S)^\f \xrightarrow{c} C^\infty(M,S)^\f.\]
\end{definition}
\begin{remark}
One can easily define Dirac operators on more general Clifford modules, see for example \cite{BruningRiemFoliations} and references therein. The Clifford connection condition is unimportant for our purposes.
\end{remark}

\subsection{Transverse symbols and transverse ellipticity}
\begin{definition}
A transverse differential operator $D\colon C^\infty_{E,\f}\rightarrow C^\infty_{F,\f}$ of order $r$ has a \emph{transverse principal symbol} $\sigma_D \in C^\infty(M,\Sym^r(N\f)\otimes \Hom(E,F))^\f$ (note that $\Sym^r(N\f)_m$ may be identified with homogeneous polynomials of degree $r$ on the fibre $N^*\f$). 
\end{definition}

\begin{definition}
Let $E,F$ be basic vector bundles. A (polynomial) \emph{transverse symbol} is a section $\sigma \in C^\infty(M,\Sym^{\le r}(N\f)\otimes \Hom(E,F))^\f$ for some $r\ge 0$. The \emph{support} $\supp(\sigma)$ of $\sigma$ is the subset of $N^*\f$ where $\sigma$ fails to be invertible. $\sigma$ is $\f$-\emph{transversely elliptic} if $\supp(\sigma)/\ol{\f}_{N^*\f}$ is a compact subset of $N^*\f/\ol{\f}_{N^*\f}$. A transverse differential operator $D\colon C^\infty_{E,\f}\rightarrow C^\infty_{F,\f}$ of order $r$ is $\f$-\emph{transversely elliptic} if the principal symbol $\sigma_D$ is an $\f$-transversely elliptic symbol.
\end{definition}
Assuming $M/\ol{\f}$ is compact, then an example of an $\f$-transversely elliptic operator is a transverse Dirac operator, for which one has $\sigma_D(\xi)=c(\xi) \in \End(S)$, Clifford multiplication by $\xi \in N^*\f$. Examples of $\f$-transversely elliptic operators for Riemannian foliations can be found in \cite{Kacimi,glazebrook1991transversal,
habib2009brief,park1996basic,BruningRiemFoliations,GorokhovskyLott}.

Let $a \colon \g \rightarrow \scr{K}(M/\f,g)$ be a transverse isometric action. Recall that there is a Lie algebroid $\g\ltimes \f$ describing the infinitesimal actions of $\f$, $\g$. Let $N_\g^*\f=T^*_{\g\ltimes \f}M \subset N^*\f$ denote the conormal space to the $\g\ltimes \f$-orbits (not a vector bundle in general). Note that $N_\g^*\f$ is saturated for the singular foliation $\ol{\f}_{N^*\f}$; we write $\ol{\f}_{N^*_\g\f}$ for the induced equivalence relation on the topological space $N^*_\g\f$.
\begin{definition}
Let $E,F$ be $\g\ltimes \f$-equivariant basic vector bundles, and let $\sigma\in C^\infty(M,\Sym^{\le r}(N\f)\otimes \Hom(E,F))^{\g\ltimes\f}$ be a ($\g$-invariant) transverse symbol. $\sigma$ is $\g\ltimes \f$-\emph{transversely elliptic} if $(\supp(\sigma)\cap N_\g^*\f)/\ol{\f}_{N_\g^*\f}$ is a compact subset of $N_\g^*\f/\ol{\f}_{N_\g^*\f}$. 
\end{definition}
Given a transverse symbol $\sigma\in C^\infty(M,\Sym^{\le r}(N\f)\otimes \Hom(E,F))^\f$, applying the functor $\T$ results in a section $\T(\sigma)$ of $\Sym^{\le r}(\T(N\f))\otimes \Hom(\T(E),\T(F)))$. Analogous to $\sigma$, $\T(\sigma)$ may be regarded as a $\Hom(\T(E),\T(F))$-valued function on $\T(N\f)^*=\T(N^*\f)$, fiberwise polynomial of degree at most $r$, and in particular $\supp(\T(\sigma)) \subset \T(N^*\f)$ is defined. 

Let $\pi_{M*} \colon \pi_M^*N^*\f\rightarrow N^*\f$, $\pi_{W*} \colon \pi_W^*\T(N^*\f) \rightarrow \T(N^*\f)$ be the canonical maps. Recall $\pi_W^*\T(N^*\f)=\pi_M^*N^*\f$ (see Remark \ref{rem:ranks}). It is convenient to extend our notation $\T(-)$ to invariant closed subsets of $N^*\f$ as follows. 
\begin{definition}
\label{d:transfersubset}
Let $V\subset N^*\f$ be an $\f_{N^*\f}$-invariant closed subset. Define the $O_q$-invariant closed subset
\[ \T(V)=\pi_{W*}\big(\pi_{M*}^{-1}V\big)\subset \T(N^*\f).\]
\end{definition}
\begin{proposition}
\label{p:transfersubsets}
If $\sigma$ is a symbol, then 
\[ \T(\supp(\sigma))=\supp(\T(\sigma)).\]
Let $V,V' \subset N^*\f$ be closed and $\f_{N^*\f}$-invariant, then
\[ \T(V\cap V')=\T(V)\cap \T(V'), \quad \text{and} \quad V/\ol{\f}_{N^*\f}\simeq \T(V)/O_q.\]
\end{proposition}
\begin{proof}
The first claim is immediate from the definitions. The second claim follows from the definition and the fact that $\pi_{M*}^{-1}(V),\pi_{M*}^{-1}(V')$ are separately unions of fibers of $\pi_{W*}$. 

For the last claim, we begin by recalling that $\pi_W^*\T(N^*\f)=\pi_M^*N^*\f\simeq P\times \bR^q$ as foliated vector bundles, thus $\pi_{W*}=\pi_W\times \id_{\bR^q}$. Let $\L_{N^*\f}\subset N^*\f$ be a leaf of $\f_{N^*\f}$ covering a leaf $\L\subset M$ of $\f$, and let $v \in \L_{N^*\f}\cap N^*_m\f$. Choose any $p \in \pi_M^{-1}(m)$, let $\L_{P,p}$ be the lift of $\L$ passing through $p$, and let $p(v)\in \bR^q$ denote the components of $v$ in the frame $p$. It follows from the definitions that
\[ \pi_{M*}^{-1}(\L_{N^*\f})=O_q\cdot (\L_{P,p}\times \{p(v)\}),\]
where $O_q$ acts diagonally on $P\times \bR^q$. Taking the closure of both sides,
\[ \pi_{M*}^{-1}(\ol{\L}_{N^*\f})=O_q\cdot (\ol{\L}_{P,p}\times \{p(v)\}).\]
Applying $\pi_{W*}=\pi_W\times \id_{\bR^q}$ yields
\[ \T(\ol{\L}_{N^*\f})=(\pi_W\times \id_{\bR^q})\big(O_q\cdot (\ol{\L}_{P,p}\times \{p(v)\})\big)=O_q\cdot (\pi_W(p)\times \{p(v)\}).\] 
This sets up a one-one correspondence between $\f_{N^*\f}$ leaf closures and $O_q$-orbits in $\T(N^*\f)=W\times \bR^q$, extending the one-one correspondence between $\f$-leaf closures and $O_q$ orbits in $W=W\times \{0\}$. The last claim follows.
\end{proof}
We now return to discussing a transverse symbol $\sigma\in C^\infty(M,\Sym^{\le r}(N\f)\otimes \Hom(E,F))^\f$. In order to obtain a symbol on $W$ from $\T(\sigma)$, we use the morphism $h \colon N\f \rightarrow N\ol{\f}_P/O_q$, which we recall (see equation \eqref{e:anchorsplitting}) is the composition of horizontal lift $N\f\rightarrow N\f_P/O_q$ with the quotient map $N\f_P/O_q\rightarrow N\ol{\f}_P/O_q$. The map $\T(h)$ is a bundle morphism from $\T(N\f)$ to $\T(N\ol{\f}_P/O_q)=TW$. Therefore composing $\T(\sigma)$ with $\Sym^{\le r}(\T(h))\colon \Sym^{\le r}(\T(N\f))\rightarrow \Sym^{\le r}(TW)$ yields a symbol on $W$:
\[ \T_h(\sigma):=\Sym^r(\T(h))\T(\sigma) \in C^\infty(W,\Sym^{\le r}(TW)\otimes \Hom(\T(E),\T(F))).\]
In case $\sigma=\sigma_D$ is the principal symbol of a transverse differential operator, $\T_h(\sigma_D)=\sigma_{\T(D)}$ is the principal symbol of the operator $\T(D)$ on the Molino manifold.

\begin{theorem}
\label{t:Ttransell}
Let $E,F$ be $\g\ltimes \f$-equivariant basic vector bundles and let $\sigma \in C^\infty(M,\Sym^{\le r}(N\f)\otimes \Hom(E,F))^{\g\ltimes \f}$ be a $\g\ltimes \f$-transversely elliptic symbol. Then $\supp(\T_h(\sigma))\cap T_{\g\times \mf{o}_q}^*W$ is a closed subset of the compact set $\T(\supp(\sigma)\cap N_\g^*\f)$.
\end{theorem}
\begin{proof}
Let $\xi \in T^*W$. By definition $\T(\sigma)$ is a section of $\Sym^{\le r}(\T(N\f))\otimes \Hom(\T(E),\T(F))$ and $\T_h(\sigma)$ is given by
\[ \T_h(\sigma)(\xi)=\T(\sigma)(\T(h)^*\xi),\]
where $\T(h)^*\xi=\xi\circ \T(h)\in \T(N\f)^*$, and $\T(\sigma)$ is to be regarded as a function on the total space of $\T(N\f)^*$ (polynomial along the fibres) with values in $\Hom(\T(E),\T(F))$. By assumption $(\supp(\sigma)\cap N_\g^*\f)/\ol{\f}_{N^*\f}$ is compact. By Proposition \ref{p:transfersubsets}, 
\[(\supp(\T(\sigma))\cap \T(N_\g^*\f))/O_q=\T(\supp(\sigma)\cap N_\g^*\f)/O_q \simeq (\supp(\sigma)\cap N_\g^*\f)/\ol{\f}_{N^*\f} \] 
is compact, and hence $\T(\supp(\sigma)\cap N_\g^*\f)=\supp(\T(\sigma))\cap \T(N_\g^*\f)$ is itself compact.

Notice that $\T(h)=\T(b)\circ \T(l)\colon \T(N\f)\rightarrow TW$ is transverse to the $\mf{o}_q$ orbits, because $l(N\f)$ is complementary to the vertical bundle $\Ad(P)=(P\times \mf{o}_q)/O_q \subset N\f_P/O_q$. It follows that the restriction of $\T(h)^*$ to $T^*_{\mf{o}_q}W$ is injective. Restricting further to $T^*_{\g\times \mf{o}_q}W$, we conclude that the map
\begin{equation} 
\label{e:Th}
\T(h)^* \colon T^*_{\g\times \mf{o}_q}W\rightarrow \T(N\f)^* 
\end{equation}
is injective.

Let $\xi \in T^*_{\g\times \mf{o}_q}W$ and $X \in \g$. We have
\[ (\T(h)^*\xi)(\T(a(X)))=\xi(\T(b(l(a(X)))))=\xi(\T(b(a^\dagger(X)-\theta(a(X))))) \]
and both terms vanish separately since $\T(b(a^\dagger(X)))$ lies in the $\g$-orbit directions while $\T(b(\theta(a(X))))$ lies in the $\mf{o}_q$-orbit directions. This proves that the image of the map \eqref{e:Th} is contained in $\T(N_\g^*\f)$, and therefore \eqref{e:Th} is an injective map with image contained in $\T(N_\g^*\f)$. This map identifies $\supp(\T_h(\sigma))\cap T_{\g \times \mf{o}_q}^*W$ with a closed subset of the compact set $\supp(\T(\sigma))\cap \T(N_\g^*\f)=\T(\supp(\sigma)\cap N_\g^*\f)$.
\end{proof}

\begin{corollary}
Let $E,F$ be $\g\ltimes \f$-equivariant basic Hermitian vector bundles and let $\sigma \in C^\infty(M,\Sym^{\le r}(N\f)\otimes \Hom(E,F))^\f$ be a $\g\ltimes \f$-transversely elliptic symbol. Assume that the closure $\ol{\rho(G)}$ of the image of $G$ in the isometry group of $W$ is compact. Then $\T_h(\sigma) \in C^\infty(W,\Sym^{\le r}(TW)\otimes \Hom(\T(E),\T(F)))$ is a $\ol{\rho_{E\oplus F}(G)}\times O_q$-transversely elliptic symbol. If $\sigma$ is $\f$-transversely elliptic, then $\T_h(\sigma)$ is $O_q$-transversely elliptic (and $\ol{\rho_{E\oplus F}(G)}$-equivariant).
\end{corollary}

\subsection{The index of a $\g\ltimes \f$-transversely elliptic symbol}
Let $K$ be a compact Lie group. A $K$-transversely elliptic symbol $\sigma_W$ on $W$ has an analytic index defined by Atiyah \cite{AtiyahTransEll} that we denote by $\index_K(\sigma_W) \in R^{-\infty}(K)$, where $R^{-\infty}(K)$ denotes the set of possibly infinite formal linear combinations of irreducible complex representations of $K$ with finite multiplicities. The index only depends on the homotopy class of the symbol and is defined regardless of whether the underlying manifold $W$ is compact, so long as $\supp(\sigma_W)\cap T^*_K W$ is compact. An additional more recent reference for the theory of transversely elliptic symbols is \cite{WittenNonAbelian}.

Let $R(\g)=R(G)$ denote the ring generated by finite dimensional irreducible complex representations of $\g$, or equivalently of $G$, the connected simply connected integration of $\g$. Let $R^{-\infty}(\g)=R^{-\infty}(G)$ denote the $R(\g)$-module of possibly infinite formal linear combinations of irreducible finite dimensional representations of $\g$ with finite multiplicities. Let $K$ be a compact Lie group and suppose $\rho \colon G \rightarrow K$ is a smooth homomorphism with dense image. By continuity and density, any irreducible representation of $K$ remains irreducible after restriction along $\rho$. Therefore we get well-defined, injective restriction maps $R(K)\rightarrow R(G)=R(\g)$ and $R^{-\infty}(K)\rightarrow R^{-\infty}(G)=R^{-\infty}(\g)$.
\ignore{
\begin{lemma}
Let $G$, $K$ be Lie groups with $K$ compact. Let $\rho\colon G \rightarrow K$ be a smooth homomorphism with dense image. Restriction of representations along $\rho$ induces well-defined, injective maps $R(\ol{G})\rightarrow R(G)$ and $R^{-\infty}(\ol{G})\rightarrow R^{-\infty}(G)$.
\end{lemma}
\begin{proof}
Let $(V,\pi)$ be a finite dimensional complex representation of $\ol{G}$. Since $\ol{G}$ is compact, we may assume the representation is unitary. Then the induced representation $\pi\circ \rho$ of $G$ is unitary, hence any $G$-invariant subspace $W\subset V$ has a $G$-invariant complement $W^\perp$. This shows that restriction along $\rho$ induces a well-defined map from $R(\ol{G})$ to $R(G)$. 

If two finite dimensional representations of $\ol{G}$ are isomorphic as representations of $G$, then the isomorphism must intertwine the $\ol{G}$ actions as well by continuity. This shows that the map $R(\ol{G})\rightarrow R(G)$ is injective, and also that the map $R^{-\infty}(\ol{G}) \rightarrow R^{-\infty}(G)$ is well-defined (multiplicities remain finite).
\end{proof}
}
\begin{definition}
\label{d:symbolindex}
Let $(M,\f,g)$ be a complete Riemannian foliation equipped with a transverse isometric $\g$ action, and with Molino manifold $W$. Let $E,F$ be $\g\ltimes \f$-equivariant basic Hermitian vector bundles and let $\sigma \in C^\infty(M,\Sym^{\le r}(N\f)\otimes \Hom(E,F))^\f$ be a $\g\ltimes \f$-transversely elliptic symbol. Assume that the closure $\ol{\rho(G)}$ of the image of $G$ in the isometry group of $W$ is compact. We define $\index^\f_\g(\sigma) \in R^{-\infty}(\g)$ to be the element obtained by restriction of $\index_{\ol{\rho_{E\oplus F}(G)}\times O_q}(\T_h(\sigma))^{O_q}$.
\end{definition}

\begin{remark}
If $G$ is the simply connected integration of $\g$ then $R^{-\infty}(\g)\simeq R^{-\infty}(G)$, and accordingly we use the notation $\index^\f_G$ and $\index^\f_\g$ interchangeably. Occasionally there may be additional information guaranteeing that $\index^\f_\g(\sigma)$ lies in the subspace $R^{-\infty}(G')\subset R^{-\infty}(G)=R^{-\infty}(\g)$ for some particular integration $G'$ of $\g$. To remember this, we sometimes use the notation $\index^\f_{G'}(\sigma)$ for the index.
\end{remark}

\begin{remark}
\label{r:indexmetricindependence}
Although a choice of complete $\g\ltimes \f$-invariant transverse metric appears in Definition \ref{d:symbolindex}, the index does not depend on this choice. This follows from Proposition \ref{p:metricdependence} and Remark \ref{r:canonicalisoequivariance}. It is not difficult to see that the index is also independent of the choices of Hermitian structures on $E,F$.
\end{remark}

\begin{theorem}
\label{t:indexell}
Let $(M,\f,g)$ be a complete Riemannian foliation with a transverse isometric $\g$-action. Assume $M/\ol{\f}$ is compact. Let $E,F$ be $\g\ltimes \f$-equivariant basic Hermitian vector bundles. Let $D$ be an $\f$-transversely elliptic operator. Passing to global sections results in an operator that we denote by the same symbol $D \colon C^\infty(M,E)^\f \rightarrow C^\infty(M,F)^\f$. Then
\[ \ker(D)\simeq \ker(\T(D))^{O_q}, \quad \tn{coker}(D)=C^\infty(M,F)^\f/D(C^\infty(M,E)^\f)\simeq \tn{coker}(\T(D))^{O_q} \] 
are finite dimensional representations of $\ol{\rho_{E\oplus F}(G)}$. Define
\[ \index^\f_\g(D) \in R(\g) \]
to be the formal difference of the kernel and cokernel of $D$, restricted to $\g$. Then $\index^\f_\g(D)=\index^\f_\g(\sigma_D)$.
\end{theorem}
\begin{proof}
Note that $M/\ol{\f}\simeq W/O_q$ is compact if and only if $W$ is compact. The isomorphisms $C^\infty(M,E)^\f\simeq C^\infty(W,\T(E))^{O_q}$, $C^\infty(M,F)^\f \simeq C^\infty(W,\T(F))^{O_q}$ identify $D$ with the restriction $\T(D)^{O_q}$ of the $O_q$-transversely elliptic operator $\T(D)$ to $O_q$-invariant sections. The kernel and cokernel are thus finite dimensional by a fundamental result on transversely elliptic operators \cite[Corollary 2.5]{AtiyahTransEll}. Since $\T(D)^{O_q}$ is $G$-equivariant, the kernel and cokernel are finite dimensional representations of $G$. The Lie group $\ol{\rho_{E\oplus F}(G)}$ acts on the domain and codomain of $\T(D)$. As the image $G \rightarrow \ol{\rho_{E\oplus F}(G)}$ is dense and commutes with $O_q$ and $\T(D)$, the $\ol{\rho_{E\oplus F}(G)}$ action also commutes with $O_q$ and $\T(D)$. Thus the kernel and cokernel of $\T(D)^{O_q}$ are representations of the compact Lie group $\ol{\rho_{E\oplus F}(G)}$. The last claim is immediate from $\sigma_{\T(D)}=\T_h(\sigma)$.
\end{proof} 

\begin{example}[locally free isometric actions]
\label{ex:freeisometric3}
Let $(M,\f,g)$ be the Riemannian foliation considered in Examples \ref{ex:freeisometric}, \ref{ex:freeisometric2}. Assume $(M,g_M)$ is complete and that $\ol{H}\subset \tn{Isom}(M,g_M)$ acts cocompactly on $M$. Then the Molino manifold $W=P/\ol{H}$ is compact. Choose a differential operator $D_M$ on $M$ whose restriction to $\f$-invariant sections equals $D$; for example, if $D$ is a transverse Dirac operator, then an extension $D_M$ is determined by the choice of a bundle-like metric $g_M$ compatible with $g$. Then
\[ \index^\f(D)=\index^H(D_M) \]
where $\index^H(D_M)$ denotes the index of $D_M$ restricted to the space of smooth $H$-invariant (equivalently, $\ol{H}$-invariant) sections. We can describe the index in `compact terms' using Abel's slice theorem \cite{abelslice}, which says that $M$ admits a global slice $M\simeq \ol{H}\times_K N$, where $K\subset \ol{H}$ is a maximal compact subgroup and $N$ is a compact $K$-manifold. By continuity, $D_M$ is $\ol{H}$-invariant and hence $\ol{H}$-transversely elliptic. Using a principal $K$-connection to lift $D_M$ to $\ol{H}\times N$ and then restricting to $\ol{H}$-invariant sections yields a $K$-transversely elliptic operator $D_N$ on $N$, with the property that the restriction of $D_N$ to $K$-invariant sections over $N$ agrees with the restriction of $D_M$ to $\ol{H}$-invariant sections over $M$, under the canonical identification. Then 
\[ \index^\f(D)=\index_K(D_N)^K.\]
\end{example}
\begin{example}[Haefliger suspensions]
\label{ex:Haefliger3}
Let $(M,\f,g)$ be the Riemannian foliation considered in Examples \ref{ex:Haefliger}, \ref{ex:Haefliger2}. Assume $(N,g_N)$ is complete and that $\ol{\Gamma}\subset \tn{Isom}(N,g_N)$ acts cocompactly on $N$. Then the Molino manifold $W\simeq P_N/\ol{\Gamma}$ is compact. The index of an $\f$-transversely elliptic operator $D$ can be described in terms of $N$. Indeed $D$ lifts uniquely to an $\ti{\f}$-transversely elliptic operator $\ti{D}$ on the $\Gamma$-covering space $\ti{M}=\ti{B}\times N$. Restricting to $\ti{\f}$-invariant sections yields a $\Gamma$-invariant differential operator $D_N$ on $N$, with the property that the restriction of $D$ to $\f$-invariant sections agrees with the restriction of $D_N$ to $\Gamma$-invariant sections, under the canonical identification. Therefore 
\[ \index^\f(D)=\index^\Gamma(D_N) \]
where $\index^\Gamma(D_N)$ denotes the index of $D_N$ restricted to the space of smooth $\Gamma$-invariant (equivalently $\ol{\Gamma}$-invariant) sections.
\end{example}

\subsection{Weak equivalence and the index}\label{s:weakindex}
Let $(B,\pi_1,\pi_2)$ be an equivariant isometric weak equivalence of complete Riemannian foliations $(M_i,\f_i,g_i)$, $i=1,2$. Suppose $E_i,F_i$ are $\g\ltimes \f$-equivariant basic Hermitian vector bundles on $M_i$ for $i=1,2$, which are intertwined by the weak equivalence, and fix isomorphisms $\pi_1^*E_1\simeq \pi_2^*E_2$, $\pi_1^*F_1\simeq \pi_2^*F_2$. Since $\pi_1^*N\f_1\simeq \pi_2^*N\f_2$, $B$ induces a one-one correspondence between foliation-invariant symbols: $\sigma_1,\sigma_2$ are in correspondence if $\pi_1^*\sigma_1=\pi_2^*\sigma_2$ under the isomorphisms $\pi_1^*N\f_1\simeq \pi_2^*N\f_2$, $\pi_1^*E_1\simeq \pi_2^*E_2$, $\pi_1^*F_1\simeq \pi_2^*F_2$. Since $M_1/\ol{\f}_1\simeq M_2/\ol{\f}_2$, $B$ induces a one-one correspondence between $\g\ltimes \f_i$-transversely elliptic symbols. 
\begin{proposition}
\label{p:Moritaindex}
An equivariant isometric weak equivalence preserves the index of transversely elliptic symbols.
\end{proposition}
\begin{proof}
As noted above, a weak equivalence induces a one-one correspondence between transversely elliptic symbols. If $\sigma_1,\sigma_2$ are transversely elliptic symbols related via this one-one correspondence, then the transversely elliptic symbols $\T_{h_1}(\sigma_1)$, $\T_{h_2}(\sigma_2)$, on $W_1,W_2$ respectively, are intertwined by the isometry $W_1\simeq W_2$ described in Proposition \ref{p:moritamolino}. Thus they have the same index.
\end{proof}

By Remark \ref{r:indexmetricindependence}, the index of a transversely elliptic symbol is also insensitive to the choice of complete $\g\ltimes \f$-invariant transverse metric and to the choices of Hermitian metrics on $E,F$. This flexibility may be combined with Proposition \ref{p:Moritaindex} to see that the index of a transversely elliptic symbol is nearly an invariant of the weak equivalence class of the foliation $(M,\f)$ and homotopy class of the transverse symbol $(\sigma,E,F)$. For example, let us call a foliation $(M,\f)$ \emph{completable} if it admits a complete bundle-like metric. Then the index is invariant under equivariant completable weak equivalences of completable foliations intertwining the transverse symbols in the sense discussed above.

\subsection{Infinitesimally free actions}\label{s:free}
Let $(M,\f,g)$ be a complete Riemannian foliation with a transverse isometric action $a \colon \g \times \h \rightarrow \scr{K}(M/\f,g)$, where $\g,\h$ are Lie algebras with simply connected integrations $G,H$ respectively. If the anchor map for the Lie algebroid $\h\ltimes \f$ is injective then it defines a new complete Riemannian foliation $\ti{\f}$ whose leaves are the orbits of the $\h\ltimes \f$-action. 

Suppose $\theta$ is an $\h\ltimes \f$-equivariant $\f$-basic connection on a foliated principal bundle $(Q,\f_Q)$. To obtain an $\ti{\f}$-basic connection, replace $\theta$ with $\theta-\theta\circ p$, where $p\colon N\f_Q\rightarrow N\f_Q$ is projection onto the subbundle $a^\dagger(\h)\subset N\f_Q$ along the complementary subbundle $\ker(\theta)\oplus l(a(\h)^\perp)$. 
\ignore{
To see this consider a local foliation chart $U$ with local quotient $U/\f|_U=T$. The basic condition means that $(Q,\theta)|_U$ is pulled back from a pair $(Q_T,\theta_T)$ on $T$. One replaces the connection $\theta_T$ with $\theta_T-\theta_T\circ p_T$, where $p_T \colon TQ_T\rightarrow TQ_T$ is projection onto the $\h$-orbit directions $a_{Q_T}(\h)$ along $\ker(\theta_T)\oplus l(a_T(\h)^\perp)$.
Since the construction is canonical, it glues on overlaps of foliation charts. (Simpler: one can write an explicit formula for the modified connection globally on $M$ just involving smooth basic objects. Alternatively one could just do this at a point in $M$ and then observe  that the formula just involves globally defined objects.)} It follows that a Hermitian $\h_M\ltimes \f$-equivariant $\f$-basic vector bundle determines a Hermitian $\ti{\f}$-basic vector bundle, and this correspondence intertwines invariant sections. One has $N^*\ti{\f}=N^*_\h \f$ and there is a quotient map $N\f\rightarrow N\ti{\f}$ inducing a map $\Sym^{\le r}(N\f)\rightarrow \Sym^{\le r}(N\ti{\f})$.
\begin{theorem}
\label{t:infinitesimalactionfoliation}
Let $(M,\f,g)$ be a complete Riemannian foliation with a transverse isometric action $a \colon \g \times \h \rightarrow \scr{K}(M/\f,g)$. Assume that the closure $\ol{\rho(G\times H)}$ of the image of $G \times H$ in the isometry group of the Molino manifold $W$ is compact. Suppose the anchor map for the Lie algebroid $\h_M\ltimes \f$ is injective, and let $\ti{\f}$ be the corresponding foliation. Let $E,F$ be Hermitian $(\g\times \h)\ltimes \f$-equivariant $\f$-basic vector bundles and let $\sigma \in C^\infty(M,\Sym^{\le r}(N\f)\otimes \Hom(E,F))^\f$ be a $(\g\times \h)\ltimes \f$-transversely elliptic symbol. Let $\ti{\sigma} \in C^\infty(M,\Sym^{\le r}(N\ti{\f})\otimes \Hom(E,F))^{\ti{\f}}$ be the corresponding $\g\ltimes \ti{\f}$-transversely elliptic symbol. Then
\[ \index^\f_{\g\times \h}(\sigma)^\h=\index_\g^{\ti{\f}}(\ti{\sigma}).\]
\end{theorem} 
\begin{proof}
To simplify notation we will assume $\g=0$, but the general case involves no additional difficulty. Let $\ti{q}=q-h$, $h=\dim(\h)$ be the codimension of $\ti{\f}$, and let $\ti{P}$ be the principal $O_{\ti{q}}$-bundle of orthonormal frames of $N\ti{\f}$. Using the transverse metric $g$ there is an orthogonal foliated splitting
\begin{equation} 
\label{e:split}
N\f=N\ti{\f}\oplus \h_M
\end{equation}
where $\h_M\simeq a(M\times \h) \subset N\f$ is identified with the subbundle of transverse $\h$-orbit directions. Let $\ti{P}'\subset P$ be the closed subbundle whose fibre $\ti{P}'_m$ at $m \in M$ is the subset of oriented orthonormal frames of $N_m\f$ such that the first $\ti{q}=q-h$ vectors in the frame form a frame of $N_m\ti{\f}$. Since \eqref{e:split} is $\f$-equivariant, $\ti{P}'$ is $\f_P$-saturated. Moreover $\ti{P}'$ is a principal $O_{\ti{q}}\times O_h$-bundle and
\begin{equation} 
\label{e:isos1}
P\simeq O_q\times_{O_{\ti{q}}\times O_h} \ti{P}', \qquad \ti{P}\simeq \ti{P}'/O_h.
\end{equation}
Let $\ti{W}$ be the Molino manifold of $\ti{\f}$ and let $\ti{W}'=\ti{P}'/\ol{\f}_P \subset W=P/\ol{\f}_P$. The first isomorphism in \eqref{e:isos1} descends to an isomorphism
\begin{equation}
\label{e:iso2}
W\simeq O_q\times_{O_{\ti{q}}\times O_h} \ti{W}'.
\end{equation}
In addition the induced action of $O_h$ on $\ti{W}'$ is free. To see this, suppose $p_i\rightarrow p$ is a convergent sequence of points in a leaf of $\f_P|_{\ti{P}'}$. Let $m_i=\pi_M(p_i)$ and $m=\pi_M(p)$. Each $p_i$ represents a frame consisting of a frame of $N_{m_i}\ti{\f}$ and the image of a frame $f_i$ of $\h$ under the action map $a_{m_i}\colon \h\rightarrow N_{m_i}\f$. But by $\f$-invariance of the action and since all $p_i$ lie on a common leaf of $\f_P$, the frame $f_i=f$ of $\h$ must be independent of $i$. Thus the last $h$ elements of the frame $p$ must be the frame $a_m(f)$ of $\h_M|_m$. Since the $O_h$ action is induced by change of frame in $\h$, this shows that if $1\ne g \in O_h$ and $p \in \ti{P}'$ then $p,g\cdot p$ are contained in distinct $\f_P$ leaf closures.

There is an isomorphism
\begin{equation}
\label{e:tildeWpp}
\ti{P}/\ol{\f}_{\ti{P}}\simeq \ti{W}'/O_h=:\ti{W}'' 
\end{equation}
where $\f_{\ti{P}}=\f_P|_{\ti{P}'}/O_h$ is the lift of the foliation $\f$ to $\ti{P}$. As in \eqref{e:split}, $N\f_{\ti{P}}\simeq N\ti{\f}_{\ti{P}}\oplus \h_{\ti{P}}$, where $\h_{\ti{P}}$ denotes the subbundle of orbit directions for the lifted $\h$ action. Thus the foliation $(\ti{P},\f_{\ti{P}})$ has a transverse parallelism, obtained by combining the frame $V_1,\ldots,V_d$, $d=\ti{q}(\ti{q}+1)/2$ of $N\ti{\f}_{\ti{P}}$ used in the Molino construction for $(M,\ti{\f})$ with the frame $V_1',\ldots,V_h'$ of $\h_{\ti{P}}$ obtained from any fixed basis of $\h$. The transverse vector fields $V_1,\ldots,V_d$ commute with $V_1',\ldots,V_h'$, because the $V_j'$ are natural lifts of transverse Killing vector fields on $M$ (the natural lift of any transverse Killing vector field commutes with the transverse vector fields in the tautological frame). The transverse vector fields $V_i$, $V_j'$ project to $\ti{W}''$ and span $T\ti{W}''$. It follows that for any two sufficiently close points $w_1,w_2\in \ti{W}''$ there exist $c_1,\ldots,c_d \in \bR$ such that the image of $w_1$ under the time-1 flow of $c_1T\pi_{\ti{W}''}(V_1)+\cdots+c_dT\pi_{\ti{W}''}(V_d)$ is in the orbit $H\cdot w_2$. As the $T\pi_{\ti{W}''}(V_i)$ commute with the $H$ action on $\ti{W}''$ we conclude that the $H$-action on $\ti{W}''$ has a single orbit-type in each connected component of $\ti{W}''$. By taking limits the same holds for the $K:=\ol{\rho(H)}$ action. To simplify notation we assume there is a single orbit-type. Thus all stabilizers are conjugates of a subgroup $K''\subset K$, and $\ti{W}''$ is a bundle of homogeneous spaces with typical fibre $K/K''$ over the quotient $\ti{W}''/K$. If we let $\iota \colon \tn{Stab}(K'',\ti{W}'')\hookrightarrow \ti{W}''$ denote the subbundle whose stabilizer equals $K''$, then
\begin{equation} 
\label{e:sliceW}
\ti{W}''\simeq K\times_{N(K'')}\tn{Stab}(K'',\ti{W}''),
\end{equation}
where $N(K'')$ is the normalizer of $K''$ inside $K$.

Let $\pi_{\ti{W}''}\colon \ti{P}\rightarrow \ti{W}''$ be the quotient map. Let $w \in \ti{W}''$ and $p \in \pi_{\ti{W}''}^{-1}(w)$. The closure of the leaf of $\ti{\f}_{\ti{P}}$ containing $p$ is $\pi_{\ti{W}''}^{-1}(\ol{H\cdot w})$. Moreover $\ol{H\cdot w}=\ol{\rho(H)}\cdot w=K\cdot w$; indeed suppose $w_2$ belongs to the closure of $H\cdot w_1$. If $h_i \in H$ is a sequence such that $h_iw_1\rightarrow w_2$ then there exists a subsequence $h_i'$ that converges to an isometry $h \in \ol{\rho(H)}=K$ mapping $w_1$ to $w_2$ and thus $w_2 \in K\cdot w_1$. Conversely any $h \in \ol{\rho(H)}=K$ preserves the closure of each $H$-orbit, being the limit of isometries that preserve $H$-orbits. It follows that there is an isomorphism
\begin{equation} 
\label{e:tildeW}
\ti{W}''/K\simeq \ti{W}
\end{equation}
as smooth $O_{\ti{q}}$-manifolds.

With these preparations we turn to comparing the two indices in the statement of the theorem. Let $K_{E\oplus F}=\ol{\rho_{E\oplus F}(G)}$. Let $\sigma_W=\T_h(\sigma)$ be the induced $O_q \times K_{E\oplus F}$-transversely elliptic symbol on $W$. By definition $\index^\f_\h(\sigma)=\index_{K_{E\oplus F}\times O_q}(\sigma_W)^{O_q}$. By \eqref{e:iso2}, the $O_q$-invariant part of the index of $\sigma_W$ equals the $O_{\ti{q}}\times O_h$-invariant part of the index of the $O_{\ti{q}}\times O_h\times K_{E\oplus F}$-transversely elliptic symbol $\sigma_{\ti{W}'}$ obtained by restriction of $\sigma_W$ to $\ti{W}'$. Since the $O_h$-action is free (see \eqref{e:tildeWpp}) the latter equals the $O_{\ti{q}}$-invariant part of the index of the induced $O_{\ti{q}}\times K_{E\oplus F}$-transversely symbol $\sigma_{\ti{W}''}$ on the quotient $\ti{W}''$ (\cite[Theorem 3.1]{AtiyahTransEll}). Let $\ti{E}'',\ti{F}''$ denote the corresponding vector bundles on $\ti{W}''$.

Let $K_{E\oplus F}''$ denote the inverse image of $K''$ under the quotient map $K_{E\oplus F}\rightarrow K$. By \eqref{e:sliceW},
\[ \ti{W}''=K_{E\oplus F}\times_{N(K_{E\oplus F}'')}\tn{Stab}(K''_{E\oplus F},\ti{W}'').\]
By \cite[Theorem 4.1]{AtiyahTransEll}, the $O_{\ti{q}}\times K_{E\oplus F}$-invariant part of the index of $\sigma_{\ti{W}''}$ is the $O_{\ti{q}}\times N(K_{E\oplus F}'')$-invariant part of the index of the pullback $\iota^*\sigma_{\ti{W}''}$ to $\tn{Stab}(K_{E\oplus F}'',\ti{W}'')$. 

Since $K_{E\oplus F}''$ acts trivially on the base $\tn{Stab}(K''_{E\oplus F},\ti{W}'')$, taking the $N(K_{E\oplus F}'')$-invariant part of the index of $\iota^*\sigma_{\ti{W}''}$ is equivalent to the following sequence of two steps: (i) take $K_{E\oplus F}''$-invariants in the vector bundle fibers, resulting in vector bundles $\ti{E}'$, $\ti{F}'$ (of possibly smaller rank), and a restricted symbol $\iota^*\sigma'_{\ti{W}''}$, (ii) taking the $N(K_{E\oplus F}'')/K_{E\oplus F}''$-invariant part of the index of $\iota^*\sigma'_{\ti{W}''}$. In step (ii), since $\tn{Stab}(K_{E\oplus F}'',\ti{W}'')$ is a principal $N(K_{E\oplus F}'')/K_{E\oplus F}''$-bundle over $\ti{W}''/K=\ti{W}$, \cite[Theorem 3.1]{AtiyahTransEll} implies that the result is the index of an induced $O_{\ti{q}}$-transversely elliptic symbol $\sigma_{\ti{W}}$ on $\ti{W}$. The vector bundle $\ti{E}\oplus \ti{F}=(\ti{E}'\oplus \ti{F}')/(N(K_{E\oplus F}'')/K_{E\oplus F}'')$ identifies naturally with $\ti{\T}(E\oplus F)$, where $\ti{\T}$ is the transfer functor for $\ti{\f}$, since $K_{E\oplus F}$-invariant sections of $\T(E\oplus F)|_{\ti{W}'}/O_h$ may be identified with sections of $\ti{\T}(E\oplus F)$. Finally $\sigma_{\ti{W}}=\ti{\T}_h(\ti{\sigma})$, and the $O_{\ti{q}}$-invariant part of the index of the latter is the definition of $\index^{\ti{\f}}(\ti{\sigma})$. Following the chain of identifications yields the equality of indices in the statement of the theorem.
\end{proof}

The following generalizes a well-known property of $K$-transversely elliptic symbols (cf. \cite[Theorem 3.1]{AtiyahTransEll}).
\begin{corollary}
Let $(Q,\f_Q)$ be a foliated principal $K$-bundle over $(M,\f)$, where $K$ is a compact Lie group. Let $\ti{\f}_Q$ be the foliation of $Q$ obtained by pulling back the foliation $\f$. An $\f$-transversely elliptic symbol $\sigma$ on $M$ pulls back to a $\k\ltimes \f_Q$-transversely elliptic symbol $\sigma_Q$ on $Q$, and
\begin{equation} 
\label{e:freeMorita}
\index^{\f_Q}_\k(\sigma_Q)^\k=\index^{\ti{\f}_Q}(\sigma_Q)=\index^\f(\sigma). 
\end{equation}
\end{corollary}
\begin{proof}
The first equality is Theorem \ref{t:infinitesimalactionfoliation}. The second equality follows from invariance of the index under weak equivalences, Proposition \ref{p:Moritaindex}.
\end{proof}

\subsection{Twisted products}\label{s:twistedprod}
Let $E\rightarrow M$ be a $\g\ltimes \f$-equivariant $\bZ_2$-graded basic Hermitian vector bundle. Let $\sigma \in C^\infty(M,\Sym^{\le r}(N\f)\otimes \End(E))^\f$ be an odd $\g\ltimes \f$-transversely elliptic symbol. Let $K$, $R$ be connected compact Lie groups and let $\pi_Q\colon Q \rightarrow M$ be a $\g\ltimes \f$-equivariant basic principal $K$-bundle. Let $V$ be a $K\times R$-manifold, let $F \rightarrow V$ be a $K\times R$-equivariant Hermitian vector bundle, and let $\sigma_V \in C^\infty(V,\Sym^{\le r}(TV)\otimes\End(F))$ be an odd $K\times R$-invariant $R$-transversely elliptic symbol. Choosing a basic connection on $Q$, the symbol $\sigma$ lifts to a $(\g\times \k)\ltimes \f_Q$-transversely elliptic symbol $\sigma_Q \in C^\infty(Q,\Sym^{\le r}(N\f_Q)\otimes\End(\pi_Q^*E))^{\f_Q}$. The product symbol
\[ \sigma_Q \wh{\boxtimes} 1+1\wh{\boxtimes}\sigma_V \in C^\infty(Q\times V,\Sym^{\le r}(N\f_Q \oplus TV)\otimes \End(\pi_Q^*E\wh{\boxtimes}F))^{\f_Q} \] 
is $(\g\times \k\times \r)\ltimes \f_Q$-transversely elliptic. It descends to a $(\g\times \r) \ltimes \f_\V$-transversely elliptic symbol $\sigma_\V$ on the associated bundle $\V=Q\times_K V$.

\begin{theorem}
\label{t:indextwistedproduct}
With notation and hypotheses as above, one has
\[ \index^{\f_\V}_{\g\times \mf{r}}(\sigma_\V)=\sum_{\chi \in \hat{R}}\sum_{\psi \in \hat{K}}n_{\chi\psi}\index^\f_\g(\sigma \otimes Q(\psi))\cdot \chi, \]
where $Q(\psi)=Q\times_K \psi$ is the associated $\g\ltimes \f$-equivariant vector bundle, and $n_{\chi\psi}$ is the multiplicity of $\chi\otimes \psi$ in $\index_{\k\times \mf{r}}(\sigma_V)$.
\end{theorem}
\begin{proof}
Let $\ti{\f}_\V$ be the pullback of the foliation $\f_\V$ under the map $Q\times V \rightarrow Q\times_K V=\V$. Since $\f_\V$-invariants on $\V$ are the same as $\ti{\f}_\V$-invariants on $Q\times V$, invariance of the index under weak equivalence (Proposition \ref{p:Moritaindex}) implies 
\[ \index_{\g\times \mf{r}}^{\f_\V}(\sigma_\V)=\index_{\g\times \mf{r}}^{\ti{\f}_\V}(\sigma_Q \wh{\boxtimes} 1+1\wh{\boxtimes}\sigma_V) \]
where on the right hand side we regard the product symbol as $(\g\times \r)\ltimes \ti{\f}_\V$-transversely elliptic. By Theorem \ref{t:infinitesimalactionfoliation},
\[ \index_{\g\times \mf{r}}^{\f_\V}(\sigma_\V)=\index_{\g\times \mf{r}\times \k}^{\f_Q}(\sigma_Q \wh{\boxtimes} 1+1\wh{\boxtimes}\sigma_V)^\k.\]
The foliation of the product manifold $Q\times V$ is the product of $\f_Q$ with the one-leaf foliation of $V$. In particular if $W_Q$ is the Molino manifold of $Q$, then the Molino manifold of $Q\times V$ is $W_Q\times V$. The multiplicative property of the index of transversely elliptic symbols (\cite[Theorem 3.5]{AtiyahTransEll}) on $W_Q\times V$ implies
\[ \index_{\g\times \k\times \mf{r}}^{\f_Q}(\sigma_Q \wh{\boxtimes} 1+1\wh{\boxtimes}\sigma_V)=\index_{\g\times \k}^{\f_Q}(\sigma_Q)\index_{\k\times \mf{r}}(\sigma_V).\]
Therefore
\begin{align*} 
\index_{\g\times \k\times \mf{r}}^{\f_Q}(\sigma_Q \wh{\boxtimes} 1+1\wh{\boxtimes}\sigma_V)^\k
&=\Big(\index^{\f_Q}_{\g\times \k}(\sigma_Q)\index_{\k\times \mf{r}}(\sigma_V)\Big)^\k \\
&=\sum_{\chi,\psi}n_{\chi\psi}\Big(\index^{\f_Q}_{\g\times \k}(\sigma_Q)\psi\Big)^\k\chi\\
&=\sum_{\chi,\psi}n_{\chi\psi}\index^\f_\g(\sigma\otimes Q(\psi))\chi,
\end{align*}
where in the last line we used equation \eqref{e:freeMorita}.
\end{proof}

The following instance of the above situation will be needed below. Let $V=\bR^{2r}$. Let $J$ denote a choice of complex structure on $V$, let $K=U(V,J)$ be the corresponding unitary group, and let $R\simeq U_1$ be the center of $K$. Let 
\[ \sigma_V(\xi)=c_{\wedge \ol{V}}(\xi-Jv) \]
where $c_{\wedge \ol{V}}$ is Clifford action on $\wedge \ol{V}$. Then $\sigma_V$ is $K\times R$-invariant and $R$-transversely elliptic. Atiyah \cite{AtiyahTransEll} showed that the $K\times R$-equivariant index of $\sigma_V$ is $\Sym(V)$, where the $R$-action is the grading of the symmetric algebra. Applying Theorem \ref{t:indextwistedproduct} yields the following.
\begin{corollary}
\label{c:symalg}
With notation and hypotheses as above, one has
\[ \index_{\g\times \mf{r}}^{\f_\V}(\sigma_\V)=\index_\g^{\f}(\sigma\otimes \Sym(\V)).\]
\end{corollary}

\subsection{Excision}
Atiyah \cite{AtiyahTransEll} proved an excision property for the index of a transversely elliptic symbol that leads to a similar excision property in our context. By Proposition \ref{p:openMolino} (see also Remark \ref{r:openMolinoequivariant}), an $\f$-saturated open subset $M'\subset M$ of a complete Riemannian foliation $(M,\f,g)$ admits a complete bundle-like metric, and hence may itself be considered a complete Riemannian foliation.

\begin{theorem}
\label{t:excision}
Let $(M,\f,g)$ be a complete Riemannian foliation equipped with a transverse isometric $\g$ action, and with Molino manifold $W$. Let $E,F$ be $\g\ltimes \f$-equivariant basic Hermitian vector bundles and let $\sigma \in C^\infty(M,\Sym^{\le r}(N\f)\otimes \Hom(E,F))^\f$ be a $\g\ltimes \f$-transversely elliptic symbol. Assume that the closure $\ol{\rho(G)}$ of the image of $G$ in the isometry group of $W$ is compact. Let $M'\subset M$ be a $\g\ltimes \f$-saturated open subset with restricted foliation $\f'$. Suppose $\sigma$ is invertible on $N^*_\g \f \backslash N^*_\g \f'$. Then $\sigma$ restricts to a $\g\ltimes \f'$ transversely elliptic symbol $\sigma'$ over $M'$ and
\[ \index_\g^\f(\sigma)=\index_\g^{\f'}(\sigma').\]
\end{theorem}
\begin{proof}
Since $\sigma$ is invertible on $N^*_\g \f \backslash N^*_\g \f'$, the restriction $\sigma'$ is transversely elliptic, and therefore $\index_\g^{\f'}(\sigma')$ is well-defined. The Molino manifold $W'$ of $M'$ is identified canonically (Proposition \ref{p:metricdependence}) with the open subset $\pi_W(\pi_M^{-1}(M'))\subset W$. By functoriality, the $\ol{\rho_{E\oplus F}(G)}\times O_q$ transversely elliptic symbol $\T_h(\sigma)$ on $W$ restricts to $\T_h(\sigma')$ on $W'$. By Theorem \ref{t:Ttransell}, there is a canonical inclusion
\begin{equation}
\label{e:exinclu}
\supp(\T_h(\sigma))\cap T_{\g\times \mf{o}_q}^*W \hookrightarrow \T(\supp(\sigma)\cap N^*_\g\f).
\end{equation} 
By assumption $\supp(\sigma)\cap N^*_\g\f\subset N^*_\g \f'$. Applying $\T$ we obtain $\T(\supp(\sigma)\cap N^*_\g\f)\subset \T(N^*_\g \f')$. It follows from this and \eqref{e:exinclu} that the support of $\T_h(\sigma)$ in $T^*_{\g\times \mf{o}_q}W$ lies entirely over $W'$:
\[ \supp(\T_h(\sigma))\cap T_{\g\times \mf{o}_q}^*W \subset T^*_{\g\times \mf{o}_q}W'.\] 
Therefore by Atiyah's excision theorem \cite[Theorem 3.7]{AtiyahTransEll}, 
\[ \index_{\ol{\rho_{E\oplus F}(G)}\times O_q}(\T_h(\sigma))=\index_{\ol{\rho_{E\oplus F}(G)}\times O_q}(\T_h(\sigma')),\]
from which the result follows.
\end{proof}

\section{Localization}\label{s:localization}
As an application of Theorem \ref{t:Ttransell} and its corollaries, we prove (Corollary \ref{c:generalloc}) a localization formula for the index of a transverse Dirac operator. Refinements of the formula will be explained in Sections \ref{sec:abel}, \ref{s:nonabelloc}. Throughout this section, $(M,\f,g)$ denotes a transversely oriented complete Riemannian foliation of even codimension $q$, with a transverse isometric action $a \colon \g \rightarrow \scr{K}(M/\f,g)$, and such that $M/\ol{\f}\simeq W/O_q$ is compact. The transverse metric $g$ will sometimes be used to identify $N\f\simeq N^*\f$.

\subsection{Taming maps and localization}\label{s:taming}
\begin{definition}
A \emph{taming map} (cf. \cite{Braverman2002,KarshonHarada} for similar terminology) is a $\g\ltimes \f$-equivariant map
\[ \mu\colon M \rightarrow \g \]
i.e. $\mu$ is $\f$-invariant and satisfies $a(X)\mu=\ad_X\circ \mu$ for all $X \in \g$.  Via the transverse action, $\mu$ generates a transverse vector field $a(\mu) \in \mf{X}(M/\f)$. (It is not transverse Killing, in general.)
\end{definition}
The \emph{localizing set} $Z_\mu$ of $\mu$ is the vanishing locus of the transverse vector field $a(\mu)$:
\[ Z_\mu=a(\mu)^{-1}(0)=\bigcup_{\beta \in \g} M_\beta \cap \mu^{-1}(\beta),\]
where $M_\beta$ is the fixed leaf set of $\beta$.
\begin{example}
\label{ex:beta}
Let $\beta \in \mf{z}\subset \g$ be central. Then the constant map $\mu_m=\beta$ for all $m \in M$ is a taming map. In this case $a(\beta)$ is transverse Killing.
\end{example}
\begin{example}
\label{ex:kirwan}
An important setting is that where one has a $\g\ltimes \f$-equivariant \emph{abstract moment map} (cf. \cite{GinzburgGuilleminKarshon}) 
\[ \mu \colon M \rightarrow \g^* \]
for the $\g$-action, meaning that $\pair{\mu}{\xi}$ is locally constant on $M_\xi$ for all $\xi \in \g$. Suppose $\g$ has an invariant inner product (this occurs if and only if $\g$ is the Lie algebra of a compact Lie group). Using the inner product to identify $\g^*\simeq \g$, $\mu$ becomes a taming map.
\end{example}

Let $D\colon C^\infty_{S,\f}\rightarrow C^\infty_{S,\f}$ be a $\g\ltimes \f$-equivariant transverse Dirac operator with symbol $\sigma_D \in C^\infty(M,\Sym^1(N\f)\otimes \End(S))^\f$. Let $\mu$ be a taming map. Let $U$ be a $\g\ltimes \f$-invariant open neighborhood of $Z_\mu=\{m \in M|a(\mu)_m=0\}$. The symbol $\sigma_{D,U,\mu} \in C^\infty(U,\Sym^{\le 1}(N\f|_U)\otimes \End(S|_U))^{\f_U}$ defined by
\[ \sigma_{D,U,\mu}(\xi)=\sigma_D(\xi-a(\mu)) \]
is $\g\ltimes \f|_U$-transversely elliptic; indeed as $a(\mu)$ is tangent to the transverse $\g$-orbit directions, the intersection $\supp(\sigma_{D,U,\mu})\cap N^*_\g \f|_U=Z_\mu$ (viewed as a subset of the $0$-section), and $Z_\mu/\ol{\f}|_U\simeq \pi_W(\pi_M^{-1}(Z_\mu))/SO_q$ is compact, being a closed subset of the compact space $W/SO_q$. 
\begin{corollary}
\label{c:generalloc}
$\index_\g^\f(D)=\index_\g^\f(\sigma_{D,U,\mu})$.
\end{corollary}
\begin{proof}
The index of the $\g\ltimes \f$-transversely elliptic symbol $\sigma_D(\xi-ta(\mu))$, $t \in [0,1]$ is independent of $t$. The index of the symbol $\sigma_D(\xi-a(\mu))$ equals the index of the $\g\ltimes \f|_U$-transversely elliptic symbol $\sigma_{D,U,\mu}$ by the excision property, Theorem \ref{t:excision}.
\end{proof}
\noindent In the case of a trivial foliation and $\mu$ as in Example \ref{ex:kirwan}, Corollary \ref{c:generalloc} is due to Paradan \cite{ParadanRiemannRoch}.

\begin{remark}
Interpreted on the Molino manifold, Corollary \ref{c:generalloc} localizes the index to a neighborhood of the vanishing locus of the section $\T(a(\mu))$ of $\T(N\f)$. As $\mu$ is $\f$-invariant, it determines a $\g$-valued function $\T(\mu)\colon W \rightarrow \g$, and using the $\g$-action on $W$, $\T(\mu)$ generates a vector field on $W$. However the vanishing locus of this vector field is usually different from the vanishing locus of $\T(a(\mu))$. 
\end{remark}

\subsection{Abelian localization}\label{sec:abel}
We consider the taming map $\mu=\beta$ in Example \ref{ex:beta} in more detail and derive a formula for the fixed-point contributions in Corollary \ref{c:generalloc}. The analogous result in the case of a trivial foliation goes back to work of Atiyah, Paradan, and Vergne (cf. \cite[Proposition 7.2]{WittenNonAbelian} and discussion therein). An analogous result for equivariant cohomology of Riemannian foliations was proved by Lin-Sjamaar \cite{SjamaarLinLocalization}, extending the Atiyah-Bott-Berline-Vergne localization formula to the foliated setting. Let $\beta \in \mf{z}$ be a fixed element in the centre of $\g$. Let $D$ be a transverse Dirac operator acting on $\f$-invariant sections of a transverse spinor bundle $S$.

The fixed point set $M_\beta=a(\beta)^{-1}(0)$ is a closed $\g\ltimes \f$-invariant submanifold of $M$. Let $\f_\beta=\f|_{M_\beta}$ denote the induced foliation of $M_\beta$. The metric $g$ on $N\f$ restricts to a metric $g_\beta$ on $N\f_\beta$. Let $NM_\beta=TM|_{M_\beta}/TM_\beta=N\f|_{M_\beta}/N\f_\beta$ denote the normal bundle to $M_\beta$ in $M$. Then one has an orthogonal decomposition
\[ N\f|_{M_\beta}=N\f_\beta\oplus NM_\beta.\]
Projecting the transverse Levi-Civita connection onto $NM_\beta$ yields a basic connection on $NM_\beta$.

By the linearization theorem of \cite{SjamaarLinLocalization}, $\beta$ induces an invertible linear bundle endomorphism $a_\beta \in C^\infty(M_\beta,\mf{so}(NM_\beta))$. Let
\[ J_\beta=\frac{a_\beta}{|a_\beta|} \in SO(NM_\beta) \]
be the corresponding complex structure on the vector bundle $NM_\beta$.  Note that $a_\beta$ (and hence also $J_\beta$) is $\f_\beta$-invariant, because $\g$ acts by transverse vector fields.  Also $a_\beta$ commutes with $J_\beta$, hence becomes a section of $\mf{u}(NM_\beta)=\mf{u}(NM_\beta,J_\beta)$.

Let $N^*M_\beta^{1,0}$, $N^*M_\beta^{0,1}$ denote the $+\i$, $-\i$ $J_\beta$-eigenbundles (subbundles of $N^*M_\beta\otimes \bC$).  The transverse metric $g$ and the complex structure $J_\beta$ determine Hermitian inner products. The complex exterior algebra bundle $\wedge N^*M_\beta^{0,1}$ is a spinor module for $\bC l(N^*M_\beta)$, and we obtain an induced Clifford module
\[ S_\beta=\Hom_{\Cl}(\wedge N^*M_\beta^{0,1},S|_{M_\beta}) \]
for $\bC l(N^*\f_\beta)$. By construction there is a canonical isomorphism
\[ S_\beta \wh{\otimes} \wedge\!N^*M_\beta^{0,1} \xrightarrow{\sim} S|_{M_\beta} \]
given by composition. Moreover $S_\beta$ has a canonical $\g$-equivariant basic connection induced from the connections on $\wedge N^*M_\beta^{0,1}$, $S|_{M_\beta}$. Let $D_\beta$ be the corresponding transverse Dirac operator for $M_\beta$ acting on invariant sections of $S_\beta$.

Let $\Sym(N^*M_\beta^{1,0})$ denote the complex symmetric algebra bundle.  The graded components $\Sym^k(N^*M_\beta^{1,0})$ are finite dimensional $\g$-equivariant basic complex vector bundles over $M_\beta$.  Let $\index_\g^{\f_\beta}(D_\beta \otimes \Sym(N^*M_\beta^{1,0})) \in R^{-\infty}(\g)$ denote the sum over all $k \ge 0$ of the index of $D_\beta$ twisted by $\Sym^k(N^*M_\beta^{1,0})$ (see the definition in Theorem \ref{t:indexell}, which is now applied to the Riemannian foliation $(M_\beta,\f_\beta,g_\beta)$).
\begin{theorem}[Abelian localization for $\index_\g^\f(D)$]
\label{thm:abel}
With hypotheses and notation as above, the index of $D$ is given by
\[ \index_\g^\f(D)=\index_\g^{\f_\beta}(D_\beta \otimes \Sym(N^*M_\beta^{1,0})).\]
\end{theorem}
\begin{proof}
By Corollary \ref{c:generalloc}, $\index_\g^\f(D)=\index_\g^\f(\sigma_{D,U,\mu})$. By the linearization and tubular neighborhood theorems of \cite{SjamaarLinLocalization}, we may as well assume $U$ is the total space of $NM_\beta$, and that the vector field $a(\beta)$ is the vector field on $U$ generated by the linear endomorphism $a_\beta$ of $NM_\beta$. The transverse metric induced on the total space of $NM_\beta$ by the tubular neighborhood embedding is straight-line homotopic to the metric induced from the metric on the fibres of the vector bundle with connection $NM_\beta \rightarrow M_\beta$ and the transverse metric on $M_\beta$. The homotopy of transverse metrics determines a homotopy of the transverse Clifford actions. A further homotopy replaces $a_\beta$ by its phase $J_\beta$ in the polar decomposition. The upshot of these modifications is that the proof is reduced to computing the index of the type of model symbol on the total space of $\V=NM_\beta$ considered in Corollary \ref{c:symalg}, from which the result follows.
\end{proof}

\subsection{Non-abelian localization}\label{s:nonabelloc}
We consider the taming map $\mu$ in Example \ref{ex:kirwan} in more detail. We will derive a formula for the contributions in an `abelianized' version of Corollary \ref{c:generalloc}. The analogous result in the case of a trivial foliation is due to Paradan \cite{ParadanRiemannRoch}.

Let $\g$ be a real Lie algebra admitting a positive definite invariant inner product (this occurs if and only if $\g$ is the Lie algebra of a compact Lie group). Then $\g=[\g,\g]\oplus \mf{z}$ where $\mf{z}$ is the center. Let $\t \subset \g$ be a maximal abelian subalgebra and fix a positive chamber $\t_+$. The set of finite dimensional irreducible complex representations of $\g$ is parametrized by $\wh{\g}=P_+\times \mf{z}^*$ where $P_+\subset \t_+^*\cap [\g,\g]^*$ is the set of dominant weights of $[\g,\g]$. The ring $R(\g)$ is the free abelian group on $\wh{\g}$ with product given by tensor product of representations; it is isomorphic to the character ring $R(G)$ of finite dimensional semisimple representations of the connected, simply connected integration $G$ of $\g$. 

Any element $\chi \in R(\g)\simeq R(G)$ is uniquely determined by its finitely supported multiplicity function 
\[ m_\chi \colon \wh{\g}=P_+\times \mf{z}^*\rightarrow \bZ.\]
Let $P$ be the weight lattice of $[\g,\g]$. Note that $m_\chi$ has a unique extension 
\[ \hat{m}_\chi\colon P\times \mf{z}^*\rightarrow \bZ \] 
which is alternating under the $\rho$-shifted action of the Weyl group. The Weyl character formula shows that this extension is the multiplicity function for $\chi|_T \cdot \wedge \n_- \in R(T)$, where $T=\exp(\t)\subset G$, $\n_-$ is the direct sum of the negative root spaces, and $\wedge \n_-$ is regarded as a $\bZ_2$-graded representation of $T$. Note that, unlike $G$, the group $T$ is usually not simply connected. Thus  remembering that $T$ acts on a vector space (rather than $\t$ only) is additional information about what infinitesimal weights can occur, and $R(T)$ identifies with a subspace of $R(\t)$.

Let 
\[ \mu_\g \colon M \rightarrow \g^*\simeq \g \] 
be an abstract moment map. Let 
\[ \mu=\pr_\t\circ \mu_\g \colon M\rightarrow \t, \qquad \mu_\perp\colon M \rightarrow \t^\perp \]
be the projections of $\mu_\g$ to $\t$, $\t^\perp$ respectively. Note that for any $\xi\in \t$, $\mu_\g^{-1}(\xi)=\mu^{-1}(\xi)\cap \mu_\perp^{-1}(0)$.

The localizing set for the taming map $\mu$ is
\[ Z_\mu=\bigcup_{\beta \in \B} M_\beta \cap \mu^{-1}(\beta), \qquad \B=\{\beta\in \t|M_\beta \cap \mu^{-1}(\beta)\ne \emptyset.\] 
Since $\mu$ induces a map $\mu_W \colon W \rightarrow \t$ and $W$ is compact, the image of $\mu$ in $\t$ is compact.

Suppose $\beta \ne 0$. The abstract moment map condition shows that the image of any component of $M_\beta$ that intersects $\mu^{-1}(\beta)$ non-trivially is contained in the affine hyperplane $\beta+\tn{ann}(\beta)$, where $\tn{ann}(\beta)\subset \t^*$ is the annihilator of $\beta$. Since the Molino manifold is compact, $M$ has finitely many $\t$ orbit type strata. It follows that finitely many distinct affine subspaces can appear in this way, and hence the set $\B$ is discrete.

Let $D$ be a $\g\ltimes \f$-equivariant transverse Dirac operator acting on $\f$-invariant sections of a transverse spinor bundle $S$. It has an equivariant index, $\index_\g^\f(D) \in R(\g)=R(G)$. Via the Weyl character formula, $\index_\g^\f(D) \in R(G)$ is determined by $\index_T^\f(D)\cdot \wedge \n_- \in R(T)$, and it will be convenient to study the latter instead. (Here the notation $\index_T^\f(D)$ simply means $\index_\t^\f(D)$, except we are remembering the additional information that the result lies in $R(T)\subset R(\t)$.) By Corollary \ref{c:generalloc},
\[ \index_T^\f(D)=\index_T^\f(\sigma_{D,U,\mu}) \]
where $U$ is a $\t_M\ltimes \f$-invariant open neighborhood of $Z_\mu$. We may choose $U=\sqcup_\beta U(\beta)$, where $U(\beta)$ is an open neighborhood of $M_\beta \cap \mu^{-1}(\beta)$. Let $\sigma_{D,U(\beta),\mu}$ be the restriction of the $a(\mu)$-deformed symbol to $U(\beta)$. Taking $U(\beta)$ smaller if necessary, it may be identified, via a $\t_M\ltimes \f$-equivariant tubular neighborhood embedding, with the normal bundle $\pi_\beta \colon NU(\beta)_\beta\rightarrow U(\beta)_\beta$ to the $\beta$-fixed point set $U(\beta)_\beta\subset M_\beta$. By a straight-line homotopy, the taming map $\mu|_{U(\beta)}$ can be replaced with the taming map $\beta+\pi_\beta^*\mu_\beta$, where $\mu_\beta=\mu|_{U(\beta)_\beta}-\beta$, in the deformation of the symbol $\sigma_D|_{U(\beta)}$ without changing the index. A further homotopy as in the previous section leads to a symbol which is a twisted product. Applying Corollary \ref{c:symalg} as in the proof of Theorem \ref{thm:abel}, results in the following.
\begin{theorem}
\label{t:nonabelloc1}
The $T$-equivariant index of $D$ is given by
\[ \index_T^\f(D)=\sum_{\beta \in \B}\index_T^{\f_\beta}(\sigma_{D_\beta,U(\beta)_\beta,\mu_\beta} \otimes \Sym(N^*U(\beta)^{1,0}_\beta)) \]
where $U(\beta)_\beta\subset M_\beta$ is an open neighborhood of $M_\beta\cap \mu^{-1}(\beta)$ in $M_\beta$, and $\sigma_{D_\beta,U(\beta)_\beta,\mu_\beta}$ is the induced $\t_M\ltimes \f$-transversely elliptic symbol on $U(\beta)_\beta$ obtained by deforming the symbol of $D_\beta$ with the taming map $\mu_\beta$.
\end{theorem}

After multiplying by $\wedge \n_-$ we obtain the formula
\[ \index_T^\f(D)\cdot \wedge \n_-=\sum_{\beta \in \B}\index_T^{\f_\beta}((\sigma_{D_\beta,U(\beta)_\beta,\mu_\beta}\hat{\boxtimes} 1) \otimes \Sym(N^{1,0}U(\beta)_\beta)),\]
where $\sigma_{D_\beta,U(\beta)_\beta,\mu_\beta}\hat{\boxtimes}1$ acts on $S_\beta \hat{\boxtimes}\wedge \n_-$. Note that $\n_-$ is the $-\i$-eigenspace of a complex structure on $\t^\perp$. Let $\sigma_{\n_-} \colon \t^\perp \rightarrow \End(\wedge \n_-)$ be the standard Clifford action. Deform the symbol further to
\[ \sigma_{D_\beta,U(\beta)_\beta,\mu_\beta}\hat{\boxtimes} 1+\mu_\perp^*\sigma_{\n_-}.\]
Since the first term acts trivially on $\wedge \n_-$ and the second term acts trivially on $S_\beta$, the support of the new symbol is $M_\beta\cap \mu^{-1}(\beta)\cap \mu_\perp^{-1}(0)=M_\beta\cap \mu_\g^{-1}(\beta)$. This yields the following refinement of Theorem \ref{t:nonabelloc1} for the product $\index_T^\f(D)\cdot \wedge \n_-$.
\begin{theorem}
\label{t:nonabelloc2}
Let $\B_\g$ be the subset of $\B$ consisting of $\beta$ such that $M_\beta\cap \mu_\g^{-1}(\beta)\ne \emptyset$. Then
\[ \index_T^\f(D) \cdot \wedge \n_-=\sum_{\beta \in \B_\g}\index_T^{\f_\beta}(\sigma_{D_\beta,U(\beta)_\beta,\mu_\beta} \otimes \Sym(N^*U(\beta)^{1,0}_\beta))\cdot \wedge \n_-. \]
\end{theorem}
This is a version of the non-abelian localization formula, as the contributions are labelled by the set $\B_\g$ associated to $\mu_\g$. Often $\B_\g$ is much smaller than $\B$. For example, in the case $\f$ is trivial and $M=G\cdot \xi$ is the coadjoint orbit of a regular element $\xi \in \t$, $\B_\g$ is the Weyl orbit of $\xi$, whereas $\B$ is much larger. 

\section{Quantization commutes with reduction}\label{s:qr0}
In this section we apply non-abelian localization (Theorem \ref{t:nonabelloc2}) to prove a $[Q,R]=0$ theorem for transversely symplectic Riemannian foliations. For further discussion and details regarding transverse Hamiltonian actions and reduction, see Appendix \ref{s:reductiontheory}.

\subsection{Quantization}\label{s:quantization}
Let $(M,\f)$ be a foliation of even codimension $q$. A \emph{transverse symplectic structure} is a closed $\f$-basic 2-form $\omega$ such that $\omega^{q/2}\ne 0$. Then $\omega$ induces a non-degenerate skew-symmetric bilinear form on the fibers of the normal bundle $N\f$, making the latter into a symplectic vector bundle. Suppose further that $N\f$ admits an $\omega$-compatible $\f$-invariant complex structure $J$. Then $g(-,-)=\omega(-,J-)$ is a Euclidean structure on the fibers of $N\f$, hence $(M,\f,g)$ becomes a Riemannian foliation. Any two of $(\omega,J,g)$ determines the third, and as in the case of a trivial foliation, this implies that any two choices of $J$ are connected by a homotopy. We will refer to the data $(M,\f,g,\omega,J)$ as a \emph{transversely symplectic Riemannian foliation}. If $(M,\f,g)$ is a complete Riemannian foliation, then $(M,\f,g,\omega,J)$ will be called a \emph{transversely symplectic complete Riemannian foliation}.

The transverse almost complex structure $J$ determines a $\Cl(N^*\f,g)$-module 
\[ (\wedge N^*\f^{0,1},c\colon \bC l(N^*\f,g)\rightarrow \End(\wedge N^*\f^{0,1})).\] 
The latter has a canonical basic Clifford connection, induced by the transverse Levi-Civita connection. Let $(L,\nabla^L)$ be a basic Hermitian line bundle. Twisting by $(L,\nabla^L)$ we obtain a transverse spinor module with basic Clifford connection
\[ S=\wedge N^*\f^{0,1}\otimes L.\]
This data determines a transverse Spin$_c$ Dirac operator
\[ D\colon C^\infty(M,S)^\f \rightarrow C^\infty(M,S)^\f.\]
In case $(M,\f)$ is equipped with a transverse isometric action of a Lie algebra $\g$,
\[ a \colon \g \rightarrow \scr{K}(M,\f,g), \]
and $(L,\nabla^L)$ is a $\g\ltimes \f$-equivariant basic Hermitian line bundle, then $S$ and $D$ become $\g\ltimes \f$-equivariant.
\begin{definition}
Let $(M,\f,g,\omega,J,a)$ be a transversely symplectic complete Riemannian foliation equipped with a transverse isometric $\g$-action. Let $(L,\nabla^L)$ be a $\g\ltimes \f$-equivariant basic Hermitian line bundle. Suppose $M/\ol{\f}$ is compact. We define $RR_\g(M,\f,L)$ (`Riemann-Roch') to be the $\g$-equivariant index:
\[ RR_\g(M,\f,L)=\index^\f_\g(D)\in R(\g).\]
As for the index, if $G'$ is a particular integration of $\g$ and $RR_\g(M,\f,L) \in R(G')\subset R(G)=R(\g)$, we use the notation $RR_{G'}(M,\f,L)$ when we wish to remember this additional information.
\end{definition}
An important special case is when $(L,\nabla^L)$ is a basic \emph{prequantum line bundle}, meaning that the first Chern form of $(L,\nabla^L)$ is $\omega$: $(\i/2\pi)(\nabla^L)^2=\omega$. A transverse $\g$ action on $M$ is called \emph{Hamiltonian} if there exists a $\g_M\ltimes\f$-equivariant moment map:
\[ \mu_\g \colon M \rightarrow \g^*, \qquad \iota(a(\xi))\omega=-\d\pair{\mu_\g}{\xi}.\]
The moment map equation implies that the transverse $\g$ action preserves $\omega$. If $(L,\nabla^L)$ is a basic Hermitian prequantum line bundle, then $\mu_\g$, $\nabla^L$ determine a lift of the transverse $\g$ action to $L$ via Kostant's formula:
\[ \xi \mapsto \nabla^L_{a(\xi)}+2\pi \i\pair{\mu_\g}{\xi}.\]
With this choice of lift, $RR_\g(M,\f,L)$ is our definition of the (equivariant) \emph{quantization} of $(M,\f,\omega,L)$.

\subsection{Weak equivalence and quantization}\label{s:weakequivquant}
Recall that a weak equivalence $(B,\pi_1,\pi_2)$ of foliations $(M_i,\f_i)$, $i=1,2$ induces an isomorphism $\pi_1^*N\f_1\simeq \pi_2^*N\f_2$. Thus in the case of transversely symplectic foliations, it makes sense to require that this isomorphism intertwine the pullbacks of the transverse symplectic forms $\omega_1,\omega_2$, and this is what we shall mean by a \emph{symplectic weak equivalence}. Similarly if $L_i$ is a prequantum line bundle for $(M_i,\f_i,\omega_i)$, $i=1,2$, it makes sense to require in addition that $B$ intertwines $L_1,L_2$. Then Proposition \ref{p:Moritaindex} immediately implies the following.
\begin{corollary}
\label{c:quantweakequiv}
The quantization of transversely symplectic complete Riemannian foliations defined in Section \ref{s:quantization} is invariant under equivariant complete weak equivalences that intertwine prequantum line bundles.
\end{corollary}

\subsection{Localization and the $[Q,R]=0$ theorem}
If $\g$ admits an invariant inner product, then we are in the situation of Section \ref{s:nonabelloc}. In particular $RR_\g(M,\f,L) \in R(\g)$ is fully determined by $RR_T(M,\f,L)\cdot \wedge \n_-\in R(T)$ (where recall $T \subset G$ is a maximal torus in the simply connected integration of $\g$), and Theorem \ref{t:nonabelloc2} gives the formula
\begin{equation}
\label{e:nonabelQ}
RR_T(M,\f,L) \cdot \wedge \n_-=\sum_{\beta \in \B_\g}\index^{\f_\beta}_T(\sigma_{D_\beta,U(\beta)_\beta,\mu_\beta} \otimes \Sym(N^*U(\beta)^{1,0}_\beta))\cdot \wedge \n_-. 
\end{equation}
In the $[Q,R]=0$ theorem the contribution of $\beta=0$ in \eqref{e:nonabelQ} is important. For this case $M_\beta=M$, $U(0)$ is a small neighborhood of $\mu^{-1}(0)$ in $M$, and the corresponding $\t_M\ltimes \f$-transversely elliptic symbol on $U(0)$ in \eqref{e:nonabelQ} is simply $\sigma_{D,U(0),\mu}(\xi)=\sigma_D(\xi-a(\mu))|_{U(0)}$.
\begin{lemma}
\label{l:GmultTmult}
Let $\g$ be a Lie algebra equipped with an invariant inner product. Let $(M,\f,\omega,J,g)$ be a transversely symplectic complete Riemannian foliation with compact Molino manifold. Let $(L,\nabla^L)$ be a prequantum line bundle with basic connection. Suppose $M$ is equipped with a Hamiltonian action by transverse Killing fields with moment map $\mu_\g \colon M \rightarrow \g^*$. Then
\[ RR_\g(M,\f,L)^\g=\Big(RR_T(M,\f,L)\cdot \wedge \n_-\Big)^T=\Big(\index^\f_T(\sigma_{D,U(0),\mu})\cdot \wedge \n_-\Big)^T.\]
\end{lemma}
\begin{proof}
The first equality is the Weyl character formula. For the second equality we must show that the terms of \eqref{e:nonabelQ} with $0\ne \beta \in \B_\g$ have trivial $T$-invariant part. The proof is similar to the case with trivial $\f$, cf. \cite{ParadanRiemannRoch} or \cite[Lemma 4.17]{LLSS2}. Fix $\beta$ and let $Y=U(\beta)_\beta$. It suffices to show that $\beta/\i$ acts with strictly positive eigenvalues on the $\bZ_2$-graded bundle
\[ \Hom_{\Cl}(\wedge N^*_{J_\beta}Y^{0,1},\wedge N^*\f^{0,1}|_Y)\otimes L|_Y\otimes \Sym(N^*_{J_\beta}Y^{1,0}))\boxtimes \wedge \n_-,\]
over $Y$. This is clear for $\Sym(N^*_{J_\beta}Y^{1,0})$ by the choice of $J_\beta$. By the Kostant condition, $\beta/\i$ acts on $L|_Y$ with eigenvalue $|\beta|^2>0$. We have
\[ \Hom_{\Cl}(\wedge N^{0,1}_{J_\beta}Y,\wedge N^{0,1}\f|_Y)=\wedge N^{0,1}\f_Y\hat{\otimes}\tn{det}(N^*_+Y^{1,0}) \]
where $N^*_+Y^{1,0}\subset N^*Y^{1,0}$ is the subbundle where the eigenvalue of $\beta/\i$ is positive. Note $\beta/\i$ acts trivially on $N^*\f_Y^{0,1}$. Thus we are left show that the eigenvalue for the action of $\beta/\i$ on $\tn{det}(N^*_+Y^{1,0})$ is greater than or equal to the worst negative contribution from $\wedge \n_-$:
\[ \inf_{I\subset \mf{R}_-} \sum_{\alpha \in I}\pair{\beta}{\alpha}.\]
Without loss of generality it suffices to consider the `worst case' where $\beta \in \t_+$ and $I=\mf{R}_-$ achieves the minimum. Let $p \in Y\cap \mu_\g^{-1}(\beta)$. Then $\g$-equivariance of $\mu_\g$ implies that $\g_\beta^\perp \hookrightarrow N_pY$ as a subset of the $\g$ transverse orbit directions. The phase in the polar decomposition of $\ad_\beta \in \End(\g_\beta^\perp)$ is a complex structure with $(\g_\beta^\perp)^{1,0}$ a sum of root spaces $\g_\alpha$ with $\pair{\alpha}{\beta}>0$. Under the orbit embedding, this complex structure is compatible with $\omega_p$ since for $0\ne \xi \in \g_\beta^\perp$ we have
\[ \omega_p(a(\xi),a(\ad_\beta \xi))=-\d\pair{\mu_\g}{\xi}(a(\ad_\beta \xi))=-\pair{\ad_\beta^2\xi}{\xi}=|\ad_\beta \xi|^2>0 \]
where we used the equivariance property of the moment map. Thus the $\pair{\alpha}{\beta}$, for $\alpha \in \mf{R}_+$ such that $\pair{\alpha}{\beta}>0$, occur in the list of eigenvalues of $\beta/\i$ on $N^*_+Y^{1,0}$.
\end{proof}

\begin{definition}
Let $(M,\omega,\f,\mu_\g)$ be a transverse Hamiltonian $\g$-space and suppose $0\in \g^*$ is a regular value of the moment map. The \emph{symplectic quotient} (at $0$) is the transversely symplectic manifold $(M_0=\mu_\g^{-1}(0),\f_0,\omega_0=\omega|_{M_0})$ where $\f_0$ is the codimension $q-\dim(\g)$ foliation of $M_0$ generated by $\f|_{M_0}$ and $\g$. If $L$ is a prequantum line bundle on $M$, then $L_0=L|_{M_0}$ is a prequantum line bundle on $M_0$.
\end{definition}
Note also that a transverse metric $g$ for $(M,\f)$ induces a transverse metric $g_0$ for $(M_0,\f_0)$, and that the completeness property is preserved (by the Hopf-Rinow theorem). For further discussion of symplectic quotients, see Appendix \ref{s:reductiontheory}.

The tangent map $T\mu_\g$ induces a $\g\ltimes \f$-equivariant isomorphism
\begin{equation} 
\label{e:tgtmap}
T\mu_\g|_{M_0} \colon NM_0 \rightarrow M_0 \times \g^*\simeq M_0\times \g.
\end{equation}
Let 
\[ c_{\g_\bC} \colon \Cl(\g\oplus \g)\rightarrow \wedge \g_\bC \]
be the Clifford action. Define a transverse spinor bundle $S_0$ for $M_0$ by
\[ S_0=\Hom_{\Cl(\g\oplus\g)}(\wedge \g_\bC,S|_{M_0}) \]
where $M_0\times(\g\oplus \g) \hookrightarrow N\f|_{M_0}$ as the $\g$-orbit and $NM_0$ directions. The index of the corresponding transverse Dirac operator $D_0$ on $M_0$ defines the quantization of the reduced space $RR(M_0,\f_0,L_0) \in \bZ$.

\begin{theorem}
\label{t:qr0}
Let $\g$ be a Lie algebra equipped with an invariant inner product. Let $(M,\f,\omega,J,g)$ be a transversely symplectic complete Riemannian foliation such that $M/\ol{\f}$ is compact. Suppose $M$ is equipped with a $\g\ltimes \f$-equivariant basic prequantum line bundle $(L,\nabla^L)$ and a Hamiltonian action by transverse Killing fields with moment map $\mu_\g \colon M \rightarrow \g^*$. Suppose $0$ is a regular value of $\mu_\g$ and let $(M_0,\f_0,\omega_0,L_0)$ be the reduced space and reduced prequantum line bundle. Then
\[ RR_\g(M,\f,L)^\g=RR(M_0,\f_0,L_0).\]
\end{theorem}
\begin{proof}
By Lemma \ref{l:GmultTmult}, $RR_\g(M,\f,L)^\g$ equals
\[ \Big(\index_T^\f(\sigma_{D,U(0),\mu})\cdot \wedge \n_-\Big)^T \]
where $\sigma_{D,U(0),\mu}$ is the deformed symbol on a neighborhood $U(0)$ of $\mu^{-1}(0)$ in $M$. By a homotopy we may replace $\sigma_{D,U(0),\mu}$ with $\sigma_{D,U(0),\mu_\g}$, the symbol deformed by $a(\mu_\g)$, without changing the index. By the Weyl character formula
\[ \Big(\index^\f_T(\sigma_{D,U(0),\mu_\g})\cdot \wedge \n_-\Big)^T=\index^\f_\g(\sigma_{D,U,\mu_\g})^G \]
where $U$ is a $G$-invariant open neighborhood of $\mu_\g^{-1}(0)$ in $M$. 

Precomposition of \eqref{e:tgtmap} with the inverse of a $\g\ltimes \f$-equivariant tubular neighborhood embedding shows that a neighborhood of $M_0$ in $M$ may be modelled as $M_0\times \g^*$. Up to homotopy, $\mu_\g$ can be replaced with its linearization along $M_0$, which is just projection onto the second factor of $M_0\times \g^*$. Extend $\f_0$ to a foliation $\ti{\f}_0$ of $M_0\times \g^*$ with leaves $F\times \{\xi\}$ where $F$ is a leaf of $\f_0$ and $\xi \in \g^*$. By Theorem \ref{t:infinitesimalactionfoliation},
\[ \index_\g^\f(\sigma_{D,U,\mu_\g})^\g=\index^{\ti{\f}_0}(\sigma_{D,U,\mu_\g}). \]
Over $M_0\times \g^*$, the spinor module splits into a tensor product
\[ S|_{M_0\times \g^*}\simeq S_0\hat{\boxtimes}(\g^*\times \wedge\g_\bC) \rightarrow M_0\times \g^*,\]
and the symbol $\sigma_{D,U,\mu_\g}$ is homotopic to the product
\[ \sigma_{D_0}\hat{\boxtimes}1+1\hat{\boxtimes}\sigma_{\g^*}, \]
where
\[ \sigma_{\g*}(\xi,\zeta)=c_{\g_\bC}(-\xi\oplus \zeta), \qquad \zeta \in T_\xi \g^*.\]
The $-\xi$ is the deformation by $a(\mu_\g)$, expressed in the local model. The symbol $\sigma_{\g*}$ is the Bott symbol for the vector space $\g^*$; it is elliptic (invertible except at $(0,0)\in T\g^*$) with index $1$. Thus by Theorem \ref{t:indextwistedproduct},
\[ \index^{\ti{\f}_0}(\sigma_{D,U,\mu_\g})=\index^{\f_0}(\sigma_{D_0}).\]
This completes the proof.
\end{proof}

\subsection{Some extensions of the main theorem}
Suppose that we are in the setting of Theorem \ref{t:qr0} (except that we omit the requirement that $0$ is regular value of $\mu_\g$). A version of the well-known `shifting trick' holds in this context, and leads to a formula for $RR_\g(M,\f,L)$. Recall that irreducible finite dimensional representations of $\g$ are parametrized by $\wh{\g}=P_+\times \mf{z}^*$, where $P_+$ is the set of dominant weights of $[\g,\g]$ and $\mf{z}$ is the center of $\g$. For $\lambda \in \wh{g}$, let $V_\lambda$ be the corresponding finite dimensional irreducible representation of $\g$. Note that $V_\lambda^*$ may be realized as the quantization of the coadjoint orbit $\O_{-\lambda}$ equipped with its standard Kirillov-Kostant-Souriau symplectic structure, and prequantum line bundle $G\times_{G_\lambda} \bC_{-\lambda}$. The zero level of the moment map for $M\times \O_{-\lambda}$ is identified with $\mu_\g^{-1}(\O_{\lambda})$ by projection onto the first factor. If $\lambda$ is a regular value of $\mu$, then $0$ is a regular value of the moment map for $M\times \O_{-\lambda}$. Applying Theorem \ref{t:qr0} to $M\times \O_{-\lambda}$ yields the following.
\begin{corollary}
\label{c:shift}
The multiplicity of $V_\lambda$ in $RR_\g(M,\f,L)$ equals the dimension of $RR(M_\lambda,\f_\lambda,L_\lambda)$ where $M_\lambda=\mu_\g^{-1}(\O_\lambda)$, and $\f_\lambda$ is the codimension $q-\dim(\g)$ foliation of $M_\lambda$ generated by $\f$ and the $\g$-action.
\end{corollary}
One can further extend the definitions to handle non-regular values as in \cite{MeinrenkenSjamaar} (for example by shift desingularization), and then the corresponding generalization of Theorem \ref{t:qr0} and Corollary \ref{c:shift} hold. This yields the formula
\[ RR_\g(M,\f,L)=\sum_{\lambda \in \wh{\g}} RR(M_\lambda,\f_\lambda,L_\lambda)V_\lambda. \]
In particular since the LHS is finite dimensional, $RR(M_\lambda,\f_\lambda,L_\lambda)$ vanishes for all but finitely many $\lambda$ (see also the examples in Section \ref{s:ex}).

There is a modest extension of Theorem \ref{t:qr0} involving reduction in stages. Suppose given data as in Theorem \ref{t:qr0}, except that we no longer assume $0$ is a regular value of $\mu_\g$. Let $\h \subset \g$ be an ideal. The moment map for the $\h$ action is $\mu_\h=\pr_{\h^*}\circ \mu_\g$. Suppose $0$ is a regular value of $\mu_\h$. Let $M_\h=\mu_\h^{-1}(0)=\mu_\g^{-1}(\tn{ann}(\h))$ where $\tn{ann}(\h)\subset \g^*$ denotes the annihilator of $\h$. Let $\f_\h$ be the foliation of $M_\h$ generated by $\f$ and $\h$. Let $\omega_\h$ denote the induced transverse symplectic form on $M_\h$. The transverse $\g$ action induces a transverse $\g/\h$ action on $(M_\h,\f_\h)$, with moment map $\mu_{\g/\h}=\mu_\g|_{M_\h}\colon M_\h \rightarrow \tn{ann}(\h)=(\g/\h)^*$. Moreover $L_\h=L|_{M_\h}$ is a prequantum line bundle with the restricted prequantum connection, and it becomes a basic line bundle upon using the connection to lift the enlarged foliation $\f_\h$ to $L_\h$.
\begin{theorem}
\label{t:qr0stages}
With assumptions and notation as above,
\[ RR_\g(M,\f,L)^\h=RR_{\g/\h}(M_\h,\f_\h,L_\h),\]
as representations of $\h$.
\end{theorem}
When $\h=0$, Theorem \ref{t:qr0stages} specializes to Theorem \ref{t:qr0}, and the proof of the general result is similar. An important point is that, since $\h\subset \g$ is an ideal, $\mu_\h$ is $\g$-equivariant, meaning that the non-abelian localization formula for the taming map $\mu_\h$ decomposes $RR_\g(M,\f,L)$ into $\g$-representations.

\section{Examples}\label{s:ex}
In this section we describe a number of examples of transversely symplectic Riemannian foliations and their quantizations.

\subsection{Reduction in stages}\label{s:reductionstages}
One source of examples is reduction in stages starting with an ordinary Hamiltonian manifold with trivial foliation; Theorem \ref{t:qr0stages} then provides a description of the quantization. We spell out this construction here.

Let $G$ be a compact connected Lie group. Let $(M,\omega,\mu_\g\colon M\rightarrow \g^*)$ be a compact Hamiltonian $G$-space. Choose a $G$-invariant compatible almost complex structure, which also determines a $G$-invariant Riemannian metric. Let $\h\subset \g$ be an ideal, and $H\subset G$ the connected immersed normal subgroup with Lie algebra $\h$. Let $\xi \in \mf{z}^*$ be central and such that $\pr_{\h^*}(\xi)$ is a regular value of the moment map $\mu_\h=\pr_{\h^*}\circ \mu\colon M\rightarrow \h^*$ for the $H$ action. Then 
\[ M_\h=\mu_\h^{-1}(\xi)=\mu_\g^{-1}(\xi+\tn{ann}(\h^*)) \]
is a smooth $H$-invariant submanifold. The action of $H$ on $M_\h$ is locally free, hence determines a foliation $\f_\h$ of $M_\h$. The symplectic, almost complex, and Riemannian structures on $M$ induce corresponding transverse structures for $(M_\h,\f_\h)$.
The $\g$ action induces a transverse $\g/\h$ action on $M_\h$, with moment map $\mu_{\g/\h}=\mu_\g|_M-\xi\colon M \rightarrow \tn{ann}(\h)=(\g/\h)^*$.

Let $M$ be equipped with a $G$-equivariant prequantum line bundle $(L,\nabla)$. The restriction $(L_\h=L|_{M_\h},\nabla^{L_\h}=\nabla|_{M_\h})$ is a prequantum line bundle for $M_\h$. The prequantum connection is used to lift the foliation to $L_\h$, and then $(L_\h,\nabla^{L_\h})$ is $\f_\h$-basic. The Kostant formula for $\mu_{\g/\h}$ determines the lift of the $\g/\h$ action to $L_\h$. Theorem \ref{t:qr0stages} specializes to the following.
\begin{corollary}
\label{c:qrstages}
With notation and hypotheses as above, the quantization of $M_\h$ is given by
\[ RR_{\g/\h}(M_\h,\f_\h,L_\h)=\big(RR_G(M,L)\otimes \bC_{-\xi}\big)^\h \in R(\g/\h).\]
\end{corollary}

Suppose $0 \in (\g/\h)^*$ is a regular value of $\mu_{\g/\h}$, i.e. $\xi$ is a regular value of $\mu_\g$. Then one can also apply Theorem \ref{t:qr0} to $M_\h$ itself to conclude that the $\g/\h$-invariant part of its quantization equals the quantization of the transversely symplectic manifold $\mu_\g^{-1}(\xi)$, where the foliation of the latter is given by the $\g$-orbits. By the Kostant condition, for $X \in \g$,
\[ \nabla^{L_\h}_X=\L_X-2\pi \i\pair{\xi}{X}.\]
If $X$ is in the kernel of the exponential map $\g\rightarrow G$, then $\L_X$ exponentiates to the identity, and hence the holonomy of the connection $\nabla^{L_\h}$ around a loop $t\mapsto \exp(tX)\cdot p$, $p \in M_\h$ is $e^{-2\pi \i\pair{\xi}{X}}$. Thus $L_\h$ has $\f_\h$-invariant sections if and only if $\xi$ is a character of $G$. In the latter case the quantization of $\mu_\g^{-1}(\xi)$ equals the quantization of the orbifold $\mu_\g^{-1}(\xi)/G$, which in turn equals the $G$-invariant part of $RR_G(M,\omega,L)\otimes \bC_{-\xi}$, by the $[Q,R]=0$ theorem (for a trivial foliation).

\subsection{Toric quasifolds}
Prato \cite{prato2001simple} generalized the Delzant construction of symplectic toric manifolds to arbitrary simple polytopes. The resulting objects are called \emph{symplectic toric quasifolds}. To fit with the setting of this article, we will restrict attention to the case that the toric quasi-fold comes presented as the leaf of a Riemannian foliation equipped with a transverse symplectic structure, and explain the result of applying our quantization procedure. The discussion here may be viewed as a special case of the preceding section, modulo mild non-compactness. 

Let $\t$ be a finite dimensional real vector space. Let $\Delta \subset \t^*$ be a simple polytope containing the origin. Let $v_1,\ldots,v_d \in \t$ be normal vectors to the facets and $c_1,\ldots,c_d$ constants such that
\[ \Delta=\{\xi \in \t^*\mid\pair{\xi}{v_i}\le c_i,i=1,\ldots,d\}.\]
Let
\[ \pi \colon \bR^d \rightarrow \t, \qquad \pi(e_i)=v_i \]
where $e_1,\ldots,e_d$ are the standard basis vectors. Let
\[ \iota \colon \h=\ker(\pi)\hookrightarrow \bR^d \]
be the kernel of $\pi$. Then we have an exact sequence
\[ 0 \rightarrow \h \xrightarrow{\iota} \bR^d \xrightarrow{\pi} \t \rightarrow 0 \]
with dual sequence
\[ 0 \rightarrow \t^*\xrightarrow{\pi^*}\bR^d \xrightarrow{\iota^*}\h^*\rightarrow 0.\]
Consider $\bC^d$ with its standard symplectic structure and action of $\mf{u}(1)^d$ with moment map
\[ \phi(z_1,\ldots,z_d)=-\frac{1}{2}(|z_1|^2,\ldots,|z_d|^2)+\ul{c} \]
where $\ul{c}=(c_1,\ldots,c_d)$. The image of $\phi$ is the shifted orthant in $(\bR^d)^*$ defined by the inequalities $x_1\le c_1,\ldots,x_d\le c_d$, and consequently $\tn{im}(\pi^*)\cap \tn{im}(\phi)=\pi^*(\Delta)$. The induced moment map for the $\h$ action is 
\[ \iota^*\circ\phi \colon \bC^d \rightarrow \h^*. \]
The fiber
\[ M=(\iota^*\circ \phi)^{-1}(0)=\phi^{-1}(\ker(\iota^*))=\phi^{-1}(\tn{im}(\pi^*))=\phi^{-1}(\pi^*(\Delta)) \]
is a compact manifold. The $\h$ action on $M$ is infinitesimally free, hence determines a foliation $\f$ of $M$. The symplectic, Riemannian and complex structures on $\bC^d$ induce corresponding transverse structures $(\omega,g,J)$ for $(M,\f)$. There is a residual transverse action of $\mf{u}(1)^d/\h=\t$ on $(M,\f)$ with moment map $\mu=(\pi^*)^{-1}\circ \phi$ and moment map image $\Delta$.

The $\mf{u}(1)^d$-equivariant quantization of $\bC^d$ with prequantum line bundle $\ul{\bC}_{\ul{c}}$ (corresponding to the moment map $\phi$) is the infinite dimensional representation $\Sym((\bC^d)^*)\otimes \bC_{\ul{c}}$ with weights $\{\ul{c}-\ul{b}\mid \ul{b}\in (\bZ_{\ge 0})^d\}$ (technically $\bC^d$, being non-compact, is not allowed in our quantization scheme; however this may be remedied using symplectic cutting to replace $\bC^d$ with a compact symplectic manifold containing the fiber $M$). Let $L=\ul{\bC}_{\ul{c}}|_M$ be the induced prequantum line bundle on $M$. By Corollary \ref{c:qrstages}, the quantization of $(M,\f,\omega,L)$ is the $\h$-fixed subspace; the corresponding weights are those lying in $\ker(\iota^*)$:
\[ \{\ul{c}-\ul{b}\mid\ul{b}\in (\bZ_{\ge 0})^d\}\cap \ker(\iota^*)=\pi^*(\Delta)\cap (\ul{c}+\bZ^d).\]
Since $\ker(\iota^*)$ is identified with $\t^*$, these may also be thought of as the weights of the $\t$ action on the quantization of $(M,\f,\omega,L)$ (and the multiplicities are all equal to $1$). The following summarizes these results.
\begin{corollary}
The quantization of $(M,\f,\omega,L)$ is 
\[ \big(\Sym((\bC^d)^*)\otimes \bC_{\ul{c}}\big)^\h \in R(\t). \]
The dimension of the quantization is equal to $\#((\ul{c}+\bZ^d)\cap \pi^*(\Delta))$.
\end{corollary}  
The quantization carries a representation of $U(1)^d$ such that the immersed subgroup $\exp(\h)$ acts trivially; one could interpret this as a representation of the quasi-torus $U(1)^d/H$.

We may also apply Theorem \ref{t:qr0} to $(M,\f,\omega,L)$ itself. The reduced space at a point $\xi \in \Delta$ is $M_\xi=\phi^{-1}(\ul{a})$, $\ul{a}=\pi^*(\xi)$, with the foliation $\f_\xi$ given by the $\bR^d$-orbits, and $(M_\xi,\f_\xi)$ is weakly equivalent to a point. The induced prequantum line bundle $L_\xi$ is topologically trivial, but may have a non-trivial connection $\nabla^{L_\xi}$, and hence the $\nabla^{L_\xi}$-lifted foliation may have non-trivial holonomy. Since $M_\xi$ is a single $\bR^d$-orbit, the connection is determined by the Kostant condition: 
\[ \nabla^{L_\xi}_X=\L_X-2\pi \i\pair{\ul{a}}{X}, X \in \bR^d.\]
In this expression, $\L_X=2\pi \i\pair{\ul{c}}{X}$ is the $\mf{u}(1)^d$ action on the total space of the prequantum line bundle $\ul{\bC}_{\ul{c}}$ of $\bC^d$. Therefore the holonomy of $\nabla^{L_\xi}$ for the loop corresponding to $X \in \bZ^d$ is $e^{2\pi \i\pair{\ul{c}-\ul{a}}{X}}$. It follows that the $\nabla^{L_\xi}$-lifted foliation on $L_\xi$ does not have trivial holonomy unless $\ul{a} \in (\ul{c}+\bZ^d)$. Thus the prequantum line bundle admits a foliation invariant section if and only if $\pi^*(\xi) \in (\ul{c}+\bZ^d)$. (This is a $\ul{c}$-shifted version of the Bohr-Sommerfeld condition, to which it reduces when $\ul{c} \in \bZ^d$.) Applying Theorem \ref{t:qr0} confirms what we already noted above: the dimension of the quantization of $(M,\f,\omega,L)$ is equal to $\#((\ul{c}+\bZ^d)\cap \pi^*(\Delta))$.


\subsection{Locally free isometric actions and K-contact manifolds}
Let $(M,\f,g)$ be the Riemannian foliation considered in Examples \ref{ex:freeisometric}, \ref{ex:freeisometric2}, \ref{ex:freeisometric3}. Assume $(M,g_M)$ is complete and that $\ol{H}\subset \tn{Isom}(M,g_M)$ acts cocompactly on $M$. Let $\omega$ be an $H$-invariant compatible transverse symplectic form. Suppose we are given an isometric Hamiltonian action of a compact connected Lie group $G$, such that the actions of $G,H$ commute. Assuming $0\in \g^*$ is a regular value of the moment map, the $0$-fiber $M_0$ inherits a locally free action of $G\times H$, whose orbits coincide with the Riemannian foliation $\f_0$ of $M_0$. Let $L$ be a $G\times H$-equivariant basic prequantum line bundle for $\omega$, and let $D$ be the corresponding transverse Dirac operator on $S=\wedge N^*\f^{0,1}\otimes L$. Then $RR_G(M,\f,L) \in R(G)$ and $RR(M_0,\f_0,L_0)$ are defined, and by Theorem \ref{t:qr0},
\begin{equation} 
\label{e:qr0freeiso}
RR_G(M,\f,L)^G=RR(M_0,\f_0,L_0).
\end{equation} 
By Example \ref{ex:freeisometric3}, $RR_G(M,\f,L)$ coincides with the $G$-equivariant index of the first-order differential operator $D_M$ (determined by $g_M$) restricted to the space of smooth $H$-invariant sections of $S$. Likewise $RR(M_0,\f_0,L_0)$ coincides with the index of a first-order differential operator $D_{M_0}$ restricted to the space of smooth $G\times H$-invariant sections of $S_0$. Thus \eqref{e:qr0freeiso} reduces to a $[Q,R]=0$ result (for $H$ or $\ol{H}$-invariant sections) for the differential operators $D_M$, $D_{M_0}$ on $M$, $M_0$ respectively. \ignore{More generally, \eqref{e:qr0freeiso} can be twisted by any 1-dimensional character $\alpha \in \Hom(\ol{H},U_1)$ of $\ol{H}$ as in Example \ref{ex:freeisometric4}, and \eqref{e:qr0freeiso} generalizes to
\begin{equation}
\label{e:qr0freeiso2}
RR_G(M,\f,L\otimes \bL_\alpha)^G=RR(M_0,\f_0,(L\otimes \bL_{\alpha})_0).
\end{equation}
This translates into a $[Q,R]=0$ result for $D_M$, $D_{M_0}$ acting on the $\alpha$-isotypical components of the spaces of smooth sections of $S$, $S_0$ respectively.}

Let $(M,\Pi,\theta)$ be a compact contact manifold with contact distribution $\Pi$ and contact form $\theta$. Then $\d\theta$ is a transverse symplectic structure for the codimension $1$ foliation $\f$ of $M$ generated by the Reeb vector field. If $\f$ is Riemannian for some choice of transverse metric $g$, then $(M,\Pi,\theta)$ is called a K-contact manifold. Identify $N\f\simeq \Pi=\ker(\theta)$ and define a bundle-like metric $g_M=\theta^2+g$. The Reeb vector field is Killing for $g_M$. Hence we are in a special case of Example \ref{ex:freeisometric}, with the $H=\bR$ action generated by the flow of the Reeb vector field.

Let $G$ be a compact connected Lie group acting isometrically on $M$ and preserving $\theta$. Let $z \in \mf{z}^*$ be a central element and define $\mu \colon M \rightarrow \g^*$ by
\begin{equation} 
\label{e:contactmu}
\iota(X_M)\theta+\pair{z}{X}=\pair{\mu}{X}.
\end{equation}
By Cartan's formula $\iota(X_M)\d \theta=-\d\pair{\mu}{X}$, hence $\mu$ is a moment map for the transverse symplectic structure. Assume for simplicity that $G$ acts freely on $\mu^{-1}(0)$. Then the quotient $M_{0,G}=\mu^{-1}(0)/G$ is again a K-contact manifold, with contact form $\theta_{0,G}$ determined by the property that its pullback to $\mu^{-1}(0)$ agrees with the restriction of $\theta$.

Given a $G$-equivariant prequantum line bundle $L$ for $\d \theta$ (satisfying the Kostant condition for \eqref{e:contactmu}), we obtain a $G$-equivariant quantization that we denote $RR_G(M,\theta,L)$. We also obtain a prequantum line bundle $L_{0,G}=L|_{\mu^{-1}(0)}/G$ on $M_{0,G}$. Equation \eqref{e:qr0freeiso} implies the following. 
\begin{theorem}
\label{t:qr0Kcontact}
$RR_G(M,\theta,L)^G=RR(M_{0,G},\theta_{0,G},L_{0,G})$.
\end{theorem}
\ignore{
Let $D_M$, $D_{M_0}$ be the $G\times \ol{H}$-equivariant first order differential operators determined by $g_M$ that lift $D,D_0$ respectively. Then summing over all $\alpha$, Theorem \ref{t:qr0Kcontact} implies that 
\begin{equation}
\label{e:qr0Kcontact2}
\index_G(D_M)^G=\index(D_{M_0})
\end{equation} 
holds in $R^{-\infty}(\ol{H})$, where the indices on both sides are defined by formally summing over all $\ol{H}$-isotypical components.}
The class of K-contact manifolds includes as special cases (i) Sasakian manifolds, (ii) the total space of a prequantum circle bundle over a symplectic manifold. Quantization of Sasakian manifolds was studied in \cite{hsiao2019geometricCR} (as a special case of a more general quantization scheme for suitable CR manifolds), and we expect that Theorem \ref{t:qr0Kcontact} should be related to one of their results via a foliated version of the Kodaira vanishing principle. Quantization of contact manifolds was also studied in \cite{fitzpatrick2011geometric}, and we expect a similar relation to hold.

The case of prequantum circle bundle $P \rightarrow M$ is already covered by the discussion in Section \ref{s:reductionstages} (modulo mild non-compactness), since $P$ may be realized as the reduced space at level $1$ of a Hamiltonian $S^1$ manifold $N=\tn{tot}(L^*)-0_M$ (the symplectization of the contact manifold $P$), where $L=P\times_{S^1}\bC$ is the prequantum line bundle associated to $P$. 

\subsection{Haefliger suspensions}
Let $(M,\f,g)$ be the Riemannian foliation from Examples \ref{ex:Haefliger}, \ref{ex:Haefliger2}, \ref{ex:Haefliger3}. Assume $(N,g_N)$ is complete and that the action of $\ol{\Gamma}$ on $N$ is cocompact. A $\Gamma$-equivariant symplectic form $\omega_N$ (resp. compatible almost complex structure $J_N$, resp. $\Gamma$-equivariant prequantum line bundle $L_N$) on $N$ induces a transverse symplectic form $\omega$ (resp. compatible transverse almost complex structure $J$, resp. transverse prequantum line bundle $L$) on $M$. Then $RR(M,\f,L)$ is defined and equals the index $RR(N/\Gamma,L_N)$ of the corresponding Dirac operator $D_N$ restricted to smooth $\Gamma$-invariant sections of $\wedge T^*N^{0,1}\otimes L_N$ (see the discussion of Example \ref{ex:Haefliger3}). Let $G$ be a compact connected Lie group and suppose we are given an isometric Hamiltonian action of $G$ on $N$, such that the actions of $G$, $\Gamma$ commute and $L_N$ is $G\times \Gamma$-equivariant. Let $RR_G(N/\Gamma,L_N)\in R(G)$ be the equivariant index of the Dirac operator $D_N$ restricted to smooth $\Gamma$-invariant sections. For simplicity assume $G$ acts freely on the $0$ level of the moment map, and let $N_0$ be the symplectic quotient of $N$ in the usual sense. Let $L_{N,0}$ be the induced prequantum line bundle $L_{N,0}$ on $N_0$ and let $RR(N_0/\Gamma,L_{N,0})$ be the index of the corresponding Dirac operator $D_{N_0}$ restricted to smooth $\Gamma$-invariant sections. Theorem \ref{t:qr0} reduces to the equation
\begin{equation} 
\label{e:qr0Haefliger}
RR_G(N/\Gamma,L_N)^G=RR(N_0/\Gamma,L_{N,0}).
\end{equation}
If $\Gamma$ acts freely and properly on $N$, then \eqref{e:qr0Haefliger} specializes to the classical $[Q,R]=0$ theorem for the compact symplectic manifold $N/\Gamma$.

\appendix

\section{Transverse metrics, completeness and the Molino package}\label{s:completion}
In this appendix we discuss some simple sufficient conditions for a Riemannian foliation to have the `Molino package', as well as the dependence of the Molino manifold on the choice of transverse metric.

Let $(M,\f,g)$ be a Riemannian foliation with transverse orthonormal frame bundle $P$ and lifted foliation $\f_P$. Recall that $(P,\f_P)$ is transversely parallelizable. In \cite[Section 4.5]{MolinoBook} Molino isolated the following additional assumption:
\[ (\star) \qquad N\f_P \text{ is generated by the images of global complete foliate vector fields.}  \]
For short we shall refer to a Riemannian foliation $(M,\f,g)$ satisfying $(\star)$ as a \emph{Molino foliation}. Molino foliations have the Molino package (see Section \ref{s:MolinoOverview}), in particular: (i) existence of a Molino manifold $W=P/\ol{\f}_P$ and generalized Lie groupoid morphism $\tn{Hol}(M,\f) \dashrightarrow O_q\ltimes W$, (ii) a locally constant Molino centralizer sheaf $\scr{C}_M$, (iii) leaf closure (singular) foliation $\ol{\f}$ with Hausdorff quotient $M/\ol{\f}$, (iv) transfer functor for basic vector bundles. Items (i), (ii), (iii) and related results are due to Molino cf. \cite[Section 4.5]{MolinoBook}; item (iv) is due to El Kacimi \cite{Kacimi}.
\begin{proposition}
\label{p:completeMolino}
Complete Riemannian foliations are Molino foliations. In particular any Riemannian foliation on a compact manifold is a Molino foliation.
\end{proposition}
\begin{proof}
Let $\pi_M\colon P \rightarrow M$ be the orthonormal frame bundle of $N\f$ with the lifted foliation $\f_P$. The pullback $\pi_M^*N\f$ has a tautological frame. The transverse Levi-Civita connection for $g$ determines a horizontal embedding $\pi_M^*N\f \hookrightarrow N\f_P$. A choice of basis for $\mf{o}_q$ together with the framing of the horizontal subbundle yields a framing of $N\f_{P}$. 

The pullback of $g_M$, together with a choice of scalar product on $\mf{o}_q$ yields a complete bundle-like metric $g_{P}$ for $(P,\f_{P})$. The metric in turn determines an isomorphism $N\f_{P}\simeq T\f_{P}^\perp$. Under this isomorphism, the frame of $N\f_{P}$ maps to a set of $q+\frac{1}{2}q(q-1)$ global foliate vector fields whose images generate $N\f_{P}$. The $\mf{o}_q$ vector fields in the global frame are complete. The horizontal vector fields in the global frame are also complete. Indeed suppose $X$ is one of the horizontal vector fields from the global frame and let $p \in P$. The maximal integral curve $\gamma_{X,p}=\gamma_{X,p}(t)$ of $X$ through the point $p$ corresponds to parallel transport of the frame $p$ along the geodesic through $\pi_M(p)$ with initial velocity $T\pi_M(X_p)$ orthogonal to the leaf through $\pi_M(p)$. Since $(M,g_M)$ is a complete Riemannian manifold, there is no obstruction to extending this geodesic, and hence $\gamma_{X,p}(t)$ is defined for all $t \in \bR$. Thus we have shown that there are global complete foliate vector fields on $P$ generating $N\f_{P}$; in other words $(M,\f,g)$ is a Molino foliation.
\end{proof}

\begin{proposition}
\label{p:metricdependence}
Let $(M,\f,g)$ be a Molino foliation with Molino manifold $W$. Let $g'$ be a transverse metric for $(M,\f)$. Then $(M,\f,g')$ is also a Molino foliation and there is a canonical $O_q$-equivariant diffeomorphism from $W$ to the Molino manifold $W'$ of $(M,\f,g')$.
\end{proposition}
\begin{proof}
Let $\pi \colon \P \rightarrow M$ be the full frame bundle of $N\f$, with the lifted foliation $\f_\P$. The orthonormal frame bundles $\pi_M'\colon P'\rightarrow M$ and $\pi_M\colon P \rightarrow M$ of $(N\f,g')$, $(N\f,g)$ respectively, are foliated subbundles of $\P$. Consider the composition
\[ N\f \xrightarrow{g'} N^*\f \xrightarrow{g^{-1}} N\f.\]
If $v \ne 0$ then
\[ g(g^{-1}g'(v),v)=g'(v,v)>0,\]
which shows that $g^{-1}g'$ is positive and hence $h=(g^{-1}g')^{1/2} \in \Gamma(GL(N\f)^\f)$ is well-defined, and yields a smooth foliated gauge transformation of $\P$. By construction
\[ g'(v,w)=g(hv,hw),\]
therefore $h$ maps $g'$ orthonormal frames to $g$ orthonormal frames. The inverse $h^{-1}$ restricts to a smooth foliate $O_q$-equivariant diffeomorphism
\[ h^{-1}\colon P \rightarrow P'. \]
The tangent map $Th^{-1}$ sends complete global foliate vector fields on $P$ to complete global foliate vector fields on $P'$. Thus $N\f_{P'}$ is generated by complete global foliate vector fields as well, and $(M,\f,g')$ is also a Molino foliation. By continuity $h^{-1}$ maps leaf closures in $P$ to leaf closures in $P'$, hence descends to an $O_q$-equivariant diffeomorphism from $W$ to $W'=P'/\ol{\f}_{P'}$.
\end{proof}

\begin{remark}
\label{r:canonicalisoequivariance}
Suppose a transverse $\g$ action is given on $(M,\f)$ which is isometric for both $g,g'$. Since the diffeomorphism of Molino manifolds in Proposition \ref{p:metricdependence} is canonically determined by $g,g'$, it will automatically be $\g$-equivariant.
\end{remark}

\begin{proposition}
\label{p:openMolino}
Let $(M,\f,g)$ be a complete Riemannian foliation, and let $M'\subset M$ be an $\f$-saturated open subset with restricted foliation $\f'$ and restricted transverse metric $g'$. Then $(M',\f')$ admits a complete bundle-like metric $\ti{g}'_{M'}$. Consequently $(M',\f',g')$, $(M',\f',\ti{g}')$ are Molino foliations with Molino manifolds $W'$, $\ti{W}'$ respectively, and there are canonical $O_q$-equivariant diffeomorphisms $\pi_W(\pi_M^{-1}(M'))\simeq W'\simeq \ti{W}'$.
\end{proposition}
\begin{proof}
Let $g_M$ be a complete bundle-like metric for $g$, and let $g'_{M'}$ be its restriction to $M'$. To construct a complete bundle-like metric for $M'$, we will adapt an idea from \cite{gordon1973analytical}.

Notice first that the complement $M\backslash M'$ is closed and $\f$-saturated, hence is a union of leaf closures. Thus $M'$ is itself a union of leaf closures. Set
\[ W'=\pi_W(\pi_P^{-1}(M')). \]
Then $W'$ is an $O_q$-invariant open subset of $W$. The homeomorphism $M/\ol{\f}\simeq W/O_q$ restricts to a homeomorphism $M'/\ol{\f}'\simeq W'/O_q$.

Let $r_{W'}\colon W' \rightarrow [0,\infty)$ be any smooth $O_q$-invariant proper function, and let $r \colon M' \rightarrow [0,\infty)$ be the corresponding $\f$-invariant smooth function on $M'$. Note that $r_{W'}$ descends to a continuous map $\bar{r}\colon W'/O_q\simeq M'/\ol{\f}'\rightarrow [0,\infty)$, and properness means that $\bar{r}$ extends continuously to a map (also denoted $\bar{r}$) from the one-point compactification of $M'/\ol{\f}$ to $[0,\infty]$. Pulling back to $M$, we deduce that $r$ extends to an $\f$-invariant continuous map $r\colon \ol{M'}\rightarrow [0,\infty]$ such that $r(\partial \ol{M'})=\{\infty\}$.

Since $r,g'$ are $\f$-invariant, the tensor
\[ \ti{g}'=g'+\tn{Hessian}_{g'}(r) \] 
is a transverse metric on $M'$. We claim that 
\[ \ti{g}'_{M'}:=g'_{M'}+\tn{Hessian}_{g'}(r) \] 
is a complete bundle-like metric for $g'_r$. Let $\ti{d}'$, $d'$, $d$ be the distance functions for $\ti{g}'_{M'}$, $g'_{M'}$, $g_M$ respectively. It suffices to show that any Cauchy sequence $p_n$ for $(M',\ti{d}')$ converges. Since $\ti{d}'\ge d'\ge d$, $p_n$ is a Cauchy sequence for $d$, and therefore converges to a point $p \in M$ by completeness of $d$. Thus $p_n$ is a sequence of points in $M'$ that $d$-converges to a point $p \in \ol{M'}$, and it suffices to show $p$ belongs to $M'$.

Identify $M'$ with the graph of the smooth map $r \colon M' \rightarrow \bR$. Then $\ti{g}'_{M'}$ is the metric on the graph induced by the product metric $g'_{M'}\times g_\bR$ on $M'\times \bR$, where $g_\bR$ is the standard metric on $\bR$. Thus
\[ \ti{d}'(p_n,p_m)\ge (d'\times d_\bR)((p_n,r(p_n)),(p_m,r(p_m)))\ge d_\bR(r(p_n),r(p_m))=|r(p_n)-r(p_m)|, \]
and so $r(p_n)$ is a Cauchy sequence in $\bR$. Thus $r(p_n)$ converges to some $r_\infty \in \bR$, and in particular stays finite as $n \rightarrow \infty$. It follows that $p \notin \partial \ol{M'}$, and hence $p \in M'$. This completes the proof of the first claim in the statement of the proposition. The second claim then follows immediately from Propositions \ref{p:completeMolino}, \ref{p:metricdependence}.
\end{proof}
\begin{remark}
\label{r:openMolinoequivariant}
Suppose a transverse isometric action of $\g$ on $(M,\f,g)$ is given, and that $M'$ is $\g\ltimes \f$-saturated. Then $W'=\pi_W(\pi_M^{-1}(M'))$ is $\g$-saturated. Since $W$ is complete, the isometric $\g$-action on $W$ integrates to an isometric action of the simply connected Lie group $G$ integrating $\g$, and $W'$ is $G$-invariant. Let $\ol{\rho(G)}$ be the closure of the image of $G$ in the isometry group of $W$, which we assume is compact. The complement $W \backslash W'$ is a closed $G$-invariant subset, hence $\ol{\rho(G)}$-invariant. Thus $W'$ is $\ol{\rho(G)}$-invariant. Averaging the function $r$ in the proof of Proposition \ref{p:openMolino} over the $\ol{\rho(G)}$ orbits guarantees that the $\g$ action is isometric for $\ti{g}'$.
\end{remark}

\begin{remark}
One can show, similarly, that open $\f$-saturated subsets of Molino foliations are again Molino foliations.
\end{remark}


\newcommand\eu{\mathfrak}
\newcommand\lie{\mathfrak}

\newcommand{\liet}{\lie{t}} 
\newcommand{\liek}{\lie{k}}
\newcommand{\liel}{\lie{l}}
\newcommand{\z}{\lie{z}}

\newcommand\bb{\mathbf}

\newcommand\ca{\mathscr}

%
%
\newcommand\qu[1][\kern.3ex]{/\kern-.7ex/_{\kern-.4ex#1}}
\newcommand\bigqu[1][\,\,]{\big/\kern-.85ex\big/_{\!\!#1}}

\newcommand\powl{[\kern-.3ex[} \newcommand\powr{]\kern-.3ex]}
\newcommand\bigpowl{\bigl[\kern-.6ex\bigl[}
    \newcommand\bigpowr{\bigr]\kern-.6ex\bigr]}

\newcommand\restrict{\mathop{|}}
\newcommand\bigrestrict{\mathop{\big|}}


\newcommand\inj{\hookrightarrow}
\newcommand\sur{\mathrel{\to\kern-1.8ex\to}}
\newcommand\iso{\mathrel{\overset{\cong}{\to}}}

\newcommand\longto{\longrightarrow} 
\newcommand\To{\Rightarrow}
\newcommand\Tto{\Rrightarrow} 
\newcommand\Longto{\Longrightarrow}
\newcommand\longhookrightarrow{\lhook\joinrel\longrightarrow}
\newcommand\longhookleftarrow{\longleftarrow\joinrel\rhook}
\newcommand\longinj{\longhookrightarrow}
\newcommand\longsur{\mathrel{\longrightarrow\kern-1.8ex\to}}
\newcommand\longiso{\mathrel{\overset{\cong}{\longrightarrow}}}
\newcommand\longleftiso{\mathrel{\overset{\cong}{\longleftarrow}}}


\newcommand\eps{\varepsilon}
\newcommand\bu{{\scriptscriptstyle\bullet}}

\newcommand\cc{{\rm c}} 
\newcommand\cld{{\rm cld}} 
\newcommand\cv{{\rm cv}}
\newcommand\aff{{\rm aff}}


\newcommand\coarse{{\mathrm{coarse}}} 
\newcommand\bas{{\mathrm{bas}}}
\newcommand\loc{{\mathrm{loc}}}
\newcommand\mult{{\mathrm{mult}}}
\newcommand\rig{{\mathrm{rig}}} 
\newcommand\can{{\mathrm{can}}}
\newcommand\red{{\mathrm{red}}}





\section{Reduction theory for transverse symplectic structures}\label{s:reductiontheory}

In this appendix we discuss moment maps and reduction theory for
transversely symplectic foliations. The theory works without any
Riemannian hypothesis, although the specialization to the Riemannian
case is all that is used in the main text.

What we want is a procedure which takes as input a transversely
symplectic foliated manifold with symmetries and as output gives a
transversely symplectic foliated manifold from which the symmetries
have been ``removed''.

Ordinary symplectic reduction proceeds in two steps: take a fibre of
the moment map, and then take its quotient by the symmetry group.  The
foliated case requires a slight change in our point of view.  We will
modify, or rather reinterpret, the second step: instead of dividing
the fibre, we enlarge the foliation of the fibre by incorporating the
group orbit directions into it.  This has the effect of lowering the
codimension of the foliation, so in some sense this amounts to
``reduction in the transverse direction''.  It turns out that this
procedure is applicable not just to Hamiltonian Lie group actions, but
also to transverse Hamiltonian Lie algebra actions (as defined in
Sections \ref{s:tact} and \ref{s:quantization}). The enlarged foliation is transversely symplectic and has codimension equal to the codimension
of the original foliation minus twice the dimension of the Lie
algebra.

One can reduce at a value of the moment map
(Proposition~\ref{proposition;reduction-value}) or at a coadjoint
orbit (Proposition~\ref{proposition;reduction-orbit}).  For
Hamiltonian group actions the results are equivalent.  A mild surprise
is that for Hamiltonian transverse Lie algebra actions actions this is
not necessarily the case: there is an \'etale morphism from the
reduction at value $\alpha$ to the reduced space at the coadjoint
orbit $\ca{O}_\alpha$
(Proposition~\ref{proposition;reduction-compare}).  This morphism is a
weak equivalence under additional assumptions
(Proposition~\ref{proposition;reduction-compare-molino}), which are
satisfied in the context of our main theorem (Theorem~\ref{t:qr0}).

\subsection{Transverse Hamiltonian Lie algebra actions}
\label{subsection;transverse}

Let $(M,\f,\omega)$ be a foliated manifold equipped with a transverse
symplectic structure, i.e.\ a closed $2$-form $\omega$ of constant
rank whose null foliation is equal to $\F$.  Then $\omega$ is
$\f$-basic in the sense that $\iota(X)\omega=\ca{L}_X\omega=0$ for all
$X\in\lie{X}(\f)$.

On a symplectic manifold every smooth function determines a unique
Hamiltonian vector field.  This fails in the transverse symplectic
setting.  Suppose a function $f\in C^\infty(M)$ and a vector field
$X\in\lie{X}(M)$ satisfy the Hamiltonian relationship
$df=\iota(X)\omega$.  Then $\iota(Y)df=\omega(X,Y)=0$ for all
$Y\in\lie{X}(\f)$, so $f$ is basic.  So a non-basic function cannot
have an associated Hamiltonian vector field.  Moreover, $f$ does not
uniquely determine $X$: if the vector field $X'$ also satisfies
$df=\iota(X')\omega$, then $\iota(X-X')\omega=0$, so $X-X'$ is tangent
to $\f$.  So the best we can expect is the following result, the proof
of which is a straightforward verification.

\begin{lemma}\label{lemma;hamiltonian}
Let $(M,\f,\omega)$ be a transversely symplectic foliated manifold.
\begin{enumerate}
\item\label{item;open}
For every basic function $f\in C^\infty(M)^\f$ there is a unique
transverse vector field $X_f\in\lie{X}(M/\f)$ satisfying
$\iota(X)\omega=df$.
\item\label{item;hamiltonian}
The formula $\{f,g\}=\omega(X_f,X_g)$ defines a Poisson bracket on the
algebra of basic functions $C^\infty(M)^\f$.  With respect to this
bracket the map $\ca{H}\colon C^\infty(M)^\f\to\lie{X}(M/\f)$ defined
by $\ca{H}(f)=X_f$ is a morphism of Poisson algebras, whose kernel
consists of the locally constant functions.
\end{enumerate}
\end{lemma}



We call $X_f$ the \emph{transverse Hamiltonian vector field}
associated with $f$ and the Lie algebra homomorphism $\ca{H}$ the
\emph{transverse Hamiltonian correspondence}.

Mimicking the symplectic case, we say that a transverse action
$a\colon\g\to\lie{X}(M/\f)$ of a Lie algebra $\g$ is
\emph{Hamiltonian} if $a$ lifts to a Lie algebra homomorphism
\[
\begin{tikzcd}[row sep=large]
&C^\infty(M)^\f\ar[d,"\ca{H}"]\\
\g\ar[ur,dotted,"\mu^\vee"]\ar[r,"a"]&\lie{X}(M/\f).
\end{tikzcd}
\]
We call $\mu^\vee$ the \emph{dual moment map} and the map $\mu\colon
M\to\g^*$ defined by $\pair{\mu(x)}{\xi}=(\mu^\vee(\xi))(x)$ for $x\in
M$ and $\xi\in\g$ the \emph{moment map} of the Hamiltonian action.  We
also write $\mu^\xi$ for the basic function $\mu^\vee(\xi)$ and call
this the \emph{$\xi$-component} of the moment map.  The moment map
$\mu$ completely determines the action $a=\ca{H}\circ\mu^\vee$.  A
\emph{transverse Hamiltonian $\g$-manifold} is a tuple
$(M,\f,\omega,\mu)$ consisting of a foliated manifold $(M,\f)$
equipped with a transverse symplectic form $\omega$ and a transverse
Hamiltonian $\g$-action with moment map $\mu$.

\begin{remark}[group versus algebra]\label{remark;group-algebra}
A special class of transverse Hamiltonian $\g$-manifolds arises from
the presymplectic Hamiltonian Lie group actions
of~\cite{SjamaarLinConvexity}.  Let $\hat{G}$ be a Lie group and
let $(M,\omega)$ be a connected presymplectic Hamiltonian
$\hat{G}$-manifold with moment map $\hat{\mu}\colon M\to\hat{\g}^*$,
where $\hat{\g}=\tn{Lie}(\hat{G})$.  Define $\n$ to be the set of all
$\xi\in\hat{\g}$ with the property that the component $\hat{\mu}^\xi$
of the moment map is constant on $M$.  Then $\n$ is an ideal of
$\hat{\g}$ and the moment map image $\hat{\mu}(M)$ is contained in an
affine subspace of $\hat{\g}^*$ parallel to the subspace $\tn{ann}(\n)$
(\cite[Proposition~2.9.1]{SjamaarLinConvexity}).  Suppose that
$\hat{\mu}$ actually takes values in $\tn{ann}(\n)$.  Let $\g=\hat{\g}/\n$ be the quotient Lie algebra and identify $\tn{ann}(\n)$ with the dual of $\g$; then we can view $\hat{\mu}$ as a map $\mu\colon M\to\g^*$.  The
$\hat{\g}$-action $\hat{\g}\to\lie{X}(M)$ induced by the
$\hat{G}$-action descends to a transverse $\g$-action
$\g\to\lie{X}(M/\f)$, where $\f$ is the null foliation of $\omega$,
and the tuple $(M,\f,\omega,\mu)$ is a transverse Hamiltonian
$\g$-manifold.
\end{remark}
We denote by $\g\ltimes\f=\g_M\times_{N\f}TM$ the transverse action
Lie algebroid of the transverse action $a\colon\g\to\lie{X}(M/\f)$
(Definition~\ref{d:tact}) and by $(\g\ltimes\f)_x=\g\times_{N_x\f}T_xM$
the fibre of $\g\ltimes\f$ at $x\in M$.  The anchor map of the
transverse action Lie algebroid is the bundle map
\[t\colon\g\ltimes\f\longto TM\]
defined by $t(\xi,v)=v$.  The \emph{stabilizer} of $x\in M$ is the
kernel of the anchor,
\begin{equation}\label{equation;stab}
\tn{stab}(\g\ltimes\f,x)=\ker(t_x)\subseteq(\g\ltimes\f)_x,
\end{equation}
which, as is the case for every Lie algebroid, inherits a Lie algebra
structure from $\g\ltimes\f$.  In fact the stabilizer is isomorphic
to the subalgebra of $\g$ given by
\begin{equation}\label{equation;stab-lie}
\tn{stab}(\g\ltimes\f,x)=\{\,\xi\in\g\mid a(\xi)(x)=0\,\}.
\end{equation}
We denote by $\g\ltimes\f\cdot x$ the orbit of $x$ under the Lie
algebroid $\g\ltimes\f$, i.e.\ the leaf through $x$ of the singular
foliation $t(\g\ltimes\f)\subseteq TM$.  This orbit is an immersed
submanifold of $M$, and we have
\begin{equation}\label{equation;orbit}
T_x(\g\ltimes\f\cdot x)/T_x\f=\{\,a(\xi)(x)\mid\xi\in\g\,\}.
\end{equation}
If $X$ is a subset of $M$, we write
\[\g\ltimes\f\cdot X=\bigcup_{x\in X}\g\ltimes\f\cdot x\]
for the $\g\ltimes\f$-orbit (or saturation) of $X$.

As in the symplectic case, the next lemma is a direct consequence of
the definition of a moment map.  Here we denote by $\tn{ann}(F)\subseteq
E^*$ the annihilator of a linear subspace $F$ of a vector space $E$.
We denote by $U^\omega\subseteq V$ the skew orthogonal of a linear
subspace $U$ of a vector space $V$ with respect to a $2$-form $\omega$
on $V$.

\begin{lemma}\label{lemma;moment}
Let $(M,\f,\omega,\mu)$ be a transverse Hamiltonian $\g$-manifold.
For all $x\in M$ we have
\begin{enumerate}
\item\label{item;moment-kernel}
$\ker(T_x\mu)=T_x(\g\ltimes\f\cdot x)^\omega$;
\item\label{item;moment-image}
$\tn{im}(T_x\mu)=\tn{ann}(\tn{stab}(\g\ltimes\f,x))$.
\end{enumerate}
\end{lemma}

Since the components of the moment map $\mu$ are $\f$-basic, the image
$\mu(\f\cdot x)$ of a leaf $\f\cdot x$ is a single point in $\g^*$, so
we can regard $\mu$ as a foliate map from $(M,\f)$ to $\g^*$ equipped
with the trivial zero-dimensional foliation.  Viewed as such, the map
$\mu$ is equivariant.

\begin{lemma}\label{lemma;equivariant}
Let $(M,\f,\omega,\mu)$ be a transverse Hamiltonian $\g$-manifold.
Then $\mu\colon M\to\g^*$ is equivariant with respect to the
$\g$-action on $M$ and the coadjoint action of $\g$ on $\g^*$.
\end{lemma}

\begin{proof}
Let $\xi$, $\eta\in\g$.  The dual moment map $\mu^\vee$ is a Lie
algebra homomorphism, so $\mu^{[\xi,\eta]}=\{\mu^\xi,\mu^\eta\}$, and
hence
\begin{equation}\label{equation;lie-poisson}
\iota(a(\xi))\d\mu^\eta=\iota(a(\xi))\iota(a(\eta))\omega=
\{\mu^\eta,\mu^\xi\}=\mu^{[\eta,\xi]}=-\mu^{[\xi,\eta]}=
\pair{\ad^*(\xi)\mu}{\eta}.
\end{equation}
This shows that $\d\mu(a(\xi))=\ad^*(\xi)\mu$, i.e.\ $\mu$ is
equivariant.
\end{proof}

Recall that the symplectic leaves of the linear Poisson space
$\g^*$ are the orbits of the coadjoint action.  In particular, each
coadjoint orbit $\ca{O}$ of $\g^*$ is equipped with a canonical
symplectic form $\sigma_{\ca{O}}$, the \emph{Kirillov-Kostant-Souriau
  symplectic form}, which at the point $\alpha\in\ca{O}$ is given by
\[\sigma_{\ca{O},\alpha}(\xi,\eta)=\alpha([\xi,\eta])\]
for $\xi$, $\eta\in\g$.  We denote by
$\g_\alpha=\tn{stab}(\g,\alpha)$ the coadjoint stabilizer of an
element $\alpha\in\g^*$.

The next result is a preliminary to the reduction theorem.  Though
awkward to state, it is nothing more than the Lie algebra analogue of
a very simple fact for Hamiltonian group actions: namely that in a
Hamiltonian $G$-manifold the moment map image of an orbit $L=G\cdot x$
is the coadjoint orbit $\ca{O}$ through $\mu(x)$, that the restriction
of the moment map $\mu_L\colon L\to\ca{O}$ is a fibre bundle whose
fibres are easy to describe, and that $\mu_L^*\sigma_{\ca{O}}$ is
equal to the restriction of $\omega$ to $L$.  Beware that for
transverse Hamiltonian Lie algebra actions the map $\mu_L$ is not
necessarily surjective.

\begin{lemma}\label{lemma;isotropic}
Let $(M,\f,\omega,\mu)$ be a transverse Hamiltonian $\g$-manifold.
Let $L$ be an orbit of the transverse action Lie algebroid
$\g\ltimes\f$.  There is a unique coadjoint orbit $\ca{O}$ of $\g^*$
satisfying $\ca{O}\supseteq\mu(L)$.  Let $\mu_L$ be the restriction of
the moment map to the orbit $L$.  The map
\[\mu_L\colon L\longto\ca{O}\]
is a submersion and $\mu_L^*\sigma_{\ca{O}}=\omega|_L$.  The connected
components of the fibres of $\mu_L$ are of the form
$\g_\alpha\ltimes\f\cdot x$, where $x$ ranges over all points in
$L$ and $\alpha=\mu(x)\in\ca{O}$.  In particular the submanifolds
$\g_\alpha\ltimes\f\cdot x$ of $M$ are isotropic for all $x$.
\end{lemma}

\begin{proof}
The existence and uniqueness of $\ca{O}$ and the fact that $\mu_L$ is
a submersion follow from the equivariance of $\mu$
(Lemma~\ref{lemma;equivariant}).  Let $x\in L$ and $\alpha=\mu(x)$.
Then $\alpha\in\ca{O}$ (again by equivariance), and the transverse
$\g$-action on $M$ restricts to a transverse action of the Lie
subalgebra $\g_\alpha$.  Hence we have a transverse action Lie
algebroid $\g_\alpha\ltimes\f$, whose orbit through $x$ is
$\g_\alpha\ltimes\f\cdot x\subseteq L$.
From~\eqref{equation;orbit} we get a short exact sequence of vector
spaces
\begin{equation}\label{equation;exact}
\begin{tikzcd}
T_x(\g_\alpha\ltimes\f\cdot x)\ar[r,hook]&T_x(\g\ltimes\f\cdot
x)=T_xL\ar[r,two heads,"T_x\mu_L"]&T_\alpha\ca{O},
\end{tikzcd}
\end{equation}
which tells us that the connected submanifold
$\g_\alpha\ltimes\f\cdot x$ is an open subset of the fibre
$\mu_L^{-1}(\alpha)$.  Since the orbits $\g_\alpha\ltimes\f\cdot x$
form a partition of the orbit $L$, we conclude that each
$\g_\alpha\ltimes\f\cdot x$ is a connected component of a fibre of
$\mu_L$.  Let $\xi$, $\eta\in\g$.  It follows
from~\eqref{equation;lie-poisson} that
\begin{align*}
(\mu_L^*\sigma_{\ca{O}})_x\bigl(a(\xi)(x),a(\eta)(x)\bigr)&=
  \sigma_{\ca{O},\alpha}(\xi,\eta)=\alpha([\xi,\eta])=\mu^{[\xi,\eta]}(x)\\
&=\omega_x\bigl(a(\xi)(x),a(\eta)(x)\bigr),
\end{align*}
which proves $\mu_L^*\sigma_{\ca{O}}=\omega|_L$.  In particular
$\omega$ vanishes along the fibres of $\mu_L$, so each orbit
$\g_\alpha\ltimes\f\cdot x$ is isotropic.
\end{proof}

\subsection{Reduction at a value}

Let $\g$ be a Lie algebra and $(M,\f,\omega,\mu)$ a transverse
Hamiltonian $\g$-manifold.  Transverse symplectic reduction is based
on the following simple observation.

\begin{proposition}\label{proposition;reduction-value}
Suppose that $\alpha\in\g^*$ is a regular value of $\mu\colon
M\to\g^*$.  Let $M_\alpha=\mu^{-1}(\alpha)$.
\begin{enumerate}
\item\label{item;zero-foliation}
For every $x\in M_\alpha$ the leaf $\f\cdot x$ is contained in
$M_\alpha$, so $\f$ restricts to a foliation $\f_\alpha$ of
$M_\alpha$.
\item\label{item;zero-anchor}
Let $\g_\alpha=\tn{stab}(\g,\alpha)$ be the coadjoint stabilizer of
$\alpha$.  The transverse $\g$-action $\g\to\lie{X}(M/\f)$ restricts
to a transverse $\g_\alpha$-action
$\g_\alpha\to\lie{X}(M_\alpha/\f_\alpha)$ on $(M_\alpha,\f_\alpha)$.
The anchor $t_\alpha\colon\g_\alpha\ltimes\f_\alpha\to TM_\alpha$ of
the transverse action Lie algebroid is injective and therefore
determines a foliation of $M_\alpha$, also denoted by
$\g_\alpha\ltimes\f_\alpha$.
\item\label{item;zero-symplectic}
The null foliation of the closed $2$-form
$\omega_\alpha=\omega|_{M_\alpha}$ is equal to
$\g_\alpha\ltimes\f_\alpha$.
\end{enumerate}
\end{proposition}

\begin{proof}
\eqref{item;zero-foliation}~This follows from the fact that the
components of the moment map are basic.

\eqref{item;zero-anchor}~The first assertion follows from the
equivariance of $\mu$ (Lemma~\ref{lemma;equivariant}).  Since $\alpha$
is a regular value of $\mu$, we have $\tn{stab}(\g\ltimes\f,x)=0$ for
all $x\in M_\alpha$ by Lemma~\ref{lemma;moment}\eqref{item;moment-kernel}.
It now follows from~\eqref{equation;stab}
and~\eqref{equation;stab-lie} that the anchor $t_\alpha$ is injective.

\eqref{item;zero-symplectic}~Let $x\in M_\alpha$ and
$L=\g\ltimes\f\cdot x$.  Then
$T_xM_\alpha=\ker(T_x\mu)=(T_x(\g\ltimes\f\cdot
x))^\omega=(T_xL)^\omega$ by
Lemma~\ref{lemma;moment}\eqref{item;moment-kernel}.  Therefore
\begin{align*}
\ker(\omega_{\alpha,x})&=T_xM_\alpha\cap(T_xM_\alpha)^\omega
=(T_xL)^\omega\cap\bigl((T_xL)^\omega\bigr)^\omega\\
&=(T_xL)^\omega\cap T_xL
=\ker((\omega|_L)_x),
\end{align*}
the kernel of the presymplectic form $\omega|_L$ at $x$.  It follows
from the identity $\mu_L^*\sigma_{\ca{O}}=\omega|_L$ of
Lemma~\ref{lemma;isotropic} and from the exact
sequence~\eqref{equation;exact} that
$\ker((\omega|_L)_x)=T_x(\g_\alpha\ltimes\f\cdot x)$.  Hence
$\ker(\omega_{\alpha,x})= T_x(\g_\alpha\ltimes\f\cdot
x)=T_x(\g_\alpha\ltimes\f_\alpha\cdot x)$.
\end{proof}

Ordinary symplectic reduction involves two steps: take a fibre
$\mu^{-1}(\alpha)$ of the moment map, and then take its quotient by
the action of the coadjoint stabilizer $G_\alpha$.  In our foliated
setting we will modify the second step: instead of dividing the fibre,
we enlarge the foliation.

\begin{definition}\label{definition;reduction-value}
Let $(M,\f,\omega,\mu)$ a transverse Hamiltonian $\g$-manifold.
Assume that $\alpha$ is a regular value of $\mu$.  The \emph{reduction
  at level $\alpha$} of $(M,\f,\omega,\mu)$ is the transversely
symplectic foliated manifold
$(M_\alpha,\g_\alpha\ltimes\f_\alpha,\omega_\alpha)$ of
Proposition~\ref{proposition;reduction-value}.
\end{definition}

\subsection{Reduction at a coadjoint orbit}

An alternative to reduction at a value is reduction at a coadjoint
orbit.  Let $\g$ be a Lie algebra and $(M,\f,\omega,\mu)$ a
transverse Hamiltonian $\g$-manifold.

\begin{proposition}\label{proposition;reduction-orbit}
Let $\ca{O}$ be a coadjoint orbit of $\g^*$.  Suppose that
$\mu\colon M\to\g^*$ is transverse to $\ca{O}$.  Let
$M_{\ca{O}}=\mu^{-1}(\ca{O})$ and
$\mu_{\ca{O}}=\mu|_{M_{\ca{O}}}\colon M_{\ca{O}}\to\ca{O}$.
\begin{enumerate}
\item\label{item;orbit-foliation}
For every $x\in M_{\ca{O}}$ the leaf $\f\cdot x$ is contained in
$M_{\ca{O}}$, so $\f$ restricts to a foliation $\f_{\ca{O}}$ of
$M_{\ca{O}}$.
\item\label{item;orbit-anchor}
The transverse $\g$-action $\g\to\lie{X}(M/\f)$ restricts to a
transverse $\g$-action $\g\to\lie{X}(M_{\ca{O}}/\f_{\ca{O}})$ on
$(M_{\ca{O}},\f_{\ca{O}})$.  The anchor
$t_{\ca{O}}\colon\g\ltimes\f_{\ca{O}}\to TM_{\ca{O}}$ of the
transverse action Lie algebroid is injective and therefore determines
a foliation of $M_{\ca{O}}$, also denoted by $\g\ltimes\f_{\ca{O}}$.
\item\label{item;orbit-symplectic}
The null foliation of the closed $2$-form
$\omega_{\ca{O}}=\omega|_{M_{\ca{O}}}-\mu_{\ca{O}}^*\sigma_{\ca{O}}$
is equal to $\g\ltimes\f_{\ca{O}}$.
\end{enumerate}
\end{proposition}

\begin{proof}
\eqref{item;orbit-foliation}--\eqref{item;orbit-anchor}~are proved in
the same way as Proposition~\ref{proposition;reduction-value}%
\eqref{item;zero-foliation}--\eqref{item;zero-anchor}.

\eqref{item;orbit-symplectic}~Let $x\in M_{\ca{O}}$ and
$\alpha=\mu(x)\in\ca{O}$.  The tangent space to $\ca{O}$ at $\alpha$
is $\tn{ann}(\g_\alpha)$, the annihilator of the stabilizer subalgebra
$\g_\alpha=\tn{stab}(\g,\alpha)$.  Therefore
\[
T_xM_{\ca{O}}=(T_x\mu)^{-1}(T_\alpha\ca{O})=\ker({\pr_\alpha}\circ
T_x\mu),
\]
where $\pr_\alpha$ denotes the quotient map $\g^*\to\g_\alpha^*$.  Put
$L=\g\ltimes\f\cdot x$ and $L_\alpha=\g_\alpha\ltimes\f\cdot x$.
Since ${\pr_\alpha}\circ\mu$ is a moment map for the transverse
$\g_\alpha$-action on $M$, we obtain from
Lemma~\ref{lemma;moment}\eqref{item;moment-kernel} that
\begin{equation}\label{equation;inverse-orbit}
T_xM_{\ca{O}}=(T_xL_\alpha)^\omega.
\end{equation}
Therefore $\omega$ descends to a symplectic form on the vector space
\[V=T_xM_{\ca{O}}/T_xL_\alpha=(T_xL_\alpha)^\omega/T_xL_\alpha,\]
which we will also call $\omega$.  The exact
sequence~\eqref{equation;exact} shows that $V$ has a symplectic
subspace isomorphic to $T_xL/T_xL_\alpha\cong T_\alpha\ca{O}$.  Let
$W=(T_\alpha\ca{O})^\omega$ be the orthogonal complement of
$T_\alpha\ca{O}$ in $V$; then $W$ is a symplectic subspace of $V$ with
symplectic form $\omega_W=\omega|_W$, and $V=T_\alpha\ca{O}\oplus W$
is a symplectic direct sum.  But $V$ has a second $2$-form, namely the
form obtained by subtracting from $\omega$ the form on the symplectic
subspace $T_\alpha\ca{O}$, in other words the form induced by
$\omega_{\ca{O}}=\omega|_{M_{\ca{O}}}-\mu^*\sigma_{\ca{O}}$.  This
shows that the symplectic vector space $(W,\omega_W)$ is isomorphic to
the quotient of the presymplectic vector space $(V,\omega_{\ca{O}})$
by its null space $T_\alpha\ca{O}$.  To conclude we note that
\[
W\cong
V/T_\alpha\ca{O}=\frac{T_xM_{\ca{O}}/T_xL_\alpha}{T_xL/T_xL_\alpha}=
T_xM_{\ca{O}}/T_xL,
\]
so $T_xL=T_x(\g\ltimes\f\cdot x)$ is the null space of the form
$\omega_{\ca{O}}$ on $T_xM_{\ca{O}}$.
\end{proof}

\begin{definition}\label{definition;reduction-orbit}
Let $(M,\f,\omega,\mu)$ a transverse Hamiltonian $\g$-manifold.  Let
$\ca{O}$ be a coadjoint orbit of $\g^*$ and assume that $\mu$ is
transverse to $\ca{O}$.  The \emph{reduction at the orbit $\ca{O}$} of
$(M,\f,\omega,\mu)$ is the transversely symplectic foliated manifold
$(M_{\ca{O}},\g\ltimes\f_{\ca{O}},\omega_{\ca{O}})$ of
Proposition~\ref{proposition;reduction-orbit}.
\end{definition}

How does the reduction at a point
$(M_\alpha,\g_\alpha\ltimes\f_\alpha,\omega_\alpha)$ of
Definition~\ref{definition;reduction-value} compare with the reduction
at an orbit $(M_{\ca{O}},\g\ltimes\f_{\ca{O}},\omega_{\ca{O}})$ of
Definition~\ref{definition;reduction-orbit}?
The best we can hope for is the following result, which says that the
inclusion $M_\alpha\inj M_{\ca{O}}$ induces an \'etale morphism of
$0$-symplectic stacks $M_\alpha/\g_\alpha\ltimes\f_\alpha\to
M_{\ca{O}}/\g\ltimes\f_{\ca{O}}$.  The coarse quotient of this
morphism (i.e.\ the map on leaf spaces) is in general neither
injective nor surjective.

\begin{proposition}\label{proposition;reduction-compare}
Let $(M,\f,\omega,\mu)$ a transverse Hamiltonian $\g$-manifold.  Let
$\ca{O}$ be a coadjoint orbit of $\g^*$ and
$M_{\ca{O}}=\mu^{-1}(\ca{O})$.  Let $\alpha\in\ca{O}$ and
$M_\alpha=\mu^{-1}(\alpha)$.  Suppose that $\mu\colon M\to\g^*$ is
transverse to $\ca{O}$.  Then $\alpha$ is a regular value of $\mu$,
the inclusion $j_\alpha\colon M_\alpha\to M_{\ca{O}}$ is transverse to
the foliation $\g\ltimes\f_{\ca{O}}$, the pullback foliation is
$j_\alpha^*(\g\ltimes\f_{\ca{O}})=\g_\alpha\ltimes\f_\alpha$, and
$j_\alpha^*\omega_{\ca{O}}=\omega_\alpha$.
\end{proposition}

\begin{proof}
Since $\mu$ is equivariant (Lemma~\ref{lemma;equivariant}), the
transversality of $\mu$ to $\ca{O}$ implies $\alpha\in\ca{O}$ is a
regular value.  That $\g\ltimes\f_{\ca{O}}$ is transverse to
$j_\alpha$ follows also from the transversality of $\mu$, and that
$j_\alpha^*(\g\ltimes\f_{\ca{O}})$ is equal to
$\g_\alpha\ltimes\f_\alpha$ follows from
Lemma~\ref{lemma;isotropic}.  The equality
$j_\alpha^*\omega_{\ca{O}}=\omega_\alpha$ follows immediately from
Propositions~\ref{proposition;reduction-value}
and~\ref{proposition;reduction-orbit}.
\end{proof}

Under some additional assumptions the discrepancy between the two
kinds of reduction disappears.

\begin{proposition}\label{proposition;reduction-compare-molino}
With the conventions and hypotheses of
Proposition~\ref{proposition;reduction-compare}, suppose additionally
that the foliation $\f$ is Molino, the leaf closure space $M/\ol{\f}$ is
compact, the transverse action of the Lie algebra $\g$ is isometric,
and $\g$ admits a positive definite invariant inner product.  Then
$j_\alpha\colon(M_\alpha,\g_\alpha\ltimes\f_\alpha,\omega_\alpha)\to
(M_{\ca{O}},\g\ltimes\f_{\ca{O}},\omega_{\ca{O}})$ is a weak
equivalence of transversely symplectic foliations.
\end{proposition}

\begin{proof}
We use Molino's structure theory summarized in Section~\ref{s:MolinoOverview}.  The
transverse orthonormal frame bundle $\pi_M\colon P\to M$ is a foliated
principal bundle with foliation $\f_P$.  The leaf closures of $\f_P$
form a strictly simple foliation $\ol{\f}_P$ of $P$.  Let $W=P/\ol{\f}_P$
be the Molino manifold and $\pi_W\colon P\to W$ the projection.  The
map $\mu\circ\pi_M$ is $\f_P$-invariant and therefore descends to a
smooth map $\mu_W\colon W\to\g^*$.  Let
\[
W_\alpha=\pi_W(\pi_M^{-1}(M_\alpha))=\mu_W^{-1}(\alpha),\qquad
W_{\ca{O}}=\pi_W(\pi_M^{-1}(M_{\ca{O}}))=\mu_W^{-1}(\ca{O}).
\]
Since the transverse $\g$-action is isometric, it carries over to a
$\g$-action on $W$.  It follows from Lemma~\ref{lemma;equivariant}
that $\mu_W$ is $\g$-equivariant.  Because $W$ is compact, the
$\g$-action on $W$ integrates to a $G$-action for some connected Lie
group $G$ with Lie algebra $\g$, and the map $\mu_W$ is
$G$-equivariant.  It follows that $W_{\ca{O}}=G\cdot W_\alpha$ and
$G\cdot w\cap W_\alpha=G_\alpha\cdot w$ for all $w\in W_\alpha$.
Since $\g$ is of compact type, the coadjoint stabilizer $G_\alpha$ is
connected.  Hence we obtain
\[
W_{\ca{O}}=\g\cdot W_\alpha,\qquad\g\cdot w\cap
W_\alpha=\g_\alpha\cdot w
\]
for all $w\in W_\alpha$, and therefore
\[
M_{\ca{O}}=(\g\times\lie{c})\ltimes\f\cdot
M_\alpha,\qquad(\g\times\lie{c})\ltimes\f\cdot x\cap
M_\alpha=(\g_\alpha\times\lie{c})\ltimes\f_\alpha\cdot x
\]
for all $x\in M_\alpha$, where $\lie{c}$ is the Molino centralizer
sheaf of $M$.  Since $M_\alpha$ is closed and $\f$-invariant, the
first of these equalities says that $M_{\ca{O}}=\g\ltimes\f\cdot
M_\alpha$, in other words the groupoid morphism
$J_\alpha\colon\tn{Hol}(M_\alpha,\g_\alpha\ltimes\f_\alpha)\to
\tn{Hol}(M_{\ca{O}},\g\ltimes\f_{\ca{O}})$ induced by $j_\alpha$ is
essentially surjective.  The second equality implies
\[
\g_\alpha\ltimes\f_\alpha\cdot x\subseteq\g\ltimes\f\cdot x\cap
M_\alpha\subseteq\overline{\g_\alpha\ltimes\f_\alpha\cdot x}
\]
for all $x\in M_\alpha$.  Since $\g_\alpha\ltimes\f_\alpha\cdot x$ is
a connected component of $\g\ltimes\f\cdot x\cap M_\alpha$
(Lemma~\ref{lemma;isotropic}), we obtain
$\g_\alpha\ltimes\f_\alpha\cdot x=\g\ltimes\f\cdot x\cap M_\alpha$,
which means that $J_\alpha$ is fully faithful.
\end{proof}

\subsection{Additional properties}

We explain three other features of transverse symplectic reduction
that are of use in Section~\ref{s:qr0}: the shifting trick, the reduction
of weak equivalences, and the reduction of complete Riemannian foliations.

The shifting trick for Hamiltonian group actions says that reduction
of $M$ at an orbit $\ca{O}$ is symplectically isomorphic to reduction
of the product $M\times\ca{O}^-$ at the zero level.  Here $\ca{O}^-$
denotes the opposite of $\ca{O}$, i.e.\ $\ca{O}$ equipped with the
symplectic form $-\sigma_{\ca{O}}$.  The shifting trick is proved by
identifying $M$ with the graph of the moment map in $M\times\g^*$.
The foliated version is almost the same.

\begin{proposition}\label{proposition;reduction-shift}
Let $\g$ be a Lie algebra and $(M,\f,\omega,\mu)$ a transverse
Hamiltonian $\g$-manifold.  Let $\ca{O}$ be a coadjoint orbit of
$\g^*$, let $i_{\ca{O}}\colon\ca{O}\to\g^*$ be the inclusion,
and let $M_{\ca{O}}=\mu^{-1}(\ca{O})$.
\begin{enumerate}
\item\label{item;product}
The product $M\times\ca{O}^-$ equipped with the foliation $\pr_1^*\f$,
the presymplectic form $\pr_1^*\omega-\pr_2^*\sigma_{\ca{O}}$, the
diagonal $\g$-action, and the moment map
$\pr_1^*\mu-\pr_2^*i_{\ca{O}}$ is a transverse Hamiltonian
$\g$-manifold.  The moment map $\mu\colon M\to\g^*$ is
transverse to $\ca{O}$ if and only if $0$ is a regular value of
$\pr_1^*\mu-\pr_2^*i_{\ca{O}}$.
\item\label{item;shift}
Suppose that $\mu\colon M\to\g^*$ is transverse to $\ca{O}$.  The
map $M\to M\times\g^*$ defined by $x\mapsto(x,\mu(x))$ induces an
isomorphism
\[
(M_{\ca{O}},\g\ltimes\f_{\ca{O}},\omega_{\ca{O}})\cong
\bigl((M\times\ca{O}^-)_0,\g\ltimes(\pr_1^*\f)_0,
(\pr_1^*\omega-\pr_2^*\sigma_{\ca{O}})_0\bigr)
\]
of transversely symplectic foliated manifolds.
\end{enumerate}
\end{proposition}

The proof is a straightforward verification based on
Propositions~\ref{proposition;reduction-value}
and~\ref{proposition;reduction-orbit}.

The next result says that if two transverse Hamiltonian $\g$-manifolds
are weakly equivalent, then their reduced spaces are weakly
equivalent.

\begin{proposition}\label{proposition;reduction-equivalent}
Let $\g$ be a Lie algebra and $(M_1,\f_1,\omega_1,\mu_1)$ a transverse
Hamiltonian $\g$-manifold.  Let $(M_2,\f_2)$ be a foliated manifold
and $(B,\pi_1,\pi_2)$ a weak equivalence between $(M_1,\f_1)$ and
$(M_2,\f_2)$.  There exists a unique structure of transverse
Hamiltonian $\g$-manifold $(M_2,\f_2,\omega_2,\mu_2)$ on $M_2$ such
that $B$ is equivariant, $\pi_1^*\omega_1=\pi_2^*\omega_2$, and
$\pi_1^*\mu_1=\pi_2^*\mu_2$.  For every coadjoint orbit $\ca{O}$ of
$\g$ transverse to $\mu_1$ (or to $\mu_2$), $B$ induces a symplectic
weak equivalence between the reduced spaces
$((M_1)_{\ca{O}},\g\ltimes(\f_1)_{\ca{O}},(\omega_1)_{\ca{O}})$ and
$((M_2)_{\ca{O}},\g\ltimes(\f_2)_{\ca{O}},(\omega_2)_{\ca{O}})$.
\end{proposition}

\begin{proof}
The first assertion is proved in the same way as
Proposition~\ref{p:equivMorita}.  To prove the second assertion, note that the
submanifold
$B_{\ca{O}}=(\mu_1\circ\pi_1)^{-1}(\ca{O})=(\mu_2\circ\pi_2)^{-1}(\ca{O})$
of $B$ defines a symplectic weak equivalence between
$((M_1)_{\ca{O}},\g\ltimes(\f_1)_{\ca{O}},(\omega_1)_{\ca{O}})$ and
$((M_2)_{\ca{O}},\g\ltimes(\f_2)_{\ca{O}},(\omega_2)_{\ca{O}})$.
\end{proof}

To conclude we show that the complete Riemannian foliation property is preserved under transverse symplectic reduction.

\begin{proposition}\label{proposition;reduction-molino}
Let $\g$ be a Lie algebra and $(M,\f,\omega,\mu)$ a transverse
Hamiltonian $\g$-manifold.  Let $\ca{O}$ be a coadjoint orbit of $\g$
transverse to $\mu$.  A Riemannian structure $g$ on $(M,\f)$ which is
invariant under $\g$ induces a Riemannian structure $g_{\ca{O}}$ on
the reduced foliation $(M_{\ca{O}},\g\ltimes\f_{\ca{O}})$.  If
$(M,\f,g)$ is a complete Riemannian foliation, then so is
$(M_{\ca{O}},\g\ltimes\f_{\ca{O}},g_{\ca{O}})$.
\end{proposition}

\begin{proof}
The normal bundle $N(\g\ltimes\f_{\ca{O}})$ of the foliation
$\g\ltimes\f_{\ca{O}}$ is a quotient of the normal bundle
$N\f_{\ca{O}}$ of the foliation $\f_{\ca{O}}=\f|_{M_{\ca{O}}}$.  Let
$g_{\ca{O}}$ be the fibre metric on $N(\g\ltimes\f_{\ca{O}})$ induced
by the metric $g$.  Then $g_{\ca{O}}$ is
$\g\ltimes\f_{\ca{O}}$-invariant because $g$ is
$\g\ltimes\f$-invariant. Thus $(M_{\ca{O}},\g\ltimes\f_{\ca{O}},g_{\ca{O}})$ is a Riemannian foliation. The restriction of a bundle-like metric for $(M,\f)$ to the closed (as a set) submanifold $M_{\ca{O}}$ is a bundle-like metric for $(M_{\ca{O}},\f_{\ca{O}})$. If the bundle-like metric on $M$ is complete, then so is its restriction to the closed (as a set) submanifold $M_{\ca{O}}$ (by the Hopf-Rinow theorem).
\end{proof}

\bibliographystyle{amsplain}
\providecommand{\bysame}{\leavevmode\hbox to3em{\hrulefill}\thinspace}
\providecommand{\MR}{\relax\ifhmode\unskip\space\fi MR }
\providecommand{\MRhref}[2]{%
  \href{http://www.ams.org/mathscinet-getitem?mr=#1}{#2}
}
\providecommand{\href}[2]{#2}

\end{document}